\documentclass[a4paper,10pt]{amsart}

\usepackage{amsmath}
\usepackage{amssymb}
\usepackage{amsthm}
\usepackage{verbatim}
\usepackage[all]{xy}
\usepackage{enumerate}
\usepackage{setspace}
\usepackage{color}

\usepackage[all]{xy}

\usepackage{color}

\newtheorem{thm}{Theorem}[section]
\newtheorem{prop}[thm]{Proposition}
\newtheorem{cor}[thm]{Corollary}
\newtheorem{lem}[thm]{Lemma}

\theoremstyle{definition}

\newtheorem{dfn}[thm]{Definition}

\newtheorem{example}[thm]{Example}
\newtheorem{rmk}[thm]{Remark}

\numberwithin{equation}{section}
\newcommand{\Pol}{{\rm Pol}}
\newcommand{\cHil}{\mathcal{H}}
\newcommand{\cK}{\mathcal{K}}

\newcommand{\cH}{\mathcal{H}}

\newcommand{\bG}{\mathbb{G}}

\newcommand{\id}{\textrm{id}}

\newcommand{\cB}{\mathcal{B}}

\newcommand{\bc} {\Bbb C}

\newcommand{\ot}{\otimes}

\newcommand{\cHPol}{\mathcal{H}_{{\rm Pol}}}

\newcommand{\bM}{\mathsf{M}}
\newcommand{\bA}{\mathsf{A}}
\newcommand{\bB}{\mathsf{B}}

\newcommand{\bQ}{\mathsf{Q}}

\newcommand{\boldv}{{\bf v}}
\newcommand{\boldw}{{\bf w}}

\newcommand{\cW}{\mathcal{W}}
\newcommand{\cWmin}{\mathcal{W}_{{\rm min}}}
\newcommand{\cWred}{\mathcal{W}_{{\rm red} }}

\newcommand{\Link}{{\rm Link}}
\newcommand{\Star}{{\rm Star}}

\newcommand{\cQ}{\mathcal{Q}}
\newcommand{\cE}{\mathcal{E}}
\newcommand{\cP}{\mathcal{P}}
\newcommand{\bN}{\mathsf{N}}

\newcommand{\Irr}{{\rm Irr}}
\newcommand{\cA}{\mathcal{A}}

\title[Graph products of operator algebras]{Graph products of operator algebras}

\begin{document}


\thispagestyle{empty}

\begin{center}

\textbf{\textsc{Graph products of operator algebras}}

\vspace{0.3cm}

\small

\textsc{M. Caspers$^a$, P. Fima$^b$}

\vspace{0.3cm}

{\it
\noindent  $^a$ Corresponding author. Address: Fachbereich Mathematik und Informatik der Universit\"at M\"unster,
Einsteinstrasse 62, 48149 M\"unster, Germany. \em E-mail: martijn.caspers@uni-muenster.de.\em     \\
\noindent $^b$ Address:  Univ Paris Diderot, Sorbonne Paris Cit\'e, IMJ-PRG, UMR 7586, F-75013, Paris, France.
  Sorbonne Universit\'es, UPMC Paris 06, UMR 7586, IMJ-PRG, F-75005, Paris, France.
  CNRS, UMR 7586, IMJ-PRG, F-75005, Paris, France.
\em E-mail:  pierre.fima@imj-prg.fr

\vspace{0.3cm}

\noindent MC is supported by the grant SFB 878 ``{\it Groups, geometry and actions}''.\\
PF  is supported by the ANR grants NEUMANN and OSQPI.

\vspace{0.3cm}

}

\end{center}

\small

\vspace{0.3cm}

\noindent {\sc Abstract:} Graph products for groups were defined by Green in her thesis \cite{Green} as a generalization of both Cartesian and free products.  In this paper we define the corresponding graph product for reduced and maximal C$^\ast$-algebras, von Neumann algebras and quantum groups. We prove  stability properties including permanence of II$_1$-factors,  the Haagerup property, exactness and, under suitable conditions, the property of Rapid Decay for quantum groups.

\normalsize

\vspace{0.3cm}

\section*{Introduction}

\noindent A graph product is a group theoretical construction starting from a simplicial graph with  a discrete group associated to each vertex. The graph product construction results in a new group and special cases depending on the graph are free products and Cartesian products. Important examples of graph products are right angled Coxeter groups and right angled Artin groups.

\vspace{0.2cm}

\noindent Graph products preserve many important group theoretical properties. This yields important new examples of groups having such properties and gives (alternative) proofs of such properties for existing groups. For instance the graph product preserves soficity \cite{CioHolRee2}, Haagerup property \cite{AntDre}, residual finiteness \cite{Green}, rapid decay \cite{CioHolRee}, linearity \cite{HsuWis} and many other properties, see e.g. \cite{HerMei}, \cite{AntMin}, \cite{Chi}.

 \vspace{0.2cm}

\noindent Whereas many of the stability properties above have important consequences for operator algebras, the actual operator algebras of graph products have  been unexplored so far. The current paper develops the theory of reduced and universal/maximal C$^\ast$-algebraic graph products as well as the graph product of von Neumann algebras and quantum groups. These objects generalize free products by adding commutation relations that are dictated by the graph.

\vspace{0.2cm}

\noindent Free products of operator algebras play a central role in von Neumann and C$^\ast$-algebra theory; in particular in the context of free probability and deformation and rigidity theory. Operator algebraic graph products -- which give in a suitable sense a notion of ``partial freeness'' being quantized in terms of the number of edges of the graph -- provide new examples that are closely related to these areas.  Such ideas appeared in fact already in an early stage of free probability theory, for example in \cite{SpeicherLetter} Speicher proves a mixed Fermion--Boson analogue of his central limit theorem (used in \cite{CasConnes} to show stability of the Connes embedding problem for graph products). Other important implicit occurrences of graph products can be found in the work by Bozejko and Speicher on Coxeter groups, see e.g. \cite{BozSpe}. Also in \cite{AccardiEtAl} other extensions of free probability using graphs were investigated by Accardi, Lenczewski and Salapata.  In this context we also mention the current parallel developments on bi-freeness \cite{Voic}.

\vspace{0.2cm}

\noindent We shall relate the basic properties of graph products of operator algebras/quantum groups to the ones of their vertices. This includes Tomita-Takesaki theory, commutants, GNS-representations, (co)representation theory,  et cetera. We also show that any graph product of von Neumann algebras decomposes inductively into amalgamated free products of the von Neumann algebras at its edges. For notation we refer to Section \ref{Sect=OA}.

\begin{thm}\label{Thm=AmalIntro}
Let $\Gamma$ be a simplicial graph with von Neumann algebras $\bM_v, v\in V\Gamma$ and graph product von Neumann algebra $\bM$. Fix $v \in V\Gamma$. Let $\bM_1$ be the graph product von Neumann algebra given by $\Star(v)$. Let $\bM_2$ be the graph product von Neumann algebra given by $\Gamma \backslash \{v \}$. Let $\bN$ be the graph product von Neumann algebra given by $\Link(v)$. Then $\bM \simeq \bM_1 \star_\bN \bM_2$.
\end{thm}

\noindent There is a corresponding result of Theorem \ref{Thm=AmalIntro} for C$^\ast$-algebras, see Section \ref{Sect=OA}. Theorem \ref{Thm=AmalIntro}  implies that any property of a von Neumann algebra that is being preserved by arbitrary amalgamated free products is automatically preserved by the graph product. However, there is a large number of properties which are not (or not known to be) preserved by amalgamated free products. For example, the Haagerup property is known not to be preserved by arbitrary amalgamated free products. But in fact we prove the following.

\begin{thm}\label{Thm=IntroA}
Let $\Gamma$ be a simplicial graph with von Neumann algebras $\bM_v, v \in V\Gamma$. Let $\bM$ be the graph product von Neumann algebra. Then,
\begin{enumerate}
\item\label{Item=IntroA1} Suppose that every $\bM_v$ is $\sigma$-finite.  $\bM$ has the Haagerup property if and only if for every $v \in V \Gamma$, $\bM_v$ has the Haagerup property.
\item $\bM$ is a II$_1$ factor if for every $v \in V\Gamma$, $\bM_v$ is a II$_1$ factor.
\end{enumerate}
\end{thm}

\noindent While proving stability of the Haagerup property we also included a canonical proof of extending completely positive maps to graph (and in particular free) products by considering their Stinespring dilations, see Proposition \ref{Prop=ucp}. In the case of quantum groups we find the following stability properties:

\begin{thm}\label{Thm=IntroB}
Let $\Gamma$ be a simplicial graph with compact quantum groups $\bG_v, v \in V\Gamma$. Let $\bG$ be its graph product and let $\widehat{\bG}_v, \widehat{\bG}$ be their duals. Then,
\begin{enumerate}
\item\label{Item=IntroB1} $\widehat{\bG}$ has the Haagerup property if and only if for every $v \in V\Gamma$, $\widehat{\bG}_v$ has the Haagerup property.
\item Let $\Gamma$ be finite. If for every $v \in V\Gamma$, $\widehat{\bG}_v$ is a classical group with the property of Rapid Decay (RD) or a quantum group with polynomial growth, then the graph product $\widehat{\bG}$ has (RD).
\item Let $\Gamma$ be finite without edges. Then $\bG = \star_{v \in V\Gamma} \bG_v$. If for every $v \in V\Gamma$, $\widehat{\bG}_v$ has (RD), then $\widehat{\bG}$ has (RD). I.e. (RD) is preserved by finite free products.
\end{enumerate}
\end{thm}

\noindent It must be emphasized that for compact quantum groups with tracial Haar state (i.e. of so-called Kac type)  Theorem \ref{Thm=IntroB} \eqref{Item=IntroB1} follows from Theorem \ref{Thm=IntroA} \eqref{Item=IntroA1} by \cite[Theorem 6.7]{DFSW}. However, it is unknown if the result of \cite[Theorem 6.7]{DFSW} extends beyond Kac type quantum groups. In fact \cite{CasLeeRic} shows that the behaviour of approximation properties outside the Kac case can be quite different.  In the group case our result gives an alternative proof of stability of the Haagerup property under graph products, which was first proved in \cite{AntDre}.

\vspace{0.2cm}

\noindent {\it Acknowledgements.} The authors thank Amaury Freslon for useful comments on Section \ref{Sect=RD}. The authors thank the anonymous referees for their comments.

\subsection*{Structure of this paper}

Section \ref{Sect=Preliminaries} introduces the basic notions for graph products. In section \ref{Sect=OA} we develop the theory of graph products of operator algebras: graph products of Hilbert spaces, von Neumann algebras and maximal and reduced graph products of C*-algebras, study their representation theory and develop the unscrewing technique as explained in Theorem \ref{Thm=AmalIntro}. We also prove some stability properties such as exactness for reduced graph product of C*-algebras and the Haagerup property for von Neumann algebras. In Section \ref{Sect=QG} we define graph products of quantum groups, study their representation theory and prove the stability of the Haagerup property. Section \ref{Sect=RD} proves stability of rapid decay for quantum groups under graph products.

\subsection*{General notation and preliminaries}
We denote $M_n$ for the $n\times n$ matrices over $\mathbb{C}$. We use bold face characters $\bA$ and $\bM$ for operator algebras and calligraphic characters $\cH$ and $\cK$ for Hilbert spaces. The scalar product on Hilbert spaces is supposed to be linear in the first variable. The symbol $\otimes$ denotes the tensor product of Hilbert spaces, the minimal tensor product of C$^\ast$-algebras or von Neumann algebras tensor product and it should be natural from the context which of these is meant. The symbol  $\underset{\text{max}}{\otimes}$ will denote the maximal tensor product of C*-algebras. A state on a C$^\ast$-algebra or more generally a completely positive map between C$^\ast$-algebras is called GNS-faithful if the representation given by the GNS-construction is faithful. Faithful states are GNS-faithful but the converse is false.

\section{Preliminaries}\label{Sect=Preliminaries}

\noindent Let $\Gamma$ be a simplicial graph. This means that $\Gamma$ is given by a vertex set $V\Gamma$ and edge set $E\Gamma \subseteq V\Gamma \times V\Gamma\setminus\{(v,v)\,:\,v\in V\Gamma\}$. We assume that the graph is non-oriented in the sense that $(v,w) \in E\Gamma$ if and only if $(w, v) \in E \Gamma$, for all $v,w\in V\Gamma$. For $v \in V\Gamma$ we write $\Link(v)$ for the set of all $w \in V\Gamma$ such that $(v, w) \in E \Gamma$.   We set $\Star(v) = \Link(v) \cup \{ v \}$.  Let $A \subseteq V\Gamma$ be a subset of vertices. The full subgraph of $\Gamma$ with vertex set $A$ is the graph having vertex set $A$ and  $w, w'\in A$ are connected by an edge if and only if $(w,w') \in E\Gamma$. We slightly abuse notation and also use $\Link(v)$ and $\Star(v)$ for the full subgraphs of $\Gamma$ with respective vertex sets $\Link(v)$ and $\Star(v)$. It shall always be clear from the context if $\Link(v)$ (or $\Star(v)$) is a graph or a vertex set.

\vspace{0.2cm}

\noindent {\bf Convention:} From this point we will say that $\Gamma_0 \subseteq \Gamma$ is a subgraph if $\Gamma_0$ is the {\it full} subgraph of $\Gamma$ with vertex set $V\Gamma_0$.

\begin{dfn}
A {\it clique} in the graph $\Gamma$ is a subgraph $\Gamma_0 \subseteq \Gamma$ such that for every $v, v' \in V\Gamma_0$ with $v \not = v'$  we have $(v, v') \in E\Gamma_0$ (so a complete subgraph of $\Gamma$). In particular every single vertex of $\Gamma$ forms a clique (with no edges). By convention the empty graph is a clique. We denote ${\rm Cliq}(s)$ for all cliques in $\Gamma$ with exactly $s$ vertices.
\end{dfn}

\begin{dfn}\label{Dfn=GroupGraphProd}
For each $v \in V \Gamma$ let $G_v$ be a discrete group. The graph product $G_\Gamma$ is defined as the group obtained from the free product of $G_v, v \in V \Gamma$ by adding the relations
\[
[s, t] = 1 \textrm{ for all } s \in G_v, t \in G_w \textrm{ and all }  v, w \in V\Gamma\textrm{ such that } (v,w) \in E\Gamma.
\]
\end{dfn}

\noindent A \textit{word} is a finite sequence $\boldv = (v_1, \ldots, v_n)$ of elements in $V \Gamma$. We shall commonly use bold face notation for words and write $v_i$ for the entries of $\boldv$.  The collection of words is denoted by $\cW$ and by convention the empty word is not included in $\cW$.  We say that two words $\boldv$ and $\boldw$ are \textit{equivalent} if they are equivalent modulo the equivalence relation generated by the two relations:

\begin{equation}\label{Eqn=Equivalences}
\begin{split}
{\rm I}\quad(v_1, \ldots, v_i, v_{i+1}, \ldots, v_n) \simeq (v_1, \ldots, v_i, v_{i+2}, \ldots, v_n) \qquad& {\rm if} \quad  v_i = v_{i+1},\\
{\rm II}\quad(v_1, \ldots, v_i, v_{i+1}, \ldots, v_n) \simeq (v_1, \ldots, v_{i+1}, v_{i}, \ldots, v_n) \qquad& {\rm if} \quad v_i \in \Link(v_{i+1}).
\end{split}
\end{equation}
Moreover, we say that two words $\boldv$ and $\boldw$ are \textit{type ${\rm II}$ equivalent} if they are equivalent modulo the sub-equivalence relation generated by the relation ${\rm II}$.
A word $\boldv \in \cW$ is \textit{reduced} if the following statement holds:
 \begin{equation}\label{Eqn=ReducedWord}
 \begin{split}
& \textrm{ Suppose that there are } k > l  \textrm{ such that } v_k = v_l, \\
&\textrm{ then we do not have that all } v_{k+1}, \ldots, v_{l-1} \in \Star(v_k).
\end{split}
 \end{equation}

\noindent We let $\cWred$ be the set of all reduced words. Observe that if $\boldv$ is reduced and type II equivalent to $\boldv'$ then also $\boldv'$ is reduced.
\begin{lem}\label{Lem=ReducedRep}
We have,
\begin{enumerate}
\item \label{Item=ReducedRepI} Every word $\boldv$ is equivalent to a reduced word $\boldw=(w_1,\dots,w_n)$.
 \item \label{Item=ReducedRepII} If $\boldv$ is also equivalent to a reduced word $\boldw'$, then the lengths of $\boldw$ and $\boldw'$ are equal.
 \item \label{Item=ReducedRepIII} Moreover, there exists a permutation $\sigma$ of $\{ 1, \ldots, n \}$ such that $\boldw' = (w_{\sigma(1)}, w_{\sigma(2)}, \ldots, w_{\sigma(n)})$.
\item \label{Item=ReducedRepIV} There is a unique such $\sigma$ if we impose the condition that if $k > l$ and $w_k = w_l$, then $\sigma(k) > \sigma(l)$.
\end{enumerate}
\end{lem}
\begin{proof}
\eqref{Item=ReducedRepI} Note that any word that can not be made shorter my means of permutations and cancellations \eqref{Eqn=Equivalences} is reduced. Hence, statement $(1)$ follows from an obvious induction.

\vspace{0.2cm}

\noindent\eqref{Item=ReducedRepII}  This is essentially the normal form theorem \cite[Theorem 3.9]{Green}. It can be derived as follows. For each $v \in V \Gamma$ let $G_v$ be the group $\mathbb{R}^+$ with multiplication. For $x \in \mathbb{R}^+$ we shall explicitly write $x_v$ to identify it as an element of $G_v$. Associate to the word $\boldw$ of length $n$ the group element $g_{\boldw} := 2_{w_1} 2_{w_2} \ldots 2_{w_n}$ in the graph product of the groups $G_v, v\in V \Gamma$, see Definition \ref{Dfn=GroupGraphProd}. Since $\boldw$ is reduced, it follows that $g_{\boldw}$ is reduced in the sense of \cite{Green}.  Assume that $\boldw'$ has length $m$. Since $\boldw$ is equivalent to $\boldw'$, there exists elements $x_{1}, \ldots, x_{m}$ with $x_i \in G_{w_i'}$ and $x_i > 1$ such that $g_{\boldw}$ is equivalent to the graph product element $g_{\boldw'} = x_{1}  \ldots  x_{m}$ (this can easily be seen by checking this on each step \eqref{Eqn=Equivalences} to obtain this equivalence, in particular $x_i$ is either a power or a root of 2). Since $\boldw'$ is reduced, it follows that $g_{\boldw'}$ is reduced. Hence, $g_{\boldw'}$ and $g_{\boldw}$ are reduced equivalent elements in the graph product of $G_v, v \in \Gamma$ and by the normal form Theorem \cite[Theorem 3.9]{Green}, this implies that $m=n$. In fact \cite[Theorem 3.9]{Green} implies also that $x_i = 2$.

\vspace{0.2cm}

\noindent \eqref{Item=ReducedRepIII}
 Let $m$ be the total number of times that a given $v$ appears in $\boldw$. We need to show that $v$ appears exactly $m$ times in $\boldw'$. Suppose that this is not the case. Since $\boldw$ and $\boldw'$ have the same word length we may assume, without loss of generality, that it appears less than $m$ times in $\boldw'$ since else, we may change $v$ to  another vertex for which this is true. But since $\boldw'$ is obtained from $\boldw$ through the equivalences \eqref{Eqn=Equivalences} this means that there exists some $l> k$ such that $w_l = w_k = v$ and $w_{k+1}, \ldots, w_{l-1} \in \Star(v)$ which contradicts the fact that $\boldw$ is reduced.

\vspace{0.2cm}

\noindent\eqref{Item=ReducedRepIV} Write $\boldw'=(w'_1,\dots,w'_n)$ and note that any permutation $\sigma$ of $\{1,\dots,n\}$ such that $\boldw'=(w_{\sigma(1)},\dots,w_{\sigma(n)})$ induces, for any vertex $w$ occurring in the word $\boldw$, a bijection from $K_{w}' := \{ i \mid w_i' = w\}$ to $K_w := \{ i \mid w_i = w\}$. There is a unique such bijection which is moreover increasing.
\end{proof}

\noindent Let $\cWmin$ be a complete set of representatives of the reduced words under the equivalence relation described above. We call an element of $\cWmin$ a \textit{minimal word}. It is then clear that every word $\boldv$ is equivalent to a unique minimal word $\boldw$. Note that $\cWmin$ excludes the empty word.

\section{Graph products of operator algebras}\label{Sect=OA}

\noindent In this section we construct graph products of operator algebras. In case the graph $\Gamma$ does not have edges, the graph product coincides with the free product for which we refer to \cite{Voi}. In addition, it is important to emphasize that our constructions are different from \cite{FimFre}: indeed graph products impose {\it commutation relations} on the resulting algebra which, in general, cannot be written in terms of the {\it amalgamations} imposed by the constructions in \cite{FimFre}.

\subsection{The graph product Hilbert space}\label{Section=GraphProductHilbert}

For all $v\in V\Gamma$ let $\cH_v$ be a Hilbert space with a norm one vector $\xi_v\in \cH_v$. Define $\cH_v^{\circ}=\cH_v\ominus\mathbb{C}\xi_v$ and let $\mathcal{P}_v$ be the orthogonal projection onto $\cH_v^{\circ}$. For $\boldv \in \cWred$ we let,
\[
\cHil_{\boldv} = \cHil_{v_1}^{\circ} \otimes \ldots \otimes \cHil_{v_n}^{\circ}.
\]
By Lemma \ref{Lem=ReducedRep} we see that if $\boldv \in \cWred$ is equivalent to $\boldw \in \cWred$ then there exists a uniquely determined unitary map,
\begin{equation}\label{Eqn=QMap}
\cQ_{\boldv, \boldw}: \cHil_{\boldv} \rightarrow \cHil_{\boldw}: \xi_{1} \otimes  \ldots \otimes \xi_{n} \mapsto
\xi_{{\sigma(1)}} \otimes  \ldots \otimes \xi_{{\sigma(n)}},
\end{equation}
where $\sigma$ is as in Lemma \ref{Lem=ReducedRep} \eqref{Item=ReducedRepIV}.
Since every $\boldv \in \cWred$ has a unique minimal form $\boldv'$ we may simply write $\cQ_{\boldv}$ for $\cQ_{\boldv, \boldv'}$. 

\vspace{0.2cm}
\noindent Define the \textit{graph product Hilbert space} $(\mathcal{H},\Omega)$ by:
$$\mathcal{H}=\bc\Omega\oplus\bigoplus_{\boldw \in \cW_{\text{min}}}\mathcal{H}_{\boldw}.$$
For $v\in V\Gamma$, let $\cW_v$ be the set of minimal reduced words $\boldw$ such that the concatenation $v\boldw$ is still reduced and write $\cW_v^c=\cW_{\text{min}}\setminus\cW_v$. Define
$$\mathcal{H}(v)=\bc\Omega\oplus\bigoplus_{\boldw\in\cW_v}\mathcal{H}_\boldw.$$
We define the isometry $U_v\,:\,\cH_v\ot\cH(v)\rightarrow\cH$ in the following way:
$$\begin{array}{llcl}
U_v\,:\,&\cH_v\ot\cH(v) &\longrightarrow& \cH\\
&\xi_v\ot\Omega &\overset{\simeq}{\longrightarrow}&\Omega\\
&\cH_v^{\circ}\ot\Omega&\overset{\simeq}{\longrightarrow}&\cH_v^{\circ}\\
&\xi_v\ot\cH_\boldw&\overset{\simeq}{\longrightarrow}&\cH_\boldw\\
&\cH_v^{\circ}\ot\cH_\boldw&\overset{\simeq}{\longrightarrow}&\cQ_{v\boldw}(\cH_v^{\circ}\ot\cH_\boldw)
\end{array}$$
Here the actions are understood naturally. Observe that, for any reduced word $\boldw$, the word $v\boldw$ is not reduced if and only if $\boldw$ is equivalent to a reduced word that starts with $v$. It follows that $U_v$ is surjective, hence unitary.
Define, for $v\in V\Gamma$, the faithful unital normal $*$-homomorphism $\lambda_v\,:\,\mathcal{B}(\cH_v)\rightarrow\mathcal{B}(\cH)$ by
$$\lambda_v(x)=U_v(x\ot 1)U_v^*\quad\text{for all}\quad x\in\cB(\cH_v).$$
Observe the $\lambda_v$ intertwines the vector states $\omega_{\xi_v}$ and $\omega_{\Omega}$.  Let $\xi_v \in \cH_v$. We use $\xi_v^\ast: \cH_v \rightarrow \mathbb{C}$ for the mapping $\eta \mapsto \langle \eta, \xi_v \rangle$.

\begin{prop}\label{Prop=commutation} For all $v\in V\Gamma$ and all $x\in\cB(\cH_v)$ one has:
\begin{enumerate}
\item $\lambda_v(x)\Omega=\mathcal{P}_v(x\xi_v)+\langle x\xi_v,\xi_v\rangle\Omega$.
\item $\lambda_v(x)\xi=\mathcal{P}_v(x\xi)+\langle x\xi,\xi_v\rangle\Omega$ for all $\xi\in\cH_v^{\circ}$.
\item $\lambda_v(x)\xi=\cQ_{v\boldw}(\mathcal{P}_v(x\xi_v)\ot\xi)+\langle x\xi_v,\xi_v\rangle\xi$ for all $\boldw\in\cW_v$ and all $\xi\in\cH_\boldw$.
\item If $\boldw\in\cW^c_v$ then there exists a unique $\boldw_v\in\cW_v$ such that $\boldw\simeq v\boldw_v$ are II equivalent and, for all $\xi\in\cH_\boldw$, one has
$$\lambda_v(x)\xi=\cQ_{v\boldw_v}(\cP_vx\ot\id)\cQ_{v\boldw_v}^*\xi+(\xi_v^*x\ot\id)\cQ_{v \boldw_v}^*\xi.$$
\end{enumerate}
Moreover, the images of $\lambda_v$ and $\lambda_{v'}$ commute whenever $(v,v')\in E\Gamma$.
\end{prop}

\begin{proof}
The first part of the proposition is an immediate consequence of the definition of $U_v$.

\noindent $(1)$ One has
$
\lambda_v(x)\Omega =U_v(x\xi_v\ot \Omega)=U_v(\cP_v(x\xi_v)\ot\Omega)+\langle x\xi_v,\xi_v\rangle U_v(\xi_v\ot\Omega)=\cP_v(x\xi_v)+\langle x\xi_v,\xi_v\rangle\Omega.
$

\noindent $(2)$ Let $\xi\in\cH^{\circ}_v$, one has $\lambda_v(x)\xi=U_v(x\xi\ot\Omega)=\cP_v(x\xi)+\langle x\xi,\xi_v\rangle\Omega$.

\noindent $(3)$ Let $\boldw\in\cW_v$ and $\xi\in\cH_\boldw$, one has $\lambda_v(x)\xi=U_v(x\xi_v\ot\xi)=\cQ_{v\boldw}(\cP_v(x\xi_v)\ot\xi)+\langle x\xi_v,\xi_v\rangle\xi$.

\noindent $(4)$ Let $\boldw_v\in\cW_v$, $\xi\in\cH_v^{\circ}$ and $\eta\in\cH_{\boldw_v}$. We find $\lambda_v(x)\cQ_{v\boldw_v}(\xi\ot\eta)=\cQ_{v\boldw_v}(\mathcal{P}_v(x\xi)\ot\eta))+\langle x\xi,\xi_v\rangle\eta$. Hence, for all $\xi\in\cH_{v\boldw_v}$, one has $\lambda_v(x)\cQ_{v\boldw_v}\xi=\cQ_{v\boldw_v}(\mathcal{P}_vx\ot\id)\xi+(\xi_v^*x\ot\id)\xi$. Since $\cQ_{v\boldw_v}\,:\,\cH_{v\boldw_v}\rightarrow\cH_\boldw$ is unitary, this gives the result.

\vspace{0.2cm}

\noindent We can now finish the proof of the proposition. Let $v,v'\in V\Gamma$ be such that $(v,v')\in E\Gamma$. Let $x\in\cB(\cH_v)$ and $y\in\cB(\cH_{v'})$. Writing $\lambda_v(x)=\lambda_v(x-\langle x\xi_v,\xi_v\rangle 1)+ \langle x\xi_v,\xi_v\rangle 1$ and $\lambda_{v'}(x)=\lambda_{v'}(y-\langle y\xi_{v'},\xi_{v'}\rangle 1)+ \langle y\xi_{v'},\xi_{v'}\rangle 1$ we see that we may and will assume that $\langle x\xi_v,\xi_v\rangle=\langle y\xi_{v'},\xi_{v'}\rangle=0$. Note that this implies that $\cP_{v}(x\xi_{v}) = x\xi_{v}$ and $\cP_{v'}(y\xi_{v'}) = y\xi_{v'}$ .

\vspace{0.2cm}

\noindent Let  $\boldw$ be the unique reduced minimal word equivalent to $(vv')$. Then $\boldw=(vv')$ or $\boldw=(v'v)$. In both cases we find $\cQ_{vv'}=\cQ_{v'v}\circ\Sigma$, where $\Sigma\,:\,\cH_{vv'}=\cH_v^{\circ}\ot\cH_{v'}^{\circ}\rightarrow \cH_{v'}^{\circ}\ot\cH_{v}^{\circ}=\cH_{v'v}$ is the flip map. Hence, we find:
$$
\lambda_v(x)\lambda_{v'}(y)\Omega = \lambda_v(x) y\xi_{v'} =\cQ_{vv'}( x\xi_v \ot y\xi_{v'})=\cQ_{v'v}(\cP_{v'}(y\xi_{v'})\ot x\xi_{v})=\lambda_{v'}(y)\lambda_{v}(x)\Omega.
$$

\noindent Let $\xi\in\cH_{v'}^{\circ}$. One has $\lambda_v(x)\lambda_{v'}(y)\xi= \lambda_v(x)( \cP_{v'} y\xi + \langle y\xi, \xi_{v'} \rangle \Omega  ) =\cQ_{vv'}( x\xi_v \ot  \cP_{v'} y\xi) + \langle y \xi, \xi_{v'} \rangle x \xi_v $ and,

\begin{eqnarray*}
\lambda_{v'}(y)\lambda_{v}(x)\xi &=& \lambda_{v'}(y)(\cQ_{vv'}( x\xi_v \ot\xi))=\lambda_{v'}(y)(\cQ_{v'v}(\xi\ot x\xi_v))=\cQ_{v'v}( \cP_{v'} y\xi\ot x\xi_v) + \langle y \xi, \xi_v \rangle x \xi_v \\
&=&\cQ_{vv'}( x\xi_v \ot \cP_{v'} y\xi) + \langle y \xi, \xi_v \rangle x \xi_v =\lambda_{v}(x)\lambda_{v'}(y)\xi.
\end{eqnarray*}

\noindent \textit{Claim. Let $(v,v')\in E\Gamma$ and $\boldw\in\cW_{\text{min}}$.
\begin{enumerate}
\item\label{Item=Claim1dot1} Suppose that $\boldw\in\cW_v\cap\cW_{v'}$ and define $\boldw_1, \boldw_2\in\cW_{\text{min}}$ such that $v'\boldw\simeq \boldw_1$ and $v\boldw\simeq \boldw_2$. Then $\boldw_1\in\cW_v$, $\boldw_2\in\cW_{v'}$, $v\boldw_1\simeq v'\boldw_2$
and, for all $\eta_v\in\cH_v^{\circ}$, $\eta_{v'}\in\cH_{v'}^{\circ}$ and $\xi\in\cH_\boldw$ one has
$$\cQ_{v\boldw_1}\left(\eta_v\ot\cQ_{v'\boldw}(\eta_{v'}\ot\xi) \right)=\cQ_{v'\boldw_2}\left(\eta_{v'}\ot\cQ_{v\boldw}(\eta_{v}\ot\xi) \right).$$
\item\label{Item=Claim1dot2} Suppose that $\boldw\in\cW_v\setminus\cW_{v'}$ and define $\boldw_1\in\cW_{v'}$, $\boldw_2\in\cW_{\text{min}}$ such that $\boldw\simeq v' \boldw_1$ and $\boldw_2\simeq v \boldw$. Then, $\boldw_1\in\cW_v$ and $\boldw_2\in\cW^c_{v'}$. Let $\boldw_3\in\cW_{v'}$ such that $\boldw_2\simeq v'\boldw_3$. For all $\xi\in\cH_w$, $y\in\cB(\cH_{v'})$, $\eta\in\cH_v^{\circ}$,
$$\eta_v\ot\cQ_{v'\boldw_1}(\cP_{v'}y\ot\id)\cQ_{v'\boldw_1}^*\xi=\cQ_{v\boldw}^*\cQ_{v'\boldw_3}(\mathcal{P}_{v'}y\ot\id)\cQ_{v'\boldw_3}^*\cQ_{v\boldw}(\eta_v\ot\xi)\quad\text{and,}$$
$$\cQ_{v\boldw_1}\left(\eta_v\ot(\xi_{v'}^*y\ot\id)\cQ_{v'\boldw_1}^*\xi\right)=(\xi_{v'}^*y\ot\id)\cQ_{v'\boldw_3}^*\cQ_{v\boldw}(\eta_v\ot\xi).$$
\item\label{Item=Claim1dot3} Suppose that $\boldw\in\cW^c_v\cap\cW^c_{v'}$ and define $\boldw_1\in\cW_{v'}$, $\boldw_2\in\cW_{v}$ such that $\boldw\simeq v'\boldw_1$ and $\boldw\simeq v\boldw_2$. Then, $\boldw_1\in\cW_v^c$ and $\boldw_2\in\cW_{v'}^c$. Define $\boldw_1'\in\cW_v$ and $\boldw_2'\in\cW_{v'}$ such that $\boldw_1\simeq v\boldw_1'$ and $\boldw_2\simeq v'\boldw_2'$. One has
$$\cQ_{v\boldw_2}(\cP_vx\ot\id)\cQ_{v\boldw_2}^*\cQ_{v'\boldw_1}(\cP_{v'}y\ot\id)\cQ_{v'\boldw_1}^*=\cQ_{v'\boldw_1}(\cP_{v'}y\ot\id)\cQ_{v'\boldw_1}^*\cQ_{v\boldw_2}(\cP_vx\ot\id)\cQ_{v\boldw_2}^*,$$
$$(\xi_v^*x\ot\id)\cQ_{v\boldw_2}^*\cQ_{v'\boldw_1}(\cP_{v'}y\ot\id)\cQ_{v'\boldw_1}^*=\cQ_{v'\boldw_2'}(\cP_{v'}y\ot\id)\cQ_{v'\boldw_2'}^*(\xi_v^*x\ot\id)\cQ_{v\boldw_2}^*,$$
$$\cQ_{v\boldw_1'}(\cP_vx\ot\id)\cQ_{v\boldw_1'}^*(\xi_{v'}^*y\ot\id)\cQ_{v'\boldw_1}^*=(\xi_{v'}^*y\ot\id)\cQ_{v'\boldw_1}^*\cQ_{v\boldw_2}(\cP_vx\ot\id)\cQ_{v\boldw_2}^*,$$
$$(\xi_v^*x\ot\id)\cQ_{v\boldw_1'}^*(\xi_{v'}^*y\ot\id)\cQ_{v'\boldw_1}^*=(\xi_{v'}^*y\ot\id)\cQ_{v'\boldw_2}^*(\xi_v^*x\ot\id)\cQ_{v\boldw_2}^*.$$
\end{enumerate}}
\vspace{0.2cm}

\noindent \textit{Proof of the Claim.} In each of the subclaims we let $u_1, \ldots, u_n$ with $u_i \in V\Gamma$ denote (part of the) letters of $\boldw$.

\vspace{0.2cm}

\noindent \eqref{Item=Claim1dot1}
We may assume that $\xi = \xi_{u_1} \otimes   \ldots \otimes \xi_{u_n}$ is a simple tensor product with $\xi_{u_i} \in \cH_{u_i}^\circ$.    Let $u_1' \ldots u_n'$ and $k$ be such that $u_1' \ldots u_k' v' u_{k+1}' \ldots u_n'$ is minimal with $u_1'\ldots u_k' \in \Link(v')$ and $u_1' \ldots u_n' \simeq \boldw$.  Let $u_1'' \ldots u_n''$ and $m_1, m_2$  be such that $u_1'' \ldots u_{m_1}'' v u_{m_1+1}'' \ldots u_{m_2}'' v' u_{m_2+1}'' \ldots  u_n''$ is minimal with $u_1'' \ldots u_{m_1}'' \in \Link(v), u_1'' \ldots u_{m_2}'' \in \Link(v')$ and $u_1'' \ldots u_n'' \simeq \boldw$ (for notational convenience we assume that $m_1 < m_2$, the other case can be treated similarly).
 Then,
\[
\begin{split}
&\cQ_{v\boldw_1}(\eta_v \otimes \cQ_{v'\boldw} (\eta_{v'} \otimes \xi))
=
\cQ_{v\boldw_1}( \eta_v \otimes \xi_{u_1'} \otimes \ldots \otimes \xi_{u_k'} \otimes \eta_{v'} \otimes \xi_{u_{k+1}'} \otimes \ldots \otimes \xi_{u_n'} ) \\
=&
\xi_{u_1''}\otimes \ldots \otimes \xi_{u_{m_1}''} \otimes \eta_v \otimes \xi_{u_{m_1+1}''} \otimes \ldots \otimes \xi_{u_{m_2}''} \otimes \eta_{v'} \otimes \xi_{u_{m_2+1}''} \otimes \ldots \otimes \xi_{u_n''},
\end{split}
\]
and using the fact that $(v, v') \in E \Gamma$ the same computation shows that this expression equals $\cQ_{v'\boldw_2}( \eta_{v'} \otimes \cQ_{v\boldw}(\eta_v \otimes \xi) )$.

\vspace{0.2cm}

\noindent \eqref{Item=Claim1dot2} We may assume that $\xi = \xi_{u_1} \otimes \ldots \otimes \xi_{u_{k}} \otimes \eta_{v'} \otimes \xi_{u_{k+1}} \otimes \ldots \otimes \xi_{u_n}$ and that $u_1\ldots u_k v' u_{k+1}\ldots u_n$ is minimal with $u_1, \ldots, u_k \in \Link(v')$. Then, letting $u_1' \ldots u_n'$ and $l$ be such that $u_1' \ldots u_l' v u_{l+1}' \ldots u_n'$ is minimal, $u_1' \ldots u_l' \in \Link(v)$ and $u_1'\ldots u_n' \simeq u_1 \ldots u_n$,  we find:
\[
\begin{split}
&\cQ_{v \boldw}^\ast \cQ_{v' \boldw_3} (\cP_{v'} y \otimes \id ) \cQ_{v' \boldw_3}^\ast \cQ_{v\boldw} (\eta_v \otimes \xi)
=  \cQ_{v \boldw}^\ast \cQ_{v' \boldw_3}  ((\cP_{v'} y \eta_{v'}) \otimes \xi_{u_1'} \otimes \ldots \otimes \xi_{u_l'} \otimes \eta_v \otimes \xi_{u_{l+1}' } \otimes \ldots \otimes \xi_{u_n'} ) \\
= &
 \eta_v \otimes \xi_{u_1} \otimes \ldots \otimes \xi_{u_k} \otimes \cP_{v'} y \eta_{v'} \otimes \xi_{u_{k+1}} \otimes \ldots \otimes \xi_{u_n}
= \eta_v \otimes \cQ_{v' \boldw_1} (\cP_{v'} y \otimes \id) \cQ_{v' \boldw_1}^\ast \xi.
\end{split}
\]
And, letting $u_1'' \ldots u_n''$ be minimal and equivalent to $u_1 \ldots u_n$ we have,
\[
\begin{split}
& (\xi_{v'}^\ast y  \otimes \id ) \cQ_{v' \boldw_3}^\ast \cQ_{v\boldw} (\eta_v \otimes \xi)
=  (\xi_{v'}^\ast y  \otimes \id )  \eta_{v'} \otimes \xi_{u_1'} \otimes \ldots \otimes \xi_{u_l'} \otimes \eta_v \otimes \xi_{u_{l+1}'} \otimes \ldots \otimes \xi_{u_n'} \\
= & \langle y \eta_{v'}, \xi_{v'} \rangle \xi_{u_1'} \otimes \ldots \otimes \xi_{u_l'} \otimes \eta_v \otimes \xi_{u_{l+1}'} \otimes \ldots \otimes \xi_{u_n'}
=  \cQ_{v\boldw_1} (\eta_v \otimes \langle y \eta_{v'}, \xi_{v'} \rangle \xi_{u_1'' } \otimes \ldots \otimes \xi_{u_n'' } ) \\
= & \cQ_{v \boldw_1} (\eta_v \otimes (\xi_{v'}^\ast y \otimes \id ) \cQ_{v'\boldw_1}^\ast \xi ).
\end{split}
\]
\eqref{Item=Claim1dot3} Assume that $\xi = \xi_{u_1} \otimes \ldots \otimes \xi_{u_{l}} \otimes \xi_{v} \otimes \xi_{u_{l+1}} \ldots \otimes \xi_{u_{k}} \otimes \xi_{v'} \otimes \xi_{u_{k+1}} \otimes \ldots \otimes \xi_{u_n}$ with $u_1 \ldots u_lv u_{l+1}\ldots u_k v u_{k+1} \ldots u_n$ minimal and $u_1, \ldots, u_l \in \Link(v)$, $u_1, \ldots u_k \in \Link(v')$. Then, using $(v,v') \in E\Gamma$,
\[
\begin{split}
& \cQ_{v\boldw_2} (\cP_vx \otimes \id) \cQ_{v\boldw_2}^\ast \cQ_{v'\boldw_1} (\cP_{v'} y \otimes \id) \cQ_{v' \boldw_1}^\ast \xi \\
= & \cQ_{v\boldw_2} (\cP_vx \otimes \id) \cQ_{v\boldw_2}^\ast  \xi_{u_1} \otimes \ldots \otimes \xi_{u_k} \otimes \cP_{v'}y \xi_{v'}  \otimes \xi_{u_{k+1}} \otimes \ldots \otimes \xi_{u_n}\\
= & \xi_{u_1} \otimes \ldots \otimes \xi_{u_l} \otimes \cP_v x \xi_v \otimes \xi_{u_{l+1}} \otimes \ldots \otimes \xi_{u_{k}} \otimes \cP_{v'} y \xi_{v'} \otimes \xi_{u_{k+1}} \otimes \ldots \otimes \xi_{u_n},
\end{split}
\]
which equals $\cQ_{v' \boldw_1} (\cP_{v'} y \otimes \id) \cQ_{v' \boldw_1}^\ast \cQ_{v\boldw_2} (\cP_v x \otimes \id) \cQ_{v \boldw_2}^\ast$ by a reverse computation. The other equalities follow in a similar way.

\vspace{0.2cm}

\noindent {\it Remainder of the proof of the Proposition.} Let $\boldw\in\cW_v\cap\cW_{v'}$ and consider the minimal words $\boldw_1$, $\boldw_2$ introduced in the first assertion of the Claim. For $\xi \in \cH_{\boldw}$, one has,
\begin{eqnarray*}
\lambda_v(x)\lambda_{v'}(y)\xi&=&\lambda_v(x)\cQ_{v'\boldw}(y\xi_{v'}\ot\xi)=\cQ_{v\boldw_1}\left(x\xi_v\ot\cQ_{v'\boldw}(y\xi_{v'}\ot\xi)\right)\\
&=&\cQ_{v\boldw_2}\left(y\xi_{v'}\ot\cQ_{v\boldw}(x\xi_{v}\ot\xi)\right)=\lambda_{v'}(y)\lambda_v(x)\xi.
\end{eqnarray*}

\noindent Let $\boldw\in\cW_v\setminus\cW_{v'}$ and consider the minimal words $\boldw_1$, $\boldw_2$, $\boldw_3$ introduced in the second assertion of the Claim. One has,
\begin{eqnarray*}
\lambda_v(x)\lambda_{v'}(y)\xi&=&\lambda_v(x)\cQ_{v'\boldw_1}(\cP_{v'}y\ot\id)\cQ_{v'\boldw_1}^*\xi+\lambda_v(x)(\xi_{v'}^*y\ot\id)\cQ_{v'\boldw_1}^*\xi\\
&=&\cQ_{v\boldw}\left(x\xi_v\ot\cQ_{v'\boldw_1}(\cP_{v'}y\ot\id)\cQ_{v'\boldw_1}^*\xi\right)+\cQ_{v\boldw_1}\left(x\xi_v\ot(\xi_{v'}^*y\ot\id)\cQ_{v'\boldw_1}^*\xi\right)\,\,\text{and},
\end{eqnarray*}
\begin{eqnarray*}
\lambda_{v'}(y)\lambda_{v}(x)\xi&=&\lambda_{v'}(y)\cQ_{v\boldw}(x\xi_v\ot\xi)\\
&=&\cQ_{v'\boldw_3}(\cP_{v'}y\ot\id)\cQ_{v'\boldw_3}^*\cQ_{v\boldw}(x\xi_v\ot\xi)+ (\xi_{v'}^*y\ot\id)\cQ_{v'\boldw_3}^*\cQ_{v\boldw}(x\xi_v\ot\xi).
\end{eqnarray*}
These two expressions are equal by the second assertion of the Claim. By symmetry, this is also true when $\boldw\in\cW_{v'}\setminus\cW_v$. Finally, let $\boldw\in\cW_v^c\cap\cW_{v'}^c$ and consider the minimal words $\boldw_1,\boldw_1',\boldw_2,\boldw_2'$ introduced in the third assertion of the Claim. One has:
$$\begin{array}{lcl}
\lambda_v(x)\lambda_{v'}(y)\xi&=&
\lambda_v(x)\cQ_{v'\boldw_1}(\cP_{v'}y\ot\id)\cQ_{v'\boldw_1}^*\xi+\lambda_v(x)(\xi_{v'}^*y\ot\id)\cQ_{v'\boldw_1}^*\xi.\\
&=&\cQ_{v\boldw_2}(\cP_vx\ot\id)\cQ_{v\boldw_2}^*\cQ_{v'\boldw_1}(\cP_{v'}y\ot\id)\cQ_{v'\boldw_1}^*\xi+(\xi_v^*x\ot\id)\cQ_{v\boldw_2}^*\cQ_{v'\boldw_1}(\cP_{v'}y\ot\id)\cQ_{v'\boldw_1}^*\xi\\
&&+\cQ_{v\boldw_1'}(\cP_vx\ot\id)\cQ_{v\boldw_1'}^*(\xi_{v'}^*y\ot\id)\cQ_{v'\boldw_1}^*\xi+(\xi_v^*x\ot\id)\cQ_{v\boldw_1'}^*(\xi_{v'}^*y\ot\id)\cQ_{v'\boldw_1}^*\xi\quad\text{and,}\\
\lambda_{v'}(y)\lambda_v(x)\xi&=&
\lambda_{v'}(y)\cQ_{v\boldw_2}(\cP_vx\ot\id)\cQ_{v\boldw_2}^*\xi+\lambda_{v'}(y)(\xi_v^*x\ot\id)\cQ_{v\boldw_2}^*\xi\\
&=&\cQ_{v'\boldw_1}(\cP_{v'}y\ot\id)\cQ_{v'\boldw_1}^*\cQ_{v\boldw_2}(\cP_vx\ot\id)\cQ_{v\boldw_2}^*\xi+(\xi_{v'}^*y\ot\id)\cQ_{v'\boldw_1}^*\cQ_{v\boldw_2}(\cP_vx\ot\id)\cQ_{v\boldw_2}^*\xi\\
&&+\cQ_{v'\boldw_2'}(\cP_{v'}y\ot\id)\cQ_{v'\boldw_2'}^*(\xi_v^*x\ot\id)\cQ_{v\boldw_2}^*\xi+(\xi_{v'}^*y\ot\id)\cQ_{v'\boldw_2}^*(\xi_v^*x\ot\id)\cQ_{v\boldw_2}^*\xi.
\end{array}$$
These two expressions are equal by the third assertion of the Claim. It concludes the proof.
\end{proof}

\noindent We can also define the right versions of the unitaries $U_v$. For $v\in V\Gamma$, let $\cW^v$ be the set of minimal reduced words $\boldw$ such that the concatenation $\boldw v$ is still reduced and write $(\cW^v)^c=\cW_{\text{min}}\setminus\cW^v$. Define
$$\mathcal{H}'(v)=\bc\Omega\oplus\bigoplus_{\boldw \in \cW^v}\mathcal{H}_{\boldw}.$$
We define the isometry $U'_v\,:\,\cH'(v)\ot\cH_v\rightarrow\cH$ in the following way:
$$\begin{array}{llcl}
U'_v\,:\,&\cH'(v)\ot\cH_v &\longrightarrow& \cH\\
&\Omega\ot\xi_v &\overset{\simeq}{\longrightarrow}&\Omega\\
&\Omega\ot\cH_v^{\circ}&\overset{\simeq}{\longrightarrow}&\cH_v^{\circ}\\
&\cH_\boldw\ot\xi_v&\overset{\simeq}{\longrightarrow}&\cH_\boldw\\
&\cH_\boldw \ot \cH_v^{\circ}&\overset{\simeq}{\longrightarrow}&\cQ_{\boldw v}(\cH_\boldw\ot\cH_v^{\circ})
\end{array}$$
As before, $U_v'$ is unitary. Define, for $v\in V\Gamma$, the faithful unital normal $*$-homomorphism $\rho_v\,:\,\mathcal{B}(\cH_v)\rightarrow\mathcal{B}(\cH)$ by $\rho_v(x)=U'_v(1\ot x)(U'_v)^*$ for all $x\in\cB(\cH_v)$. Observe the $\rho_v$ intertwines the vector states $\omega_{\xi_v}$ and $\omega_{\Omega}$. The analogue of Proposition \ref{Prop=commutation} holds. We leave the details to the reader.

\begin{prop}\label{Prop=commutationright} For all $v\in V\Gamma$ and all $x\in\cB(\cH_v)$ one has:
\begin{enumerate}
\item $\rho_v(x)\Omega=\mathcal{P}_v(x\xi_v)+\langle x\xi_v,\xi_v\rangle\Omega$.
\item $\rho_v(x)\xi=\mathcal{P}_v(x\xi)+\langle x\xi,\xi_v\rangle\Omega$ for all $\xi\in\cH_v^{\circ}$.
\item $\rho_v(x)\xi=\cQ_{\boldw v}(\xi\ot\mathcal{P}_v(x\xi_v))+\langle x\xi_v,\xi_v\rangle\xi$ for all $\boldw\in\cW^v$ and all $\xi\in\cH_\boldw$.
\item Let $\boldw\in(\cW^v)^c$ then there exists a unique $\boldw'_v\in\cW^v$ such that $\boldw\simeq \boldw'_vv$ and, for all $\xi\in\cH_\boldw$, one has
$$\rho_v(x)\xi=\cQ_{\boldw'_vv}(\id\ot\cP_vx)\cQ_{\boldw'_vv}^*\xi+(\id\ot\xi_v^*x)\cQ_{\boldw'_vv}^*\xi.$$
\end{enumerate}
Moreover, the images of $\rho_v$ and $\rho_{v'}$ commute whenever $(v,v')\in E\Gamma$.
\end{prop}

\begin{prop}\label{Prop=commutant} Let $v,v'\in V\Gamma$ and $x\in\mathcal{B}(\cH_v)$, $y\in\mathcal{B}(\cH_{v'})$.  One has
$$\lambda_v(x)\rho_{v'}(y)= \rho_{v'}(y)\lambda_v(x)\quad\text{whenever}\quad (v\neq v')\,\,\,\text{or}\,\,\,(v=v'\,\,\text{and}\,\,xy=yx).$$
\end{prop}

\begin{proof}
We may and will assume that $\langle x\xi_v,\xi_v\rangle=0=\langle y\xi_{v'},\xi_{v'}\rangle$. By Propositions \ref{Prop=commutation} and \ref{Prop=commutationright}, one has
$$\lambda_v(x)\rho_{v'}(y)\Omega=\lambda_v(x)(y\xi_{v'})=\left\{\begin{array}{lcl}
\cQ_{vv'}(x\xi_v\otimes y\xi_{v'})&\text{if}&v\neq v',\\
\cP_v(xy\xi_v)+\langle xy\xi_v,\xi_v\rangle \Omega &\text{if}&v=v'.\end{array}\right.$$
Moreover, $\rho_{v'}(y)\lambda_v(x)\Omega=\rho_{v'}(y)(x\xi_{v})=\left\{\begin{array}{lcl}
\cQ_{vv'}(x\xi_v\otimes y\xi_{v'})&\text{if}&v\neq v',\\
\cP_v(yx\xi_v)+\langle yx\xi_v,\xi_v\rangle \Omega &\text{if}&v=v'.\end{array}\right.$.

\noindent To finish the proof we need the following Claim.

\noindent\textit{Claim. Let $v,v'\in V\Gamma$, $\boldw\in\cW_{\text{min}}$ and $\xi\in\cH_\boldw$.
\begin{enumerate}
\item Suppose that $\boldw\in\cW_v\cap\cW^{v'}$. Let $\boldw_1, \boldw_2\in\cW_{\text{min}}$ be such that $\boldw_1\simeq \boldw v'$ and $\boldw_2\simeq v\boldw$.
\begin{itemize}
\item If $v\boldw v'$ is reduced then $\boldw_1\in\cW_v$, $\boldw_2\in\cW^{v'}$ and, for all $\eta_v\in\cH_v$, $\eta_{v'}\in\cH_{v'}$, one has
$$\cQ_{v\boldw_1}(\eta_v\ot\cQ_{\boldw v'}(\xi\ot\eta_{v'}))=\cQ_{\boldw_2 v'}(\cQ_{v \boldw}(\eta_v\ot\xi)\ot\eta_{v'}).$$
\item If $v\boldw v'$ is not reduced then $v=v'$, $\boldw_1= \boldw_2\simeq v \boldw\simeq \boldw v\in\cW_v^c\cap(\cW^v)^c$ and, for all $x,y\in\mathcal{B}(\cH_v)$,
$$\cQ_{v \boldw}(\cP_vx\ot\id)\cQ_{v \boldw}^*\cQ_{\boldw v}(\xi\ot y\xi_v)=\cQ_{\boldw v}(\id\ot\cP_vy)\cQ_{\boldw v}^*\cQ_{v \boldw}(x\xi_v\ot\xi)\quad\text{and,}$$
$$(\xi_v^*x\ot\id)\cQ_{v\boldw}^*\cQ_{\boldw v}(\xi\ot y\xi_v)=(\id\ot\xi_v^*y)\cQ_{\boldw v}^*\cQ_{v \boldw}(x\xi_v\ot\xi).$$
\end{itemize}
\item Suppose that $\boldw\in\cW_v^c\cap(\cW^{v'})^c$. Let $\boldw_1, \boldw_2\in\cW_{\text{min}}$ be such that $\boldw\simeq \boldw_1v'\simeq v\boldw_2$.
\begin{itemize}
\item If $\boldw_1\in\cW_v$ then $v=v'$, $\boldw_1\in\cW^v$, $w\simeq v \boldw_1\simeq \boldw_1 v$ and $\boldw_1=\boldw_2$ and, for all $x,y\in\cB(\cH_v)$,
    \[
    \begin{split}
  &\cQ_{v\boldw_1}(x\xi_v\ot(\id\ot\xi_v^*y)\cQ_{\boldw_1v}^*\xi)+\cQ_{v \boldw_1}(\cP_vx\ot\id)\cQ_{v\boldw_1}^*\cQ_{\boldw_1v}(\id\ot\cP_vy)\cQ_{\boldw_1v}^*\xi \\
  =&\cQ_{\boldw_1v}((\xi_v^*x\ot\id)\cQ_{v\boldw_1}^*\xi\ot y\xi_v)+\cQ_{\boldw_1v}(\id\ot\cP_vy)\cQ_{\boldw_1v}^*\cQ_{v\boldw_1}(\cP_vx\ot\id)\cQ_{v\boldw_1}^*\xi,
  \end{split}
  \]
$$(\xi_v^*x\ot\id)\cQ_{v\boldw_1}^*\cQ_{\boldw_1v}(\id\ot\cP_vy)\cQ_{\boldw_1v}^*\xi=(\id\ot\xi_v^*y)\cQ_{\boldw_1v}^*\cQ_{v\boldw_1}(\cP_vx\ot\id)\cQ_{v\boldw_1}^*\xi.$$
\item If $\boldw_1\in\cW_v^c$ write $\boldw_1\simeq v\boldw_3$, $\boldw_3\in\cW_{\text{min}}$ then $\boldw_2\simeq \boldw_3v'\in(\cW^{v'})^c$ and, $\forall x\in\cB(\cH_v)$, $y\in\cB(\cH_{v'})$,
$$\cQ_{v\boldw_2}(\cP_vx\ot\id)\cQ_{v\boldw_2}^*\cQ_{\boldw_1v'}(\id\ot\cP_{v'}y)\cQ_{\boldw_1v'}^*\xi=\cQ_{\boldw_1v'}(\id\ot\cP_{v'}y)\cQ_{\boldw_1v'}^*\cQ_{v\boldw_2}(\cP_{v}x\ot\id)\cQ_{v\boldw_2}^*\xi,$$
$$(\xi_v^*x\ot\id)\cQ_{v\boldw_2}^*\cQ_{\boldw_1v'}(\id\ot\cP_{v'}y)\cQ_{\boldw_1v'}^*\xi=\cQ_{\boldw_3v'}(\id\ot\cP_{v'}y)\cQ_{\boldw_3v'}^*(\xi_{v}^*x\ot\id)\cQ_{v\boldw_2}^*\xi,$$
$$\cQ_{v\boldw_3}(\cP_vx\ot\id)\cQ_{v\boldw_3}^*(\id\ot\xi_{v'}^*y)\cQ_{\boldw_1v'}^*\xi=(\id\ot\xi_{v'}^*y)\cQ_{\boldw_1v'}^*\cQ_{v\boldw_2}(\cP_{v}x\ot\id)\cQ_{v\boldw_2}^*\xi,$$
$$(\xi_v^*x\ot\id)\cQ_{v\boldw_3}^*(\id\ot\xi_{v'}^*y)\cQ_{\boldw_1v'}^*\xi=(\id\ot\xi_{v'}^*y)\cQ_{\boldw_3v'}^*(\xi_{v}^*x\ot\id)\cQ_{v\boldw_2}^*\xi.$$
\end{itemize}
\item If $\boldw\in\cW_v^c\cap\cW^{v'}$ write $\boldw_1\simeq \boldw v'$, $w\simeq v\boldw_2$, $\boldw_1,\boldw_2\in\cW_{\text{min}}$. Then, $\boldw_1\in\cW_v^c$, $\boldw_2\in\cW^{v'}$ and, if $\boldw_3\in\cW_{\text{min}}$ is such that $\boldw_1\simeq v \boldw_3$, then we have $\boldw_2 v'\simeq \boldw_3$ and, for all $x\in\cB(\cH_v)$, $y \in \cB(\cH_{v'})$,
$$\cQ_{v\boldw_3}(\cP_vx\ot\id)\cQ_{v \boldw_3}^*\cQ_{\boldw v'}(\xi\ot y\xi_{v'})=\cQ_{\boldw v'}\left(\cQ_{v \boldw_2}(\cP_vx\ot\id)\cQ_{v \boldw_2}^*\xi\ot y\xi_{v'}\right),$$
$$(\xi_v^*x\ot\id)\cQ_{v \boldw_3}^*\cQ_{\boldw v'}(\xi\ot y\xi_{v'})=\cQ_{\boldw_2 v'}(\xi_v^*x\ot\id)\cQ_{v \boldw_2}^*\xi \otimes y \xi_{v'}.$$
\item If $\boldw\in\cW_v\cap(\cW^{v'})^c$ write $\boldw\simeq w_1v'$, $\boldw_2\simeq v \boldw$, $\boldw_1,\boldw_2\in\cW_{\text{min}}$. Then, $\boldw_1\in\cW_v$, $\boldw_2\in(\cW^{v'})^c$ and, if $\boldw_3\in\cW_{\text{min}}$ is such that $\boldw_2\simeq \boldw_3v'$, then we have $\boldw_3\simeq v \boldw_1$ and, for all $x\in\cB(\cH_v)$, $y \in \cB(\cH_{v'})$,
$$\cQ_{v\boldw}\left(x\xi_v\ot\cQ_{\boldw_1 v'}(\id\ot\cP_{v'}y)\cQ_{\boldw_1 v'}^*\xi\right)=\cQ_{\boldw_3 v'}(\id\ot\cP_{v'}y)\cQ_{\boldw_3 v'}^*\cQ_{v \boldw}(x\xi_v\ot\xi),$$
$$\cQ_{v \boldw_1}(x\xi_v\ot(\id\ot\xi_{v'}^*y)\cQ_{\boldw_1 v'}^*\xi)=(\id\ot\xi_{v'}^*y)\cQ_{\boldw_3 v'}^*\cQ_{v \boldw}(x\xi_v\ot\xi).$$
\end{enumerate}}

\vspace{0.2cm}

\noindent The proof of the Claim is analogous to the proof of the Claim used in the proof of Proposition \ref{Prop=commutation} and we shall leave the details to the reader.

\vspace{0.3cm}

\noindent\textit{Remainder of the proof of the Proposition.} Let $\boldw\in\cW_{\text{min}}$ and $\xi\in\cH_\boldw$. We use freely the results and notations of the Claim and Propositions \ref{Prop=commutation} and \ref{Prop=commutationright}.

\vspace{0.2cm}

\noindent\textbf{Case 1: $\boldw\in\cW_v\cap\cW^{v'}$}. If moreover $v\boldw v'$ is reduced we have,
\begin{eqnarray*}
\lambda_v(x)\rho_{v'}(y)\xi&=&\lambda_v(x)(\cQ_{\boldw v'}(\xi\ot y\xi_{v'}))=\cQ_{v \boldw_1}(x\xi_v\ot\cQ_{\boldw v'}(\xi\ot y\xi_{v'})),\\
\rho_{v'}(y)\lambda_v(x)\xi&=&\rho_{v'}(y)(\cQ_{v \boldw}(x\xi_v\ot\xi))=\cQ_{\boldw_2 v'}(\cQ_{v \boldw}(x\xi_v\ot\xi)\ot y\xi_{v'}).
\end{eqnarray*}
These two expressions are equal by the Claim. Suppose now that $v\boldw v'$ is not reduced. Then $v=v'$ and,
\begin{eqnarray*}
\lambda_v(x)\rho_{v}(y)\xi&=&\lambda_v(x)(\cQ_{\boldw v}(\xi\ot y\xi_{v}))=\cQ_{v \boldw}(\cP_vx\ot\id)\cQ_{v \boldw}^*\cQ_{\boldw v}(\xi\ot y\xi_v)
+(\xi_v^*x\ot\id)\cQ_{v \boldw}^*\cQ_{\boldw v}(\xi\ot y\xi_v),\\
\rho_v(y)\lambda_v(x)\xi&=&\rho_v(y)(\cQ_{v \boldw}(x\xi_v\ot\xi))=\cQ_{\boldw v}(\id\ot\cP_vy)\cQ_{\boldw v}^*\cQ_{v \boldw}(x\xi_v\ot\xi)
+(\id\ot\xi_v^*y)\cQ_{\boldw v}^*\cQ_{v \boldw}(x\xi_v\ot\xi).
\end{eqnarray*}
These two expressions are equal by the Claim.

\vspace{0.2cm}

\noindent\textbf{Case 2: $\boldw \in\cW_v^c\cap(\cW^{v'})^c$}. If moreover $\boldw_1\in\cW_v$ then $v=v'$, $\boldw_1=\boldw_2\in\cW^v$, $\boldw\simeq v \boldw_1\simeq \boldw_1v$ and,
\begin{eqnarray*}
\lambda_v(x)\rho_v(y)\xi&=&\lambda_v(x)\left(\cQ_{\boldw_1 v}(\id\ot\cP_vy)\cQ_{\boldw_1 v}^*\xi+(\id\ot\xi_v^*y)\cQ_{\boldw_1 v}^*\xi\right)\\
&=&\cQ_{v \boldw_1}(\cP_vx\ot\id)\cQ_{v \boldw_1}^*\cQ_{\boldw_1 v}(\id\ot\cP_vy)\cQ_{\boldw_1 v}^*\xi
+(\xi_v^*x\ot\id)\cQ_{v \boldw_1}^*\cQ_{\boldw_1 v}(\id\ot\cP_vy)\cQ_{\boldw_1 v}^*\xi\\
&&+\cQ_{v \boldw_1}(x\xi_v\ot(\id\ot\xi_v^*y)\cQ_{\boldw_1 v}^*\xi),\\
\rho_v(y)\lambda_v(x)\xi&=&\rho_v(y)\left(\cQ_{v \boldw_1}(\cP_vx\ot\id)\cQ_{v \boldw_1}^*\xi+(\xi_v^*x\ot\id)\cQ_{v \boldw_1}^*\xi\right)\\
&=&\cQ_{\boldw_1 v}(\id\ot\cP_vy)\cQ_{\boldw_1 v}^*\cQ_{v \boldw_1}(\cP_vx\ot\id)\cQ_{v \boldw_1}^*\xi
+(\id\ot\xi_v^*y)\cQ_{\boldw_1 v}^*\cQ_{v \boldw_1}(\cP_vx\ot\id)\cQ_{v\boldw_1}^*\xi\\
&&+\cQ_{\boldw_1 v}((\xi_v^*x\ot\id)\cQ_{v \boldw_1}^*\xi\ot y\xi_v).\\
\end{eqnarray*}
These two expressions are equal by the Claim. Suppose now that $\boldw_1\in\cW_{v}^c$, $\boldw_1\simeq v \boldw_3$, $\boldw_3\in\cW_{\text{min}}$. We have:
\begin{eqnarray*}
\lambda_v(x)\rho_{v'}(y)\xi&=&\lambda_v(x)\left(\cQ_{\boldw_1 v'}(\id\ot\cP_{v'}y)\cQ_{\boldw_1 v'}^*\xi+(\id\ot\xi_{v'}^*y)\cQ_{\boldw_1v'}^*\xi\right)\\
&=&\cQ_{v\boldw_2}(\cP_vx\ot\id)\cQ_{v \boldw_2}^*\cQ_{\boldw_1v'}(\id\ot\cP_{v'}y)\cQ_{\boldw_1v'}^*\xi
+(\xi_v^*x\ot\id)\cQ_{v \boldw_2}^*\cQ_{\boldw_1 v'}(\id\ot\cP_{v'}y)\cQ_{\boldw_1 v'}^*\xi\\
&&+\cQ_{v\boldw_3}(\cP_vx\ot\id)\cQ_{v \boldw_3}^*(\id\ot\xi_{v'}^*y)\cQ_{\boldw_1 v'}^*\xi+(\xi_v^*x\ot\id)\cQ_{v \boldw_3}^*(\id\ot\xi_{v'}^*y)\cQ_{\boldw_1v'}^*\xi.
\end{eqnarray*}
Moreover, since $\boldw_2\in(\cW^{v'})^c$ and $\boldw_2\simeq \boldw_3v'$, we find,
\begin{eqnarray*}
\rho_{v'}(y)\lambda_v(x)\xi&=&\rho_{v'}(y)\left(\cQ_{v \boldw_2}(\cP_vx\ot\id)\cQ_{v \boldw_2}^*\xi+(\xi_v^*x\ot\id)\cQ_{v \boldw_2}^*\xi\right)\\
&=&\cQ_{\boldw_1 v'}(\id\ot\cP_{v'}y)\cQ_{\boldw_1 v'}^*\cQ_{v\boldw_2}(\cP_{v}x\ot\id)\cQ_{v \boldw_2}^*\xi
+(\id\ot\xi_{v'}^*y)\cQ_{\boldw_1 v'}^*\cQ_{v \boldw_2}(\cP_{v}x\ot\id)\cQ_{v \boldw_2}^*\xi\\
&&+\cQ_{\boldw_3 v'}(\id\ot\cP_{v'}y)\cQ_{\boldw_3 v'}^*(\xi_{v}^*x\ot\id)\cQ_{v\boldw_2}^*\xi+(\id\ot\xi_{v'}^*y)\cQ_{\boldw_3v'}^*(\xi_{v}^*x\ot\id)\cQ_{v\boldw_2}^*\xi.
\end{eqnarray*}
These two expressions are equal by the Claim.
\vspace{0.2cm}

\noindent\textbf{Case 3: $\boldw\in\cW_v^c\cap\cW^{v'}$}. We have,
\begin{eqnarray*}
\lambda_v(x)\rho_{v'}(y)\xi&=&\lambda_v(x)(\cQ_{\boldw v'}(\xi\ot y\xi_{v'}))=\cQ_{v \boldw_3}(\cP_vx\ot\id)\cQ_{v \boldw_3}^*\cQ_{\boldw v'}(\xi\ot y\xi_{v'})
+(\xi_v^*x\ot\id)\cQ_{v \boldw_3}^*\cQ_{\boldw v'}(\xi\ot y\xi_{v'}),\\
\rho_{v'}(y)\lambda_v(x)\xi&=&\rho_{v'}(y)\left(\cQ_{v \boldw_2}(\cP_vx\ot\id)\cQ_{v \boldw_2}^*\xi+(\xi_v^*x\ot\id)\cQ_{v\boldw_2}^*\xi\right)\\
&=&\cQ_{\boldw v'}\left(\cQ_{v \boldw_2}(\cP_vx\ot\id)\cQ_{v \boldw_2}^*\xi\ot y\xi_{v'}\right)
+\cQ_{\boldw_2 v'}(\xi_v^*x\ot\id)\cQ_{v \boldw_2}^*\xi \otimes y \xi_{v'}.
\end{eqnarray*}
These two expressions are equal by the Claim.

\vspace{0.2cm}

\noindent\textbf{Case 4: $\boldw\in\cW_v\cap(\cW^{v'})^c$}. We have,
\begin{eqnarray*}
\lambda_v(x)\rho_{v'}(y)\xi&=&\lambda_v(x)\left(\cQ_{\boldw_1 v'}(\id\ot\cP_{v'}y)\cQ_{\boldw_1 v'}^*\xi+(\id\ot\xi_{v'}^*y)\cQ_{\boldw_1 v'}^*\xi\right)\\
&=&\cQ_{v \boldw}\left(x\xi_v\ot\cQ_{\boldw_1v'}(\id\ot\cP_{v'}y)\cQ_{\boldw_1v'}^*\xi\right)+\cQ_{v\boldw_1}(x\xi_v\ot(\id\ot\xi_{v'}^*y)\cQ_{\boldw_1v'}^*\xi),\\
\rho_{v'}(y)\lambda_v(x)\xi&=&\rho_{v'}(y)(\cQ_{v\boldw}(x\xi_v\ot\xi))
=\cQ_{\boldw_3 v'}(\id\ot\cP_{v'}y)\cQ_{\boldw_3v'}^*\cQ_{v\boldw}(x\xi_v\ot\xi)+(\id\ot\xi_{v'}^*y)\cQ_{\boldw_3 v'}^*\cQ_{v\boldw}(x\xi_v\ot\xi).
\end{eqnarray*}
These two expressions are equal by the Claim.
\end{proof}

\subsection{The graph product C*-algebra}

For all $v\in V\Gamma$, let $\bA_v$ be a unital C*-algebra.

\subsubsection{The maximal graph product C$^\ast$-algebra}

\begin{dfn}
The maximal graph product C*-algebra $\bA_{\Gamma,m}$ is the universal unital C*-algebra generated by the C*-algebras $\bA_v$, for $v\in V\Gamma$ and the relations
$$a_va_{v'}=a_{v'}a_v\quad\text{for all}\quad a_v\in\bA_v,a_{v'}\in\bA_{v'}\quad\text{whenever}\quad(v,v')\in E\Gamma.$$
Here the unit of $\bA_{\Gamma,m}$ is presumed to agree with the unit of each $\bA_v$.
\end{dfn}

\begin{rmk}
It is clear that $\bA_{\Gamma,m}$ is not $\{0\}$ i.e. that the relations admit a non-trivial representation as bounded operators. Indeed, for any family of representations $\pi_v\,:\,\bA_v\rightarrow\cB(\cH_v)$ and any family of norm one vectors $\xi_v\in\cH_v$, the representations $\widetilde{\pi}_v=\lambda_v\circ\pi_v\,:\,\bA_v\rightarrow\cB(\cH)$, where $\cH$ is the graph product Hilbert space of the family of pointed Hilbert spaces $(\cH_v,\xi_v)_{v\in V\Gamma}$ and $\lambda_v\,:\,\cB(\cH_v)\rightarrow\cB(\cH)$ are the unital faithful morphisms defined in Section \ref{Section=GraphProductHilbert}, satisfy the relations $\widetilde{\pi}_{v}(a_v)\widetilde{\pi}_{v'}(a_{v'})=\widetilde{\pi}_{v'}(a_{v'})\widetilde{\pi}_{v}(a_v)$ for all $a_v\in\bA_v,a_{v'}\in\bA_{v'}$ and all $v,v'\in V\Gamma$ such that $(v,v')\in E\Gamma$ by Proposition \ref{Prop=commutation}. The associated representation $\pi\,:\,\bA_{\Gamma,m}\rightarrow\cB(\cH)$ such that $\pi|_{\bA_v}=\widetilde{\pi}_v$ for all $v\in V\Gamma$ obtained by the universal property is called the \textit{graph product representation}.
\end{rmk}

\begin{example}
Using the universal property of $\bA_{¤\Gamma,m}$ one can easily check the following statements.
\begin{itemize}
\item Let $\bA_v=C_m^*(G_v)$ be the maximal C*-algebra of a discrete group $G_v$, $v\in V\Gamma$. Then $\bA_{\Gamma,m}=C^*_m(G_\Gamma)$.
\item Let $\Gamma$ be a finite graph having every two vertices connected by one edge. Then $\bA_{\Gamma,m}= \underset{v \in V\Gamma, \textrm{max}}{\ot} \bA_{v}$.
\item Let $\Gamma$ be a graph with no edges. Then $\bA_{\Gamma,m}=\underset{v\in V\Gamma}{*^m}\bA_v$, where $*^m$ denotes the maximal free product.
\item If $\Gamma=\Star(v)$ then $\bA_{\Gamma,m}$ is a quotient of $\left(\underset{w\in\Link(v)}{*^m} \bA_w\right)\underset{\text{max}}{\otimes} \bA_v$.
\end{itemize}
\end{example}

\begin{rmk}\label{Rmk=MaxDenseSubalgebra}
Let $\mathcal{A}\subset\bA_{\Gamma,m}$ be the linear span of elements of the form $a_1\dots a_n$ with $n\geq 1$ and $a_k\in\bA_{v_k}$, where $\boldv=(v_1,\dots,v_n)$ is a reduced word. Observe that $\mathcal{A}$ is a dense $*$-subalgebra of $\bA_{\Gamma,m}$. Indeed, the commutation relations defining $\bA_{\Gamma,m}$ show that $\mathcal{A}$ is a $*$-subalgebra. It is dense since it contains all the $\bA_v$. Moreover, if $a=a_1\dots a_n\in\bA_{\Gamma,m}$ with $n\geq 1$, $a_k\in\bA_{v_k}$ and $\boldv=(v_1,\dots,v_n)$ is a reduced word and if $\boldw=(w_1,\dots w_n)$ is a reduced word (type ${\rm II}$) equivalent to $\boldv$ it follows from the commutation relations that $a=a_{\sigma(1)}\dots a_{\sigma(n)}$, where $\sigma\in S_n$ is the unique permutation such that $\boldw=\sigma(\boldv)$ defined in Lemma \ref{Lem=ReducedRep} \eqref{Item=ReducedRepIV}.
\end{rmk}

\subsubsection{The reduced graph product C$^\ast$-algebra}

From this point we assume that each unital C$^\ast$-algebra $\bA_v, v \in V \Gamma$ is equipped with a GNS-faithful state $\omega_v$. Since the GNS-representation is faithful we may assume that $\bA_v\subset\cB(\cH_v)$, where $(\cH_v,\id,\xi_v)$ is a GNS-construction for $\omega_v$. Let $(\cH,\Omega)$ be the graph product of the pointed Hilbert spaces $(\cH_v,\omega_v)$. Recall that $\cH$ comes with faithful unital normal $*$-homomorphisms $\lambda_v\,:\,\cB(\cH_v)\rightarrow\cB(\cH)$.

\begin{dfn}\label{Dfn=GraphProduct}
The \textit{reduced graph product} C*-algebra $\bA_\Gamma$ is defined as the sub-C*-algebra of $\cB(\cH)$ generated by $\bigcup_{v\in V\Gamma}\lambda_v(\bA_v)$.
\end{dfn}

\noindent Since the $\lambda_v$ are faithful, we may and will assume that $\bA_v\subset\bA_\Gamma$ and $\lambda_v|_{\bA_v}$ is the inclusion for all $v\in V\Gamma$.
\begin{rmk}
 It follows from Proposition \ref{Prop=commutation} that there exists a unique unital surjective $*$-homomorphism $\lambda_\Gamma\,:\,\bA_{\Gamma,m}\rightarrow\bA_\Gamma$ such that $\lambda_\Gamma(a)=a$ for all $a\in\bA_v$ and all $v\in V\Gamma$. Moreover, suppose that $a=a_1\dots a_n\in\bA$ with  $a_k\in\bA_{v_k}$ and $\boldv=(v_1,\dots,v_n)$ a reduced word. Let $\boldw=(w_1,\dots w_n)$ be a reduced word that is equivalent to $\boldv$. Let $\sigma \in S_n$ be the permutation obtained from Lemma \ref{Lem=ReducedRep} \eqref{Item=ReducedRepIV} using the words $\boldv$ and $\boldw$ instead of $\boldw$ and $\boldw'$ respectively. Then it follows from the commutation relations that $a=a_{\sigma(1)}\dots a_{\sigma(n)}$.
\end{rmk}

\begin{dfn} An operator $a=a_1\dots a_n\in \bA_\Gamma$ is called \textit{reduced} if $a_i\in \bA_{v_i}^{\circ}$ with $\bA_{v_i}^{\circ} = \{ x \in \bA_{v_i} \mid \omega_{v_i}(x) = 0\}$ and the word $\boldv=(v_1,\dots,v_n)$ is reduced. The word $\boldv$ is called the \textit{associated word}.
\end{dfn}

\noindent Observe that the linear span of $1$ and the reduced operators in a dense $*$-subalgebra of $\bA_\Gamma$.

\begin{rmk}\label{Rmk=MaxReducedOp}
For all $v\in V\Gamma$, let $\omega_v$ be a not necessarily GNS-faithful state on $\bA_v$. The notion of reduced operators, relative to the family of states $(\omega_v)_{v\in V\Gamma}$, also makes sense in the maximal graph product C*-algebra and the linear span of $1$ and the reduced operators in the maximal graph product C*-algebra is the $*$-algebra $\mathcal{A}$ introduced in Remark \ref{Rmk=MaxDenseSubalgebra}, which is dense.
\end{rmk}

\noindent It is clear from Proposition \ref{Prop=commutation} that, whenever $a=a_1\dots a_n\in \bA_\Gamma$ is a reduced operator (with associated word in $\mathcal{W}_{\text{min}}$) one has $a\Omega=\widehat{a}_1\otimes\dots\otimes\widehat{a}_n$. Hence, the vector $\Omega$ is cyclic for $\bA_\Gamma$ and $(\mathcal{H},\id,\Omega)$ is a GNS-construction for the (GNS-faithful) state $\omega_\Gamma(\: \cdot \:) =\langle \:\cdot\: \Omega,\Omega\rangle$. We call $\omega_\Gamma$ the \textit{graph product state}. It can be characterized as follows: it is the unique state on $\bA_\Gamma$ satisfying $\omega_\Gamma(a)=0$ for all reduced operators $a\in \bA_{\Gamma}$. In particular, $\omega_\Gamma|_{\bA_v}=\omega_v$ for all $v\in V\Gamma$. Actually the commutation relations and the properties of the graph product state determine the graph product C*-algebra.

\begin{prop}\label{Prop=UniversalPropertyReduced}
Let $\bB$ be a unital C*-algebra with a GNS-faithful state $\omega$ and suppose that, for all $v\in V\Gamma$, there exists a unital faithful $*$-homomorphism $\pi_v\,:\,\bA_v\rightarrow \bB$ such that:
\begin{itemize}
\item $\bB$ is generated, as a C*-algebra, by $\cup_{v\in V\Gamma}\pi_v(\bA_v)$ and the images of $\pi_v$ and $\pi_{v'}$ commute whenever $(v,v')\in E\Gamma$.
\item For any operator $a=\pi_{v_1}(a_1)\dots\pi_{v_n}(a_n)\in \bB$, where $\boldv=(v_1,\dots v_n)$ is a reduced word and $a_i\in \bA_{v_i}^{\circ}$ one has $\omega(a)=0$
\end{itemize}
Then, there exists a unique $*$-isomorphism $\pi\,:\,\bA_\Gamma\rightarrow \bB$ such that $\pi|_{\bA_v}=\pi_v$. Moreover, $\pi$ intertwines the graph product state and $\omega$.
\end{prop}

\begin{proof}
The proof is a routine. We include it for the convenience of the reader. The uniqueness being obvious, let us show the existence. Since $\omega$ is GNS-faithful we may and will assume that $B\subset\cB(\cK)$ and $(\cK,\id,\eta)$ is a GNS-construction for $\omega$. Define $V\,:\,\cH\rightarrow \cK$ by $V(\Omega)=\eta$ and,
$$V(a_1\dots a_n)\Omega=\pi_{v_1}(a_1)\dots\pi_{v_n}(a_n)\eta\quad\text{for all reduced}\,\,a=a_1\dots a_n\in\bA_\Gamma\,\,\text{with associated word}\,\,(v_1,\dots,v_n).$$
It is easy to check that $V$ is well defined and isometric hence, it extends to an isometry. Since it also has a dense image, it is a unitary. Then, $\pi(x):=VxV^*$ does the job.
\end{proof}

\begin{rmk} Proposition \ref{Prop=UniversalPropertyReduced} implies the following.
\begin{itemize}
\item Let $\bA_v = C^\ast_r(G_v)$ be the reduced group C$^\ast$-algebra of a discrete group $G_v, v \in V\Gamma$. Then $(\bA_\Gamma,\omega_\Gamma) = (C_r^\ast(G_\Gamma),\tau)$, where $\tau$ is the canonical trace on the reduced C$^\ast$-algebra of the graph product group $G_\Gamma$.
\item Let $\Gamma$ be a graph in which every two vertices are connected by one edge. Then $(\bA_\Gamma,\omega_\Gamma)= \underset{v \in V\Gamma}{\ot} (\bA_{v},\omega_v)$.
\item Let $\Gamma$ be a graph with no edges. Then $(\bA_{\Gamma},\omega_\Gamma)=\underset{v\in V\Gamma}{*}(\bA_v,\omega_v)$, the reduced free product with respect to the GNS-faitfhul states $\omega_v$, $v\in V\Gamma$.
\item If $\Gamma_0\subset\Gamma$ is a subgraph and, for all $v\in V_{\Gamma_0}$, $\bB_v\subset \bA_v$ is a unital C*-algebra then the sub-C*-algebra of $\bA_\Gamma$ generated by $\cup_{v\in V\Gamma_0} \bB_v$ is canonically isomorphic to graph product C*-algebras $\bB_{\Gamma_0}$ obtained from $\bB_v$, $v\in V_{\Gamma_0}$.
\end{itemize}
\end{rmk}

\begin{rmk}\label{Rmk=CE}
Let $\Gamma_0 \subseteq \Gamma$ be a subgraph and consider the graph product C$^\ast$-algebras $\bA_{\Gamma_0}$ and $\bA_{\Gamma}$. By the universal property of Proposition \ref{Prop=UniversalPropertyReduced}, we may view $\bA_{\Gamma_0}\subset\bA_\Gamma$ canonically. Denote by $\mathcal{W}_{\text{min}}^0\subset\cW_{\text{min}}$ the subset of minimal reduced words in $\Gamma_0$ and let $\cH_0=\bc\Omega\oplus\bigoplus_{\boldw\in\cW_{\text{min}}^0}\cH_{\boldw} \subset \mathcal{H}$. Let $P$ be the orthogonal projection onto $\cH_0$. Then, it is easy to check that $\cE_{\Gamma_0}: x \mapsto P x P$ is a graph product state-preserving conditional expectation from $\bA_{\Gamma}$ onto $\bA_{\Gamma_0}$. In particular, $\cE_{\Gamma_0}$ is GNS-faithful since it preserves the graph product states which are GNS-faithful. Moreover, $\cE_{\Gamma_0}$ is the unique conditional expectation from $\bA_\Gamma$ to $\bA_{\Gamma_0}$ such that $\cE_{\Gamma_0}(a)=0$ for all reduced operators $a\in \bA_{\Gamma_0}$ with associated reduced word $\boldv=(v_1,\dots, v_n)$ satisfying the property that one of the $v_i$ is not in $\Gamma_0$. In particular, for all $v\in V\Gamma$, there exists a unique conditional expectation $\cE_v\,:\,\bA_\Gamma\rightarrow\bA_v$ such that $\cE_v(a)=0$ for all reduced operators $a\in \bA_\Gamma\setminus \bA_v$.
\end{rmk}

\subsubsection{Unscrewing technique}

Let $v \in \Gamma$, $\Gamma_1=\Star(v)$, $\Gamma_2=\Gamma\setminus\{v\}$ and set the following graph product C$^\ast$-algebras: $\bA_1 = \bA_{\Gamma_1}$, $\bB = \bA_{\Link(v)}$, and $\bA_2 = \bA_{\Gamma_2}$. By convention $\bA_\emptyset = \mathbb{C}$.  Recall that, by the universal property of Proposition \ref{Prop=UniversalPropertyReduced}, we may view $\bB\subset\bA_1\subset\bA_\Gamma$ and $\bB\subset\bA_2\subset\bA_\Gamma$ canonically. Moreover, by Remark \ref{Rmk=CE}, we have GNS-faithful conditional expectations $\cE_1:=\cE_{\Link(v)}|_{\bA_1}\,:\,\bA_1\rightarrow \bB$ and $\cE_2:=\cE_{\Link(v)}|_{\bA_2}\,:\,\bA_2\rightarrow \bB$. Let us denote by $\bA_1 \star_\bB \bA_2$ the reduced amalgamated free product with respect to these conditional expectations.

\begin{thm}\label{Thm=Amalgam}
There exists a unique $\ast$-isomorphism $\pi:  \bA_1 \star_\bB \bA_2 \rightarrow \bA_\Gamma $ such that $\pi|_{\bA_1}$ (resp. $\pi|_{\bA_2}$) is the canonical inclusion $\bA_1\subset \bA_\Gamma$ (resp. $\bA_2\subset \bA_\Gamma$). Moreover, $\pi$ is state-preserving.
\end{thm}

\begin{proof}
Observe that $\bA_\Gamma$ is generated by by $\bA_1$ and $\bA_2$. Let $\cE=\cE_{\Link(v)}$ be the canonical conditional expectation from $\bA_\Gamma$ onto $\bB$. Define, for $k=1,2$, $\bA_k^{\circ}=\text{ker}(\cE_k)$. By the universal property of amalgamated free products it suffices to show that for any $n\geq 2$, for any $a_1,\dots,a_n$ with $a_k\in\bA_{l_k}^{\circ}$ and $l_k\neq l_{k+1}$, one has $\cE(a_1\dots a_n)=0$. Since $\bA_k^{\circ}$ is the closed linear span of reduced operators $a\in \bA_k$ with associated reduced word $\boldv=(v_1,\ldots, v_n)$, $v_i\in\Gamma_k$ satisfying the property that one of the $v_i$ is not in $\Link(v)$ we may and will assume that each $a_k$ is a reduced operator $a_k=x_1^k\dots x_{r_k}^k\in \bA_{l_k}$ with associated reduced word $\boldv_k=(v_1^k,\dots v_{r_k}^k)$, $v_i^k\in\Gamma_{l_k}$ satisfying the property that one of the $v_i^k$ is not in $\Link(v)$. One has $a:=a_1\dots a_n=x_1^1\dots x_{r_1}^1x_1^2\dots x_{r_2}^2\dots x_{1}^n\dots x_{r_n}^n$ with $x_{r_i}^k\in \bA_{v_i^k}^{\circ}$. Let $\boldv=(v_1^1,\dots, v_{r_1}^1,v_1^2,\dots, v_{r_2}^2,\dots, v_{1}^n,\dots, v_{r_n}^n)$ be the associated, not necessarily reduced, word. Let $l=r_1+\dots +r_n\geq n$.

\vspace{0.2cm}

\noindent Let us show, by induction on $l$, that $\cE(a)=0$. If $l=n$ then $a_k\in \bA_{v_k}^{\circ}\subset \bA_{l_k}$ and $v_k\in\Gamma_{l_k}\setminus\Link(v)$ for all $k$. Then $\boldv$ is reduced and since $v_k\notin\Link(v)$ we have $\cE(a)=0$. Indeed, if $\boldv$ is not reduced, there exists $i<j$ such that $v_i=v_j=w$ and $v_k\in\Link(v)$ for all $i< k< j$. Since $v_k\notin\Link(v)$ for all $k$, it follows that $j=i+1$. Hence, $w\in(\Gamma_{l_i}\setminus\Link(v))\cap(\Gamma_{l_{i+1}}\setminus\Link(v))=\{v\}\cap(\Gamma\setminus\{v\})=\emptyset$, a contradiction.

\vspace{0.2cm}

\noindent Let $l\geq n$ and $a=a_1\dots a_{n}$  is of the form described previously. We use the notations introduced at the beginning of the proof. If the word $\boldv$ associated to $a$ is reduced then $\cE(a)=0$. Hence, we will assume that $\boldv$ is not reduced. Then there exists $i<j$ such that $v^i_{s_i}=w=v^j_{s_j}$  and $v^k_s\in\Link(w)$ whenever:
\begin{enumerate}
\item $i<k<j$ and $1\leq s\leq r_k$,
\item $k=i$ and $s_i<s\leq r_i=r_k$,
\item $k=j$ and $1\leq s<s_j$.
\end{enumerate}
Since we can replace $\boldv$ by a type ${\rm II}$ equivalent word and since any subword $\boldv_k$ is reduced, we may and will assume that $j=i+1$ and $w=v^i_{r_i}=v^{i+1}_1$. Hence we have $w\in\Gamma_{l_i}\cap\Gamma_{l_j}=\Gamma_{l_i}\cap\Gamma_{l_{i+1}}=\Gamma_1\cap\Gamma_2=\Star(v)\cap\Gamma\setminus\{v\}=\Link(v)$. Write, for $x\in \bA_w$, $\mathcal{P}_w(x)=x-\omega_w(x)$. One has
\begin{eqnarray*}
\cE(a_1\dots a_n)&=&\cE(a_1\dots a_{i-1}x_1^i\dots x_{r_{i-1}}^i\cP_w(x^i_{r_i}x^{i+1}_1)x_2^{i+1}\dots x_{r_{i+1}}^{i+1}\dots a_{i+2}\dots a_n)\\
&&+\omega_w(x^i_{r_i}x^{i+1}_1)\cE(a_1\dots a_{i-1}x_1^i\dots x_{r_{i-1}}^ix_2^{i+1}\dots x_{r_{i+1}}^{i+1}\dots a_{i+2}\dots a_n).
\end{eqnarray*}

\noindent The right hand side of this expression is zero by the induction hypothesis.

\noindent
\end{proof}

\begin{rmk} Theorem \ref{Thm=Amalgam} is trivially true when we consider the maximal graph product and the maximal amalgamated free product.
\end{rmk}

\begin{cor}
$\bA_\Gamma$ is exact if and only if $\bA_v$ is exact for all $v\in V\Gamma$.
\end{cor}

\begin{proof}
By an inductive limit argument we may suppose the graph $\Gamma$ is finite.  We explain now this inductive limit argument which will be used several time in this paper (even in the von Neumann algebra context). Let $\mathcal{F}(\Gamma)$ the set of finite subgraphs of $\Gamma$ ordered by the inclusion. If $\mathcal{G}_1,\mathcal{G}_2\in\mathcal{F}(\Gamma)$ and $\mathcal{G}_1\subset\mathcal{G}_2$, we view $\bA_{\mathcal{G}_1}\subset\bA_{\mathcal{G}_2}\subset\bA_\Gamma$. Hence, we get an inductive system of unital C*-algebras $(\bA_{\mathcal{G}})_{\mathcal{G}\in\mathcal{F}(\Gamma)}$. Let
$\bA_\infty=\overline{\bigcup_{\mathcal{G}\in\mathcal{F}(\Gamma)}\bA_{\mathcal{G}}}\subset \bA_\Gamma$
be the inductive limit. We claim that actually $\bA_\infty=\bA_\Gamma$. Indeed, it is enough to show that every reduced operator $a=a_1\dots a_n\in \bA_\Gamma$, with associated word $\bold{v}=(v_1,\dots,v_n)$ lies in $\bA_\infty$. In fact, such an operator $a$ lies in $\bA_{\mathcal{G}}$, where $\mathcal{G}$ is a finite subgrah of $\Gamma$ containing the vertices $v_1,\dots v_n$.

\vspace{0.2cm}

\noindent So we may assume that $\Gamma$ is finite. Theorem \ref{Thm=Amalgam}   and the results of \cite{Dyk04} may then be used to reduce the corollary (by induction on the vertices) to the situation of a clique, i.e. a graph in which every two vertices share an edge. In the latter case the graph product is the minimal tensor product of C$^\ast$-algebras, which preserves exactness.
\end{proof}

\begin{rmk}
If $\Gamma$ has $n$ connected component $\Gamma_1,\dots,\Gamma_n$ then $(\bA_\Gamma,\omega_\Gamma)\simeq(\bA_{\Gamma_1}*\dots *\bA_{\Gamma_n},\omega_{\Gamma_1}*\dots *\omega_{\Gamma_n})$.
\end{rmk}

\subsection{The graph product of von Neumann algebras}

Suppose that, for each $v \in V\Gamma$, we have a von Neumann algebra $\bM_v$ with a \textit{faithful} normal state $\omega_v$. We may and will assume that $\bM_v\subset\cB(\cH_v)$, where $(\cH_v,\id,\xi_v)$ is a GNS-construction for $\omega_v$. Let $(\cH,\Omega)$ by the graph product of the pointed Hilbert spaces $(\cH_v,\omega_v)$. Recall that $\cH$ comes with faithful unital normal $*$-homomorphisms $\lambda_v\,:\,\cB(\cH_v)\rightarrow\cB(\cH)$.

\begin{dfn} The \textit{graph product von Neumann algebra} is $\bM_\Gamma:=\left(\bigcup_{v\in V\Gamma}\lambda_v(\bM_v)\right)''\subset\cB(\cH)$.
\end{dfn}

\noindent As before, we will assume that $\bM_v\subset\bM_\Gamma$ and $\lambda_v|_{\bM_v}$ is the inclusion, for all $v\in V\Gamma$. We also have the same notion of reduced operators and the linear span of $1$ and the reduced operators is a weakly dense $*$-subalgebra of $\bM_\Gamma$. The graph product state $\omega_\Gamma(\: \cdot \:) =\langle\: \cdot\: \Omega,\Omega\rangle$ is now a normal state on $\bM_\Gamma$. The graph product state is characterized as follows: it is the unique normal state on $\bM_\Gamma$ satisfying $\omega_\Gamma(a)=0$ for all reduced operators $a\in \bA_{\Gamma}$. In particular, $\omega_\Gamma|_{\bM_v}=\omega_v$ for all $v\in V\Gamma$.

\vspace{0.2cm}

\noindent Let us construct the right version of $\bM_\Gamma$. For $v\in V\Gamma$, we denote by $r_v(a)$
the right action of $\bM_v$ on $\cH_v$ i.e. $r_v(a)=J_va^*J_v$ where $J_v$ is the modular conjugation of $\omega_v$. View $r_v$ a faithful normal unital $*$-homomorphism from $\bM_v^{\text{op}}$ to $\cB(\cH_v)$. Denote by $\bM_\Gamma^r$ the von Neumann subalgebra of $\cB(\cH)$ generated by $\bigcup_{v\in V\Gamma}\rho_v\circ r_v(\bM_v)$. Write $\rho_v^\Gamma=\rho_v\circ r_v$ and note that $\rho_v^\Gamma$ is a faithful unital normal $*$-homomorphism from $\bM_v^{\text{op}}$ to $\bM^r_\Gamma$.

\vspace{0.2cm}

\noindent Observe that, by Proposition \ref{Prop=commutant}, $\bM_\Gamma^r\subset\bM_\Gamma'$.

\vspace{0.2cm}

\noindent As before, we call an operator $a=\rho^\Gamma_{v_1}(a_1)\dots\rho^\Gamma_{v_n}(a_n)\in \bM^r_\Gamma$ \textit{reduced} if $a_i\in \bM_{v_i}^{\circ}$ and the word $\boldv=(v_1,\dots,v_n)$ is reduced. It is clear from the definitions that, whenever $a=\rho^\Gamma_{v_1}(a_1)\dots\rho^\Gamma_{v_n}(a_n)\in \bM^r_\Gamma$ is a reduced operator (with associated word in $\mathcal{W}_{\text{min}}$) one has $a\Omega=\widehat{a}_n\otimes\dots\otimes\widehat{a}_1$. Hence, the vector $\Omega$ is cyclic for $\bM^r_\Gamma$ so it is separating for $\bM_\Gamma$ and the graph product state $\omega_\Gamma$ is faithful with GNS-construction $(\mathcal{H},\id,\Omega)$. It is now easy to compute the modular theory of $\omega_\Gamma$. We denote by $\nabla_v$, $J_v$ and $(\sigma^v_t)_{t\in\mathbb{R}}$ the ingredients of the modular theory of $\omega_v$, for $v\in V\Gamma$. For $\boldw \in\mathcal{W}$ a reduced word of the form $\boldw=(v_1,\dots,v_n)$, let $\overline{\boldw}$ be the unique minimal reduced word equivalent to the reduced word $\boldw^*=(v_n,\dots,v_1)$ and $\sigma_\boldw$ the unique bijection of $\{1,\dots,n\}$ such that $\overline{\boldw}=(v_{\sigma_\boldw(n)},\dots,v_{\sigma_\boldw(1)})$. Define the unitary operator $\Sigma_\boldw\,:\,\mathcal{H}_\boldw\rightarrow\mathcal{H}_{\overline{\boldw}}$ by $\Sigma_\boldw(\xi_1\ot\dots\ot\xi_n)=\cQ_{\boldw^*,\bar{\boldw}}(\xi_n\ot\dots\ot\xi_1)=\xi_{\sigma_\boldw(n)}\ot\dots\ot\xi_{\sigma_\boldw(1)}$. Finally, denote by $J_{\mathbb{C}}$ the conjugation map on $\mathbb{C}$.

\begin{prop}
Let $J$, $\nabla$ and $(\sigma_t)_{t\in\mathbb{R}}$ be the ingredients of the modular theory of $\omega_\Gamma$. One has
\begin{enumerate}
\item $J=J_{\mathbb{C}}\oplus\bigoplus_{\boldw =(v_1,\dots,v_n)\in\cW_{\text{min}}} (J_{v_{\sigma_\boldw(n)}}\ot\dots\ot J_{v_{\sigma_\boldw(n)}})\Sigma_\boldw$
\item $\nabla=\id_{\mathbb{C}\Omega}\oplus\bigoplus_{\boldw=(v_1,\dots,v_n)\in\cW_{\text{min}}} \Sigma_\boldw^*(\nabla_{v_{\sigma_\boldw(n)}}\ot\dots\ot \nabla_{v_{\sigma_\boldw(n)}})\Sigma_\boldw$
\item For any reduced operator $a=a_1\dots a_n\in\bM_\Gamma$ with associated word $\boldv=(v_1,\dots ,v_n)$ one has
$$\sigma_t(a_1\dots a_n)=\sigma_t^{v_1}(a_1)\dots\sigma_t^{v_n}(a_n)\quad\text{for all}\quad t\in\mathbb{R}.$$
\end{enumerate}
\end{prop}

\begin{proof}
$(3)$ follows easily from $(2)$. Let $S_v$ (resp. $S$) be the modular operator for $\omega_v$ (resp. $\omega_\Gamma$). To get $(1)$ and $(2)$, it suffices to prove, by uniqueness of the polar decomposition, that
$$S=\id_{\mathbb{C}\Omega}\oplus\bigoplus_{\boldw=(v_1,\dots,v_n)\in\cW_{\text{min}}} (S_{v_{\sigma_\boldw(n)}}\ot\dots\ot S_{v_{\sigma_\boldw(n)}})\circ\Sigma_\boldw.$$
Denote by $T$ the right hand side of the preceding equation. An easy computation gives that, for all reduced operators $a=a_1\dots a_n\in \bA_\Gamma$ or for $a\in\mathbb{C}1$, one has $S(a\Omega)=T(a\Omega)$. Hence, $S|_{\mathcal{M}_\Gamma\Omega}=T|_{\mathcal{M}_\Gamma\Omega}$, where $\mathcal{M}_\Gamma\subset\bM_\Gamma$ is the linear span of $1$ and the reduced operators and it suffices to show that $\mathcal{M}_\Gamma\Omega$ is a common core for $S$ and $T$. By definition, $\bM_{\Gamma}\Omega$ is a core for $S$. Since $\mathcal{M}_\Gamma$ is a weakly dense unital $*$-subalgebra of $\bM_\Gamma$, it follows from the Kaplansky's density Theorem that $\mathcal{M}_\Gamma\Omega$ is also a core for $S$: indeed, let $x \in \bM_\Gamma$. By Kaplansky's density theorem $x$ is in the strong closure of the convex set of elements in $\mathcal{M}_\Gamma$ that have norm at most $\Vert x \Vert$. For convex sets the strong and strong-$\ast$ closure coincide. Therefore, we may find a net $(x_i)_i$  in $\mathcal{M}_\Gamma$ converging to $x$ in the strong-$\ast$ topology. It follows that $x_i \Omega$ is bounded and converges weakly to $x \Omega$ and similarly $x_i^\ast \Omega$ is bounded and converges weakly to $x^\ast \Omega$ in the GNS-Hilbert space. This concludes that $\mathcal{M}_\Gamma$ is a core for $S$ with respect to the weak/weak-topology on the graph of $S$. Hence by a standard convexity argument it is also a norm/norm core for $S$.

By definition of $T$, a core for $T$ is given by the subspace
$$\mathbb{C}\Omega\oplus\bigoplus_{\boldw=(v_1,\dots,v_n)\in\cW_{\text{min}}}\bM_{v_1}^{\circ}\xi_{v_1}\otimes\dots\otimes \bM_{v_n}^{\circ}\xi_{v_n},$$
where the direct sums and tensor products are the algebraic ones. This subspace is exactly the linear span of $\Omega$ and vectors of the form $a\Omega$, where $a$ is a reduced operator i.e. this is the space $\mathcal{M}_\Gamma\Omega$.
\end{proof}

\begin{rmk} It follows from the preceding proposition that, for all reduced operators $a=a_1\dots a_n\in\bM_\Gamma$, with $a_i\in\bM_{v_i}$, one has $JaJ=\rho_{v_1}^\Gamma(a_1)\dots\rho_{v_n}^\Gamma(a_n)$. Hence we actually have $\bM_\Gamma'=\bM_\Gamma^r$.
\end{rmk}

\noindent The graph product von Neumann algebra also satisfies a universal property. The following Proposition \ref{Prop=UniversalPropertyvN} can be proved exactly as Proposition \ref{Prop=commutation} since the isomorphism appearing in the proof of Proposition \ref{Prop=UniversalPropertyReduced} is spacial.

\begin{prop}\label{Prop=UniversalPropertyvN}
Let $\bN$ be a von Neumann algebra with a GNS-faithful normal state $\omega$ and suppose that, for all $v\in V\Gamma$, there exists a unital normal faithful $*$-homomorphism $\pi_v\,:\,\bM_v\rightarrow \bN$ such that:
\begin{itemize}
\item $\bN$ is generated, as a von Neumann algebra, by $\cup_{v\in V\Gamma}\pi_v(\bM_v)$ and the images of $\pi_v$ and $\pi_{v'}$ commute whenever $(v,v')\in E\Gamma$.
\item For any operator $a=\pi_{v_1}(a_1)\dots\pi_{v_n}(a_n)\in \bN$, where $\boldv=(v_1,\dots v_n)$ is a reduced word and $a_i\in \bM_{v_i}^{\circ}$ one has $\omega(a)=0$
\end{itemize}
Then, there exists a unique normal $*$-isomorphism $\pi\,:\,\bM_\Gamma\rightarrow \bN$ such that $\pi|_{\bM_v}=\pi_v$. Moreover, $\pi$ intertwines the graph product state and $\omega$. In particular, $\omega$ is faithful.
\end{prop}

\begin{rmk}\label{Rmk=IsoVn}
The preceding proposition implies the following.
\begin{itemize}
\item If $\bM_v = {\rm L}(G_v)$ is the group von Neumann algebra of a discrete group $G_v, v \in V\Gamma$ then $(\bM_\Gamma,\omega_\Gamma) = ({\rm L}(G_\Gamma),\tau)$, where $\tau$ is the canonical trace.
\item Let $\Gamma$ be a graph for which any two vertices are connected by one edge. Then $(\bM_\Gamma,\omega_\Gamma)= \underset{v\in V\Gamma}{\ot}(\bM_{v},\omega_{v})$.
\item Let $\Gamma$ be a graph with no edges. Then $(\bM_{\Gamma},\omega_\Gamma)=\underset{v\in V\Gamma}{*}(\bM_v,\omega_v)$.
\item If $\Gamma_0\subset\Gamma$ is a subgraph and, for all $v\in\Gamma_0$, $\bN_v\subset \bM_v$ is a unital von Neumann subalgebra then the graph product von Neumann algebra $\bN_{\Gamma_0}$ obtained from the $\bN_v$, $v\in\Gamma_0$, is canonically isomorphic to $\left(\bigcup_{v\in V\Gamma_0} \bN_v\right)''$. In the sequel we will always do this identification without further explanations.
\item There is a unique (state preserving) $*$-isomorphism $\bM_{\Star(v)}\simeq \bM_v\ot \bM_{\Link(v)}$ identifying $x\ot y$ with $xy$, for all $x\in \bM_v$ and all $y\in \bM_{\Link(v)}$. In particular, $\bM_v'\cap \bM_{\Star(v)}=\bN_{\Star(v)}$, where
$$ \bN_w=\left\{\begin{array}{lcl}
\bM_w&\text{if}&w\in\Link(v),\\
Z(\bM_v)&\text{if}&w=v.
\end{array}\right.$$
\end{itemize}
\end{rmk}

\begin{rmk}\label{Rmk=CEVN}
Let $\Gamma_0 \subseteq \Gamma$ be a subgraph and consider the graph product von Neumann algebras $\bM_{\Gamma_0}$ and $\bM_{\Gamma}$. As in a C*-algebraic case, there exists a unique normal conditional expectation $\cE_{\Gamma_0}$ from $\bM_\Gamma$ to $\bM_{\Gamma_0}$ preserving the graph product states and such that $\cE_{\Gamma_0}(a)=0$ for all reduced operator $a\in \bM_{\Gamma_0}$ with associated reduced word $\boldv=(v_1,\dots v_n)$ satisfying the property that one of the $v_i$ is not in $\Gamma_0$. In particular, for all $v\in V\Gamma$, there exists a unique state preserving normal conditional expectation $\cE_v\,:\,\bM_\Gamma\rightarrow\bM_v$ such that $\cE_v(a)=0$ for all reduced operator $a\in \bM_\Gamma\setminus \bM_v$.
\end{rmk}

\begin{prop}\label{Prop=intersection}
Let $\Gamma_0,\Gamma_1\subset\Gamma$ be subgraphs. One has $\bM_{\Gamma_0}\cap\bM_{\Gamma_1}=\bM_{\Gamma_0\cap\Gamma_1}$.
\end{prop}

\begin{proof}
The inclusion $\bM_{\Gamma_0\cap\Gamma_1}\subset\bM_{\Gamma_0}\cap\bM_{\Gamma_1}$ being obvious, let us show the other one. Let $\mathcal{M}_{\Gamma_0}$ be the linear span of $1$ and the reduced operator in $\bM_{\Gamma_0}$. It suffices to show that $\mathcal{M}_{\Gamma_0}\cap \bM_{\Gamma_1}\subset \bM_{\Gamma_0\cap\Gamma_1}$. Indeed, if it is the case, then $ \bM_{\Gamma_0\cap\Gamma_1}$ contains $\left(\mathcal{M}_{\Gamma_0}\cap \bM_{\Gamma_1}\right)''=\left(\mathcal{M}_{\Gamma_0}'\cup \bM_{\Gamma_1}'\right)'=\mathcal{M}_{\Gamma_0}''\cap \bM_{\Gamma_1}=\bM_{\Gamma_0}\cap\bM_{\Gamma_1}$. Let $x\in \mathcal{M}_{\Gamma_0}$ and write $x=\omega_{\Gamma}(x)1+\sum_{i}x_i$, where the sum is finite and the $x_i$ are some reduced operators in $\bM_{\Gamma_0}$. If $x\in\bM_{\Gamma_1}$ we have $x=\cE_{\Gamma_1}(x)=\omega_{\Gamma}(x)1+\sum_{i}\cE_{\Gamma_1}(x_i)$. By definition of the conditional expectation, for all $i$, $\cE_{\Gamma_1}(x_i)$ is either $0$ or a reduced operator with associated vertices in $\Gamma_0\cap\Gamma_1$. Hence, $x\in\bM_{\Gamma_0\cap\Gamma_1}$.
\end{proof}

\noindent Let $v \in \Gamma$, $\Gamma_1=\Star(v)$, $\Gamma_2=\Gamma\setminus\{v\}$ and set the following graph product von Neumann algebras: $\bM_1 = \bM_{\Gamma_1}$, $\bN = \bM_{\Link(v)}$, and $\bM_2 = \bM_{\Gamma_2}$. By the universal property of Proposition \ref{Prop=UniversalPropertyvN}, we may view $\bN\subset\bM_1\subset\bM_\Gamma$ and $\bN\subset\bM_2\subset\bM_\Gamma$ canonically. Let us denote by $\bM_1 \star_\bN \bM_2$ the von Neumann algebraic amalgamated free product with respect to the graph product states. The following result can be proved exactly as Theorem \ref{Thm=Amalgam}, using the universal property of von Neumann algebraic amalgamated free products.

\begin{thm}\label{Thm=AmalgamvN}
There exists a unique $\ast$-isomorphism $\pi:  \bM_1 \star_\bN \bM_2 \rightarrow \bM_\Gamma $ such that $\pi|_{\bM_1}$ (resp. $\pi|_{\bM_2}$) is the canonical inclusion $\bM_1\subset \bM_\Gamma$ (resp. $\bM_2\subset \bM_\Gamma$). Moreover, $\pi$ is state-preserving.
\end{thm}

\noindent Before the next lemma, let us recall some standard notations. Let $(\bM,\tau)$ be a finite von Neumann algebra and $\bA,\bB\subset\bM$ two unital von Neumann subalgebras. We write $\bA\underset{\bM}{\nprec} \bB$ if there exists a net $(u_i)_i$ of unitaries in $\bA$ such that  $\Vert \mathcal{E}_\bB(xu_iy) \Vert_2 \rightarrow 0$ for all $x,y\in \bM$. We also write $\mathcal{N}_\bM(\bA)$ the \textit{normalizer} of $\bA$ in $\bM$ i.e.
$$\mathcal{N}_\bM(\bA)=\{u\in\mathcal{U}(\bM)\,:\,u\bA u^*=\bA\}.$$
A von Neumann algebra is called {\it diffuse} if it does not contain minimal projections. In particular, type ${\rm II}$-factors are diffuse von Neumann algebras.

\begin{lem}\label{Lem=IPP1}
Suppose that $\omega_v$ is a trace for all $v\in V\Gamma$. Fix $v\in V\Gamma$.  If $\bQ \subset \bM_v$ is a diffuse von Neumann subalgebra then
\[
\bQ \underset{\bM_{\Star(v)}}{\nprec} \bM_{\Link(v)}
 \]
and any $\bQ$-$\bM_{\Star(v)}$-sub-bimodule of ${\rm L}^2(\bM_{\Gamma})$ which has finite dimension as right $\bM_{\Star(v)}$-module is contained in $L^2(\bM_{\Star(v)})$. In particular, $\bQ'\cap \bM_\Gamma\subset\mathcal{N}_{\bM_\Gamma}(\bQ)''\subset \bM_{\Star(v)}$.
\end{lem}

\begin{proof}
Since $\bM_\Gamma=\bM_{\Star(v)}\underset{ \bM_{\Link(v)}}{*} \bM_{\Gamma\setminus\{v\}}$, we may apply \cite[Theorem1.1]{IPP08} (which needs $\bQ$ to be diffuse) to conclude that the last statement of the lemma follows from the fact that $\bQ \underset{\bM_{\Star(v)}}{\nprec} \bM_{\Link(v)}$ which is obvious since, by the last point of Remark \ref{Rmk=IsoVn}, we have $\bM_{\Star(v)}\simeq \bM_v\ot \bM_{\Link(v)}$.

\end{proof}

\begin{cor}\label{Cor=RelativeCommutant}
Suppose that $\omega_v$ is a trace for all $v\in V\Gamma$. Fix $v\in V\Gamma$. For $w\in\Star(v)$ define
$$\bN_w=\left\{\begin{array}{lcl}
\bM_w&\text{if}&w\in\Link(v),\\
Z(\bM_v)&\text{if}&w=v.
\end{array}\right.$$
If $\bM_v$ is diffuse then $\bM_v'\cap \bM_\Gamma= \bN_{\Star(v)}$ (here $\bN_{\Star(v)}$ is the graph product of $\bN_v, v \in \Star(v)$ with respect to the graph $\Star(v)$). In particular, $Z(\bM_\Gamma)=\bigcap_{v\in V\Gamma} \bN_{\Star(v)}$.
\end{cor}

\begin{proof}
The inclusion $\bN_{\Star(v)}\subset \bM_v'\cap \bM_\Gamma$ being obvious, let us prove the other inclusion. By Lemma \ref{Lem=IPP1} and the last assertion of Remark \ref{Rmk=IsoVn} we have, $\bM_v'\cap \bM_\Gamma\subset \bM_v'\cap \bM_{\Star(v)}= \bN_{\Star(v)}$.
\end{proof}

\begin{cor}
If $\bM_v$ is a ${\rm II}_1$-factor for all $v\in V\Gamma$ then $\bM_\Gamma$ is a ${\rm II}_1$-factor.
\end{cor}

\begin{proof}
By the inductive limit argument we may and will assume that $\Gamma$ is finite graph. By Corollary \ref{Cor=RelativeCommutant} we find $Z(\bM_\Gamma)=\bigcap_{v\in V\Gamma} \bM_{\Link(v)}$. It follows from Proposition \ref{Prop=intersection} that  $Z(\bM_\Gamma)=\bM_{\bigcap_{v\in V\Gamma}\Link(v)}$. Since $\bigcap_{v\in V\Gamma}\Link(v)\subset\bigcap_{v\in V\Gamma}\Gamma\setminus\{v\}=\emptyset$ we conclude that $Z(\bM_\Gamma)=\mathbb{C}1$.
\end{proof}

\subsubsection{Completely positive maps of graph products}

Let $(\bM_v,\omega_v)_{v\in V\Gamma}$ and $(\bN_v,\mu_v)_{v\in V\Gamma}$ be two families of von Neumann algebras with faithful normal states.

\begin{prop}\label{Prop=ucp}
For all $v\in V\Gamma$, let $\varphi_v\,:\, \bM_v\rightarrow \bN_v$ be a state-preserving normal ucp map. Then, there exists a unique normal ucp map $\varphi\,:\, \bM_\Gamma\rightarrow \bN_\Gamma$ such that, for all $a=a_1\dots a_n\in \bM_\Gamma$ reduced, with $a_k\in \bM_{v_k}^{\circ}$,
$$\varphi(a_1\dots a_n)=\varphi_{v_1}(a_1)\dots\varphi_{v_n}(a_n).$$
Moreover, $\varphi$ intertwines the graph product states and its ${\rm L}^2$-extension is given by
$$T_\varphi\,:\,\bc\Omega\oplus\bigoplus_{\boldw =(v_1,\dots,v_n)\in\cW_{\text{min}}}{\rm L}^2(\bM_{v_1})^{\circ}\ot\dots\ot{\rm L}^2(\bM_{v_n})^{\circ}\rightarrow \bc\Omega\oplus\bigoplus_{\boldw=(v_1,\dots,v_n)\in\cW_{\text{min}}}{\rm L}^2(\bN_{v_1})^{\circ}\ot\dots\ot{\rm L}^2(\bN_{v_n})^{\circ},$$
$$T_\varphi=\id_{\bc\Omega}\oplus\bigoplus T_{\varphi_{v_1}}|_{{\rm L}^2(\bM_{v_1})^{\circ}}\ot\dots\ot T_{\varphi_{v_n}}|_{{\rm L}^2(\bM_{v_n})^{\circ}}.$$
\end{prop}

\begin{proof}
Let $(\cK_v,\eta_v)$ be the pointed $\bM_v$-$\bN_v$ bimodule obtained from the GNS-construction of $\varphi_v$ i.e. one has $\cK_v=\overline{\bM_v\eta_v \bN_v}$ and $\langle a\eta_vb,\eta_v\rangle=\mu_v(\varphi_v(a)b)$. Denote by $\pi_v^l$ (resp. $\pi_v^r$) the left (resp. right) action of $\bM_v$ on $\cK_v$. Observe that, since $\mu_v$ is faithful, the map $\pi_v^r$ is faithful and, since $\omega_v$ is faithful and $\varphi_v$ preserves the states, the maps $\pi_v^l$ is also faithful. Let $(\cK,\eta)$ be the graph product of the pointed Hilbert spaces $(\cK_v,\eta_v)$ (see Section \ref{Section=GraphProductHilbert}) with the representations $\lambda_v,\rho_v\,:\,\cB(\cK_v)\rightarrow\cB(\cK)$ and define $\widetilde{\pi}_v^l=\lambda_v\circ\pi_v^l$ and $\widetilde{\pi}_v^r=\rho_v\circ\pi_v^r$.

\vspace{0.2cm}

\noindent Let $\mathcal{M}$ (resp. $\mathcal{N}$) be the von Neumann algebra subalgebra of $\mathcal{B}(\cK)$ generated by $\cup_v\widetilde{\pi}_v^l(\bM_v)$ (resp. $\cup_v\widetilde{\pi}_v^r(\bN_v)$). Consider the vector state $\mu(x)=\langle x\eta,\eta\rangle$ on $\mathcal{M}$ and $\mathcal{N}$. Observe that, for all $a=a_1\dots a_n\in \bM_\Gamma$ reduced, with associated word $\bold{v}=(v_1,\dots,v_n)$, Proposition \ref{Prop=commutation} implies $\widetilde{\pi}_{v_1}^l(a_1)\dots\widetilde{\pi}_{v_n}^l(a_n)\eta=a_1\eta_{v_1}\otimes\dots\otimes a_n\eta_{v_n}$. Also, for all $b=b_1\dots b_n\in \bN_\Gamma$ reduced, with associated word $\bold{v}=(v_1,\dots,v_n)$, Proposition \ref{Prop=commutationright} implies $\widetilde{\pi}_{v_1}^r(b_1)\dots\widetilde{\pi}_{v_n}^r(b_n)\eta=\eta_{v_n}b_n\otimes\dots\otimes \eta_{v_1}b_1$. It follows that $\mu(a)=\mu(b)=0$ for all $a\in\bM_\Gamma$ and $b\in\bN_\Gamma$ reduced. Moreover, $(\overline{\mathcal{M}\eta},\pi,\eta)$ (resp. $(\overline{\mathcal{N}\eta},\rho,\eta)$) is a GNS-construction for $\mu$ on $\mathcal{M}$ (resp. on $\mathcal{N}$), where $\pi(x)$ is the restriction of $x$ to the subspace $\overline{\mathcal{M}\eta}$ (resp. $\rho(y)$ is the restriction of $y$ to the subspace $\overline{\mathcal{N}\eta}$). By Proposition \ref{Prop=commutant} the images of $\widetilde{\pi}_v^l$ and $\widetilde{\pi}_{v'}^r$ commute for all $v,v'\in V\Gamma$. Hence, $\mathcal{N}\subset\mathcal{M}'$ and, by the preceding computations, we find $\overline{\mathcal{N}\mathcal{M}\eta}=\overline{\mathcal{M}\mathcal{N}\eta}=\mathcal{K}$. It follows that $\mu$ is GNS-faithful on $\mathcal{M}$ (resp. on $\mathcal{N}$).

\vspace{0.2cm}

\noindent By Proposition \ref{Prop=UniversalPropertyvN}, there exists two unital normal $\ast$-homomorphisms $\widetilde{\pi}_l\,:\,\bM_\Gamma\rightarrow\cB(\cK)$ and $\widetilde{\pi}_r\,:\,\bN_\Gamma^{\text{op}}\rightarrow\cB(\cK)$ such that $\widetilde{\pi}_l|_{\bM_v}=\widetilde{\pi}_v^l$ and $\widetilde{\pi}_r|_{\bN_v}=\widetilde{\pi}_v^r$. It is easy to check that the images of $\widetilde{\pi}_l$ and $\widetilde{\pi}_r$ commute. Hence, $\cK$ is a $\bM_\Gamma$-$\bN_\Gamma$ bimodule and we will simply write $a\xi b$ for the element $\widetilde{\pi}_l(a)\widetilde{\pi}_r(b)\xi$, for $a\in \bM_\Gamma$, $b\in \bN_\Gamma$ and $\xi\in \cK$. Define $V\,:\,{\rm L}^2(\bN_\Gamma)\rightarrow \cK$ by $V\widehat{x}=\eta.x$. One can easily check that, for all $a=a_1\dots a_n\in \bM_\Gamma$ reduced, with $a_k\in \bM_{v_k}^{\circ}$,
$$V^*\widetilde{\pi}_l(a_1\dots a_n)V=\varphi_{v_1}(a_1)\dots\varphi_{v_n}(a_n).$$
The fact that $\varphi$ intertwines the graph product states and the ${\rm L}^2$-extension formula are obvious from the formula defining $\varphi$.
\end{proof}

\noindent The ucp map obtained in Proposition \ref{Prop=ucp} is called the \textit{graph product ucp map}, it generalizes Boca's construction of free product of ucp maps \cite{Boca}. As a consequence we are able to show that the graph product preserves the Haagerup property (see \cite{BocaHAP}, \cite{CasSka2} for free products of respectively finite and $\sigma$-finite von Neumann algebras). Recall the following definition from \cite{CasSka}. We refer to \cite{OkaTom} and \cite{COST} for alternative (but equivalent) approaches to the Haagerup property and to \cite{Cho}, \cite{Jol} for the case of a finite von Neumann algebra.

\begin{dfn}\label{Dfn=HAP}
A pair $(\bM, \omega)$ of a von Neumann algebra $\bM$ with normal, faithful state $\omega$ has the Haagerup property if there exists a net $\{ \varphi_i\}_{i \in I}$ of cp maps $\varphi_i: \bM \rightarrow \bM$ such that $\omega \circ \varphi_i \leq \omega$ and such that the GNS-maps $T_i: x \Omega_\omega \mapsto \varphi_i(x) \Omega_\omega$ extend to compact operators  converging to 1 strongly.
\end{dfn}

\begin{rmk}\label{Rmk=LowerThan1}
In \cite{CasSka2} it was proved that if a pair $(\bM, \omega)$ has the Haagerup property, then the cp maps $\varphi_i$ can be chosen unital and such that $\omega \circ \varphi_i = \omega$. Let $(\cH_\omega, \Omega_\omega)$ be the GNS-space with cyclic vector $\Omega_\omega$ and $\cH_\omega^\circ$ the space orthogonal to $\Omega_\omega$. Define $\varphi'_i(x) =  \frac{1}{1+\epsilon}( \varphi_i(x) + \epsilon \:\omega(x)), x \in \bM$ and let $T_i'$ be its GNS-map $\cH_\omega \rightarrow \cH_\omega$ determined by $x \Omega_\omega \mapsto \varphi_i'(x) \Omega_\omega$. The restriction of $T_i'$ to the space $\cH_\omega^\circ$ has norm less than $\frac{1}{1+\epsilon} \Vert T_i \Vert$. Letting $\epsilon \rightarrow 0$, this shows that we may always assume $\Vert T_i\vert_{\cH_\omega^\circ} \Vert < 1$ in Definition \ref{Dfn=HAP}.
\end{rmk}

\begin{cor}\label{Cor=HAPFiniteKac}
$\bM_\Gamma$ has the Haagerup property if and only if $\bM_v$ has the Haagerup property for all $v\in V\Gamma$.
\end{cor}
\begin{proof}
By the inductive limit argument, we may and will assume that the graph $\Gamma$ is finite (see \cite[Theorem 2.3]{Jol}). Suppose that $\bM_v$ has the Haagerup property for all $v\in V\Gamma$. Let $\varphi_{v,i_v}\,:\, \bM_v\rightarrow \bM_v$ be a net of state-preserving ucp maps with compact ${\rm L}^2$-implementation $T_{v,i_v}$  and such that $||\varphi_{v,i_v}(a)-a||_2\rightarrow 0$ for all $a\in \bM_v$ and all $v\in V\Gamma$.   By Remark \ref{Rmk=LowerThan1} we may assume that $\Vert T_{v,i_v} \vert_{\cH_v^\circ} \Vert < 1$. Define the net of ucp map $\varphi_i\,:\, \bM_\Gamma\rightarrow \bM_\Gamma$, each $\varphi_i$ is the graph product of the $\varphi_{v,i_v}$, $v\in V\Gamma$ and the net structure for $\varphi_i$ is given by the product of the nets for $\varphi_{v, i_v}$. Since  $\Vert T_{v,i_v} \vert_{\cH_v^\circ} \Vert < 1$ and $\Gamma$ is finite it follows that the $L^2$-implementation of $\varphi_i$ is compact. Also, $||\varphi_{i}(a)-a||_2\rightarrow 0$ for all reduced operators $a\in \bM_\Gamma$. Since the linear span of $1$ and the reduced operators is weakly dense in $\bM_\Gamma$, the convergence holds for all $a\in \bM_\Gamma$, from which one easily deduces that $\bM_\Gamma$ has the Haagerup property. The other implication is a obvious because of the existence of conditional expectations.
\end{proof}

\section{Graph products of discrete quantum groups}\label{Sect=QG}

\noindent In this paper we need compact and discrete quantum groups both in the C$^\ast$-algebraic and von Neumann algebraic framework. We recall their preliminaries here. We define graph products of quantum groups and give their basic properties.

\subsection{C$^\ast$-algebraic compact/discrete quantum groups} We write  $\overline{{\rm Span}}$ for the closed linear span.

\begin{dfn}[Wornonowicz \cite{WorQG}]\label{Dfn=QG}
A compact quantum group $\bG$ is a pair $(\bA, \Delta)$ of a unital C$^\ast$-algebra $\bA$ together with a comultiplication $\Delta: \bA \rightarrow \bA \otimes \bA$ which is a unital $\ast$-homomorphism such that $(\Delta \otimes \id) \circ \Delta = (\id \otimes \Delta) \otimes \Delta$ and such that the following cancellation laws hold:
\[
\overline{{\rm Span}} \:  \Delta(\bA) (\bA \otimes 1)  = \overline{{\rm Span}} \: (1 \otimes \bA) \Delta(\bA) = \bA \otimes \bA.
\]
\end{dfn}

\noindent Any compact quantum group $\bG$ admits a unique state $\omega$ on $\bA$ that satisfies $(\omega \otimes \id) \circ \Delta(x) = \omega(x) 1_\bA = (\id \otimes \omega) \circ \Delta(x)$. $\omega$ is called the Haar state. The GNS-space with respect to $\omega$ shall be denoted by $\cH$.

\vspace{0.2cm}

\noindent Let $\bG = (\bA, \Delta)$ be a compact quantum group. A (finite dimensional) unitary representation is a unitary operator $u \in \bA \otimes M_n$ such that $(\Delta \otimes \id)(u) = u_{13} u_{23}$ where $u_{23} = 1 \otimes u$ and $u_{13} = (\Sigma \otimes \id) (u_{23})$ with $\Sigma: \bA \otimes \bA \rightarrow \bA \otimes \bA$ the flip map. We denote by $\Irr(\bG)$ the equivalence classes of irreducible representations of $\bG$ and, for $\alpha \in \Irr(\bG)$, we choose a representative $u^\alpha$ of the class $\alpha$. Note that in the literature our notion of representation is often also called a corepresentation. We use $n_\alpha$ for the dimension of $u^\alpha$, i.e. $u^\alpha \in \bA \otimes M_{n_\alpha}$. We shall write $u^\alpha_{i,j}$ for the matrix coefficient $(\id \otimes \omega_{e_i, e_j})(u^\alpha) \in \bA$ in case $e_i, 1 \leq i \leq n_\alpha$ is an orthonormal basis of $\mathbb{C}^{n_\alpha}$. The tensor product representation is defined as $u^{\alpha} \otimes u^{\beta} := u^{\alpha}_{12} u^{\beta}_{13} \in \bA \otimes M_{n_\alpha} \otimes M_{n_\beta}$.

\vspace{0.2cm}

\noindent We set $\Pol(\bG) \subseteq \bA$ for the space of matrix coefficients of finite dimensional representations of $\bG$. It is well-known that $\Pol(\bG)$ is a $\ast$-algebra. Let $\gamma \in \Irr(\bG)$ then we denote $p_\gamma \in \mathcal{B}(\cH)$ for the projection onto the closed linear span of the coefficients of $u^\gamma$ identified within $\cH$.

\vspace{0.2cm}

\noindent Let $\widehat{\bG}$ denote the discrete dual quantum group of $\bG$. Typically we will write $\widehat{\bG} = (\widehat{\bA}, \widehat{\Delta})$. We have $\widehat{\bA} = \oplus_{\alpha \in \Irr(\bG)} M_{n_\alpha}$  (C$^\ast$-algebraic direct sum). Let $\widehat{\epsilon}$ be the counit of $\widehat{\bG}$. It is the unique non-degenerate $\ast$-homomorphism $\widehat{\bA} \rightarrow \mathbb{C}$ satisfying $(\widehat{\epsilon} \otimes \id) \circ \widehat{\Delta}=(\id \otimes \widehat{\epsilon}) \circ \widehat{\Delta}=\id$.

\vspace{0.2cm}

\noindent Every compact quantum group comes with a maximal (= universal) and a reduced version and we shall from this point fix a compact quantum group $\bG$ and let $(\bA, \Delta)$ denote the associated reduced (compact) quantum group and let $(\bA_m, \Delta_m)$ denote the associated maximal (compact) quantum group. There exists a canonical surjection $\nu: \bA_m \rightarrow \bA$ that preserves the comultiplication. We refer to \cite{KusUni} for the definition of maximal (= universal) quantum groups. There is no distinction between maximal and reduced versions of $\widehat{\bG} = (\widehat{\bA}, \widehat{\Delta})$ since for a discrete quantum group these always agree.

\begin{rmk}
If $\bG_1= (\bA_1, \Delta_1)$ and $\bG_2= (\bA_2, \Delta_2)$ are compact quantum groups then $\bG_1 \times \bG_2$ is the quantum group whose C$^\ast$-algebra is given by $\bA_1 \otimes \bA_2$ and with comultiplication $\Delta = (\id \otimes \Sigma \otimes \id) \circ \Delta_1 \otimes \Delta_2$, where $\Sigma: \bA_1 \otimes \bA_2 \rightarrow \bA_2 \otimes \bA_1$ is the flip map.
\end{rmk}

\subsection{Von Neumann algebraic quantum groups} Let $\bG$ be a compact quantum group. Let $\bM$ be the von Neumann algebra generated by $\bA$ in the GNS-construction of the Haar state $\omega$. The comultiplication $\Delta: \bA \rightarrow \bA \otimes \bA$ lifts uniquely to a unital, normal $\ast$-homomorphism $\bM \rightarrow \bM \otimes \bM$  which we keep denoting by $\Delta$. Also, the Haar state $\omega$ extends to a normal state $\omega$ on $\bM$. Then $(\bM, \Delta)$ forms a von Neumann algebraic locally compact quantum group in the sense of \cite{KusVaeVNA} with $\omega$ as left and right invariant weight.

\vspace{0.2cm}

\noindent We say that $\bG$ is of Kac type if $\omega$ is tracial. If $\bG$ is of Kac type then also $\widehat{\bG}$ has a tracial Haar weight. If $\bG$ is of Kac type then there exists a $\ast$-antihomomorphism $\kappa: \bM \rightarrow \bM$ called  the antipode and which satisfies $\kappa(u^\alpha_{i,j}) = u^\alpha_{j,i}$. We let $\widehat{\kappa}: \widehat{\bM} \rightarrow \widehat{\bM}$ be the dual antipode. It may be characterized by $(\kappa \otimes \widehat{\kappa})(W) = W$ where $W$ is the left multiplicative unitary from \cite{KusVaeVNA}  (one may define $\widehat{\kappa}$ in other ways; the definition given here is not the usual one).

\subsection{Graph products, their representation theory and Haar state}

For all $v\in V\Gamma$, let $\bG_v$ be a compact quantum group with full C*-algebra $\bA_{v,m}$, reduced C$^\ast$-algebra $\bA_{v}$, von Neumann algebra $\bM_v$, Haar state $\omega_v$ and comultiplication $\Delta_v$ (on any of these algebras). Let $\bA_m (= \bA_{\Gamma,m})$ be the maximal graph product C*-algebra associated to the family of C*-algebras $(\bA_{v,m})_{v\in V\Gamma}$. Since $\omega_v$ is faithful (resp. normal and faithful) on $\bA_v$ (resp. on $\bM_v$), we can also consider the reduced graph product C*-algebra $\bA\: (=\bA_\Gamma)$ associated to the family $(\bA_v,\omega_v)_{v \in V\Gamma}$ and the graph product von Neumann algebra $\bM \:(=\bM_\Gamma)$ associated to the family $(\bM_v,\omega_v)_{v \in V\Gamma}$.

\vspace{0.2cm}

\noindent By the universal property of $\bA_m$, there exists a unique unital $*$-homomorphism $\Delta\,:\,\bA_m\rightarrow\bA_m\ot\bA_m$ such that $\Delta|_{\bA_v}=\Delta_v$ for all $v\in V\Gamma$. From \cite[Definition 2.1']{WanCMP} we can show that $\bG=(\bA_m,\Delta)$ is a compact quantum group. Indeed, for all $v\in V\Gamma$, the inclusion $\bA_{v,m} \subset \bA_m$ intertwines the comultiplication, it induces an inclusion $\Irr(\bG_v)\subset \Irr(\bG)$. Since the matrix coefficients of $\Irr(\bG_v)$ generate $\bA_{m}$ as a $\ast$-algebra this shows that the conditions of \cite[Definition 2.1']{WanCMP} are satisfied and hence $\bG$ is a compact quantum group.

\vspace{0.2cm}

\noindent Note that it is at this point not clear that $(\bA_m, \Delta)$ is the underlying universal quantum group of $\bG$ in the sense of \cite{KusUni}. In fact this is true as follows from Theorem \ref{Thm=GraphCoreps} below. We shall also prove that $\bM$ and $\bA$ are the algebras of the underlying von Neumann and reduced C$^\ast$-algebraic quantum group.  In order to distinguish notation we shall -- only in this section -- write $C_m(\bG), C_r(\bG)$ and ${\rm L}^\infty(\bG)$ for the full and reduced C$^\ast$-algebra associated with $\bG$ as well as its von Neumann algebra. Also write $\nu_\bG: C_m(\bG) \rightarrow C_r(\bG)$ for the canonical surjection and $L^2(\bG)$ for the GNS-space of $\bG$.

\begin{dfn}\label{Dfn=ReducedRep}
A unitary representation $u$ of $\bG$ is said to be \textit{reduced} if it is of the form $u=u^{\alpha_1}\ot\ldots\ot u^{\alpha_n}$, where $n\geq 1$, $\boldv=(v_1,\dots,v_n)$ is a reduced word and $\alpha_k\in\Irr(\bG_{v_k})\setminus\{1\}$ for all $1\leq k\leq n$.
\end{dfn}

\noindent Let $\nu_v: \bA_{v,m} \rightarrow \bA_v$ be the canonical surjection. By the universal property of $\bA_m$, we have a unique surjective and unital $\ast$-homomorphism $\nu \,:\,\bA_m\rightarrow\bA$ such that $\nu |_{\bA_v}= \nu_v$.

\begin{thm}\label{Thm=GraphCoreps}
We have,
\begin{enumerate}
\item The Haar state $\omega$ of $\bG$ is given by $\omega=\omega_\Gamma\circ \nu$.
\item All the reduced representations are irreducible and any non-trivial irreducible representation of $\bG$ is unitarily equivalent to a reduced one.
\item We have $C_m(\bG)=\bA_m$, $C_r(\bG)=\bA$, ${\rm L}^{\infty}(\bG)=\bM$ and $\nu = \nu_\bG$.
\end{enumerate}
\end{thm}
\begin{proof}
$(1)$. Let $\mathcal{P}\subset\bA_m$ be the linear span of the coefficients of the reduced representations (so $1 \not \in \mathcal{P}$). Since $\mathcal{P}_m$ equals the linear span of the reduced operators $a\in\bA_m$ relative to the family of states $(\omega_v)_{v\in V\Gamma}$ (see Remark \ref{Rmk=MaxReducedOp}) and of the form $a=a_1\dots a_n$, with $a_k\in{\rm Pol}( \bG_{v_k})$ it follows that the linear span of $1$ and $\mathcal{P}$ is dense in $\bA_m$. Hence, it suffices to show the invariance of $\omega$ on $\mathcal{P}$. Since $\Delta(\mathcal{P})\subset\mathcal{P}\odot\mathcal{P}$ and $\nu(\mathcal{P})$ is again contained in $\mathcal{P}$ (viewed within $\bA_\Gamma$) of the reduced operators in $\bA_\Gamma$ we have $(\id\ot \omega)\Delta(\mathcal{P})\subset(\id \ot \omega)( \mathcal{P} \odot \mathcal{P})=\{0\}$. In the same way we find $(\omega \ot\id)\Delta(\mathcal{P})=\{0\}$. Hence, for all $a\in\mathcal{P}$, one has $(\id\ot \omega)\Delta(a)=(\omega\ot\id)\Delta(a)=0=\omega(a)$.

\vspace{0.2cm}

\noindent$(2)$. Firstly, let $u^{\alpha_1} \otimes \ldots \otimes u^{\alpha_n} \in \bM_\Gamma \otimes M_{n_{\alpha_1}} \otimes \ldots \otimes M_{n_{\alpha_n}}$ (where $n_{\alpha_i}$ is the dimension of $u^{\alpha_i}$) be a reduced representation as in Definition \ref{Dfn=ReducedRep}. To conclude the first part of the statement, we shall show that the set of elements $(\omega \otimes \iota)(u^{\alpha_1} \otimes \ldots \otimes u^{\alpha_n})$ with $\omega \in (\bM_\Gamma)_\ast$ equals the complete matrix algebra $M_{n_{\alpha_1}} \otimes \ldots \otimes M_{n_{\alpha_n}}$. Consider a tensor product of matrix units $E_{i_1,j_1}^{\alpha_1} \otimes \ldots \otimes E_{i_n,j_n}^{\alpha_n}$ in the latter algebra. We have $(\omega_{v_k}( (u^{\alpha_k}_{i_k, j_k})^\ast \: \cdot \: ) \otimes \iota)(u^{\alpha_k}) = \lambda_{i_k, j_k}^{\alpha_k} E_{i_k, j_k}^{(\alpha_k)}$ for some constant $\lambda_{i_k, j_k}^{\alpha_k} \in \mathbb{C}$, see \cite[Proposition 2.1]{DawsPAMS}, \cite{WorQG}.  And therefore using that $\omega_{\Gamma}$ is the vector state of the vacuum vector we see that   $(\omega_{\Gamma}( (u^{\alpha_1}_{i_1, j_1} \ldots u^{\alpha_n}_{i_n, j_n} )^\ast \: \cdot \: ) \otimes \iota)(u^{\alpha_1} \otimes \ldots \otimes u^{\alpha_n}) =
\prod_k (\omega_{v_k}( (u^{\alpha_k}_{i_k, j_k})^\ast \: \cdot \: ) \otimes \iota)(u^{\alpha_k}) =
 \lambda_{i_1, j_1}^{\alpha_1} \ldots \lambda_{i_n, j_n}^{\alpha_n}  E_{i_1, j_1}^{(\alpha_1)} \otimes \ldots \otimes E_{i_n, j_n}^{(\alpha_n)}$.

\noindent Moreover, since the linear span of $\mathcal{P}$ and $1$ is dense in $\bA_m$, every non trivial irreducible representation is equivalent to a reduced one.

\vspace{0.2cm}

\noindent $(3)$. Since $\nu$ is surjective and $\omega_\Gamma$ is faithful on $\bA$, it follows from $(1)$ that $\bA=C_r(\bG)$, $\cH={\rm L}^2(\bG)$ and $\bM={\rm L}^{\infty}(\bG)$. It follows from $(2)$ that ${\rm Pol}(\bG)$ is the linear span of $\mathcal{P}$ and $1$. Hence, $C_m(\bG)$ is generated, as a C*-algebra, by $\bigcup_{v\in V\Gamma}{\rm Pol}(\bG_v)$ and the relations $a_va_{v'}=a_{v'}a_v$ are satisfied in $C_m(\bG)$, for all $a_v\in{\rm Pol}(\bG_v)$, $a_{v'}\in{\rm Pol}(\bG_{v'})$ and all $v,v'\in V\Gamma$ such that $(v,v')\in E\Gamma$. From the inclusions ${\rm Pol}(\bG_v)\subset C_m(\bG)$ and the universal property of $C_m(\bG_v)$ we have, for all $v\in V\Gamma$, a unital $\ast$-homomorphism $\pi_v\,:\,C_m(\bG_v)\rightarrow C_m(\bG)$ which is the identity on ${\rm Pol}(\bG_v)$. The morphisms $\pi_v$ are such that $\pi_v(a_v)\pi_{v'}(a_{v'})=\pi_{v'}(a_{v'})\pi_v(a_v)$ for all $a_v\in{\rm Pol}(\bG_v)$, $a_{v'}\in{\rm Pol}(\bG_{v'})$ and all $v,v'\in V\Gamma$ and $C_m(\bG)$ is generated by $\bigcup_{v\in V\Gamma}\pi_v(C_m(\bG_v)$. By universal property of  $\bA_m$, we have a surjective unital $*$-homomorphism from $\bA_m$ to $C_m(\bG)$ which is the identity on ${\rm Pol}(\bG)$. Hence, $\bA_m=C_m(\bG)$. That $\nu = \nu_\bG$ follows then since these maps are $\ast$-homomorphisms that agree on $\Pol(\bG)$.
\end{proof}

\subsection{Haagerup property of discrete quantum groups}\label{Sect=HP}

We show that the Haagerup property of discrete quantum groups is preserved by the graph product. In case the quantum group is of Kac type this follows from Corollary \ref{Cor=HAPFiniteKac} and \cite[Theorem 6.7]{DFSW}. Since it is unknown if the correspondence in \cite[Theorem 6.7]{DFSW} holds beyond Kac type quantum groups the general case requires a proof. The special case of free products was proved in \cite[Theorem 7.8]{DFSW} the special case of Cartesian products of quantum groups can be found in \cite[Proposition 3.4]{FreStability}. 

\subsubsection{General discrete quantum groups} Firstly recall the following equivalent definition of the Haagerup property for discrete quantum groups, see \cite[Proposition 6.2 and Lemma 6.24]{DFSW}.

\begin{prop}\label{Prop=HAPDiscrete}
A discrete quantum group $\widehat{\bG}$ has the Haagerup property if and only if there is a sequence of states $(\mu_k)_{k \in \mathbb{N}}$ on $\Pol(\bG)$ such that:
\begin{enumerate}
\item\label{Item=HAPDiscreteI} For each $k \in \mathbb{N}$ we have $((\mathcal{F} \mu_k)^\alpha)_{\alpha \in \Irr(\widehat{\bG} ) } \in \prod_{\alpha \in \Irr(\widehat{\bG} ) } M_{n_\alpha}$ is actually in $\oplus_{\alpha \in \Irr(\bG)} M_{n_\alpha}$.
\item\label{Item=HAPDiscreteII} For each $\alpha \in \Irr(\widehat{\bG})$ the net $((\mathcal{F} \mu_k)^\alpha)_{k \in \mathbb{N}}$ converges to the identity matrix.
\end{enumerate}
If these conditions hold, then we may moreover impose the following conditions on the net $(\mu_k)_{k \in \mathbb{N}}$,
\begin{enumerate}\setcounter{enumi}{2}
\item\label{Item=HAPDiscreteIII} For each $k \in \mathbb{N}$ and $\alpha \in \Irr(\widehat{\bG})$ with $\alpha \not = 1$ we have that $\Vert (\mathcal{F} \mu_k)^\alpha \Vert \leq \exp(-\frac{1}{k})$.
\end{enumerate}
\end{prop}

\noindent Recall the following definition from \cite{BozLeiSpe}.

\begin{dfn}
Let $\cA$ be a unital $\ast$-algebra over $\mathbb{C}$. A linear map $\omega: \cA \rightarrow \mathbb{C}$ is called a state if $\omega(1) = 1$, $\omega(a^\ast) = \overline{\omega(a)}$ and $\omega(a^\ast a) \geq 0$ for every  $a \in A$.
\end{dfn}

\noindent Let $\cA_v, v \in V\Gamma$ be unital $\ast$-algebras, each equipped with a state $\varphi_v$. Let $\cA$ be its algebraic graph product which is defined as the unital $\ast$-algebra freely generated by $\cA_v, v \in V\Gamma$ subject to the relation that $a_v a_w = a_w a_v$ for any $a_v \in \cA_v, a_w \in \cA_w$ such that $(v,w) \in E\Gamma$ (and the units of each $\cA_v$ are identified). Using the decomposition $\cA_v = \mathbb{C} 1 \oplus \cA_v^\circ$ with $\cA_v = \ker \varphi_v$ we may identify $\cA$ with the vector space $\mathbb{C} 1 \oplus \bigoplus_{v_1 \ldots v_n \in \cWmin} \cA_{v_1}^\circ \otimes \cA_{v_2}^\circ  \otimes \ldots \otimes \cA_{v_n}^\circ$. Suppose that there exists a state $\psi_v$ on each $\cA_v, v \in V\Gamma$, then the algebraic graph product functional $\psi = \diamond_{v \in V\Gamma} (\psi_v, \varphi_v)$ on $\cA$ is defined as $\psi( a_1\ldots a_n ) = \psi_{v_1}(a_1)\ldots\psi_{v_n}(a_n)$ whenever $a_i \in A_{v_i}^\circ$ with $v_1 \ldots v_n \in \cWmin$.

\vspace{0.2cm}

\noindent Now let again $\bG_v, v \in V\Gamma$ be a compact quantum group and $\bG$ be its graph product.  The proof of the following theorem is similar to \cite[Theorem 7.8]{DFSW}.

\begin{thm}
The discrete quantum group $\widehat{\bG}$ has the Haagerup property if and only if for every $v \in V \Gamma$ we have that $\widehat{\bG}_v$ has the Haagerup property.
\end{thm}
\begin{proof}
By a standard inductive limit argument it suffices to prove the theorem under the condition that the graph $\Gamma$ is finite. Firstly, suppose that for every $v \in V \Gamma$ the quantum group $\widehat{\bG}_v$ has the Haagerup property. By Proposition \ref{Prop=HAPDiscrete} there exists a sequence $(\mu_{v,k})_{k \in \mathbb{N}}$ of states on $\Pol(\bG_v)$ satisfying \eqref{Item=HAPDiscreteI} - \eqref{Item=HAPDiscreteIII} of this proposition. Recall that $\omega_v$ is the Haar state of $\bG_v$. Let $\mu_k = \diamond_{v \in V\Gamma} (\mu_{v,k}, \omega_v)$ denote the graph product functional as defined in the paragraph before this theorem.

\vspace{0.2cm}

\noindent We claim that $\mu_{k}, k \in \mathbb{N}$ is again a state. This follows from the following standard argument. For convenience of notation fix  $k\in \mathbb{N}$.  From the state $\mu_{v,k}$ on $\Pol(\bG_v)$ we may follow the usual GNS-construction to find a pre-Hilbert space $\cH_{v,0}$ with cyclic unit vector $\xi_v$ and representation $\pi_v$ such that $\mu_{v,k}(x) = \langle \pi_v(x) \xi_v, \xi_v \rangle$. Let $\bA_{v,m}$ be the maximal C$^\ast$-algebra associated with the quantum group $\bG_v$. As in \cite[Lemma 4.2]{DijkKoo} the map $\pi_v$ extends to a $\ast$-homomorphism $\bA_{v,m} \rightarrow B(\cH_v)$ with $\cH_v$ the completion of $\cH_{v,0}$. Let $\bB$ be the reduced graph product C$^\ast$-algebra of $\pi_v( \bA_{v,m} ), v \in V \Gamma$ and let $\xi$ denotes its cyclic vacu\"um vector. Since $\pi_v( \bA_{v,m} )$ is included into $\bB$ naturally we may regard $\pi_v$ as a $\ast$-homomorphism $\bA_{v,m} \rightarrow \bB$. The universal property of the maximal graph product C$^\ast$-algebra $\bA_m$ then yields a $\ast$-homomorphism $\pi: \bA_m \rightarrow \bB$. Let $\mu_k$ be the state on $\bA_m$ defined by $\mu_k(x) = \langle \pi(x) \xi, \xi \rangle$ and denote by $\mu_k$ the restriction to $\Pol(\bG)$. It follows from Theorem \ref{Thm=GraphCoreps} that indeed $\Pol(\bG)$ is the algebraic graph product of $\Pol(\bG_v), v \in V\Gamma$ and by construction it follows that $\mu_k$ is the graph product of the states $\mu_{k,v}, v \in V\Gamma$. In particular $\mu_k$ is again a state.

\vspace{0.2cm}

\noindent Let $\alpha_j, 1 \leq j \leq l$ be elements of $\Irr(\bG_{v_j})$ with $v_j$ such that $v_1 v_2\ldots v_l \in \cWmin$. By definition of the graph product and the graph product representation $\alpha_1 \otimes \ldots  \otimes \alpha_n$, see Theorem \ref{Thm=GraphCoreps}, we see that,
\[
(\mathcal{F} \mu_k)^{\alpha_1 \otimes \ldots \otimes  \alpha_l} = \otimes_{j = 1}^l (\mathcal{F} \mu_{v_j,k})^{\alpha_j}.
\]
It is then straightforward to verify conditions \eqref{Item=HAPDiscreteI}  and \eqref{Item=HAPDiscreteII} of Proposition \ref{Prop=HAPDiscrete}. For condition \eqref{Item=HAPDiscreteI} one uses that $\Gamma$ is finite and that $\Vert (\mathcal{F} \mu_{k} )^{\alpha_1\otimes \ldots \otimes \alpha_l} \Vert \leq \exp(-\frac{l}{k})$.
\end{proof}

\section{Rapid Decay}\label{Sect=RD}

\noindent We prove that the property of Rapid Decay (RD) for discrete quantum groups is preserved by taking graph products of finite graphs under suitable conditions on the vertex quantum groups. In particular our result holds if every vertex quantum group is either a classical group or a quantum group with polynomial growth. This generalizes the result of \cite{CioHolRee} which proves the corresponding result for discrete groups.

\subsection{Preliminaries on elements affiliated with a C$^\ast$-algebra} For unbounded operators affiliated with a C$^\ast$-algebra we refer to \cite{WorAffiliated}. When $\widehat{\bA}$ is the C$^\ast$-algebra of a discrete quantum group $\widehat{\bG}$ the notion of affiliated elements simplifies. In that case, the $\ast$-algebra $\widehat{\bA}^\eta$ of affiliated elements with $\widehat{\bA}$ can be identified with the algebraic product $\prod_{\alpha \in \Irr(\bG)} M_{n_\alpha}$ and, for each operator in $\widehat{\bA}^\eta$, the vector space $\cHPol$ (the space of matrix coefficients of finite dimensional representations of $\bG$ identified as subspace of $\cH$) forms a core. For $L \in \widehat{\bA}^\eta$ we will write $\prod_{\alpha \in \Irr(\bG)} L^{(\alpha)}$ for this representation.  Any $\ast$-homomorphism of $\widehat{\bA}$ extends naturally to $\widehat{\bA}^\eta$ through spectral calculus. In particular this applies to the counit $\widehat{\epsilon}$ and comultiplication $\widehat{\Delta}$  as well as the antipode $\hat{\kappa}$ of a Kac type discrete quantum group.

\subsection{Definition of Rapid Decay}

 Let $\bG = (\bA, \Delta)$ be a compact quantum group with discrete dual $\widehat{\bG} = (\widehat{\bA}, \widehat{\Delta})$. Then $\widehat{\bA} = \oplus_{\alpha \in \Irr(\bG)} M_{n_\alpha}$ where $n_\alpha$ is the dimension of $\alpha$.
In case $\widehat{\bG}$ is of Kac type its Haar weight $\widehat{\omega}$   is given by $\widehat{\omega} = \oplus_{\alpha \in \Irr(\bG)} n_\alpha {\rm Tr}_{M_{n_\alpha}}$ where ${\rm Tr}_{M_{n_\alpha}}$ is the normalized trace on $M_{n_\alpha}$.
For every $\alpha \in \Irr(\bG)$ let $u^\alpha \in \bA \otimes M_{n_\alpha}$ be a representation belonging to the equivalence class $\alpha$.   The Fourier transform $\mathcal{F}$ of $x = \oplus_{\alpha \in \Irr(\bG)} x_\alpha  \in \widehat{\bA}$ with finite direct sum, is defined as the element,
\[
\sum_{\alpha \in \Irr(\mathbb{G})} (\id \otimes \widehat{\omega} ) (u^\alpha (1 \otimes x_\alpha) ).
\]

\begin{dfn}[Lengths and central lengths] A length on $\widehat{\bG}$ is an (unbounded) operator affiliated with $\widehat{\bA}$ that satisfies the following properties: $L \geq 0$, $\widehat{\epsilon}(L) = 0$, $\widehat{\kappa}(L)\vert_{\cHPol} = L\vert_{\cHPol}$ and $\widehat{\Delta}(L) \leq 1 \otimes L + L \otimes 1$. Given such a length, we denote by $q_n \in \mathcal{M}(\widehat{\bA})$ (the multiplier algebra of $\widehat{\bA}$) the spectral projection of $L$ associated to the interval $[n, n+1)$, $n \in \mathbb{N}$. $L$ is called central if each of its spectral projections are central in $\mathcal{M}(\widehat{\bA})$.
\end{dfn}

\begin{dfn}
Let $L$ be a central length on the discrete quantum group $\widehat{\bG} = (\widehat{\bA}, \widehat{\Delta})$. We say that $(\widehat{\bG}, L)$ has the property of Rapid Decay (RD) if the following condition is satisfied: there exists a polynomial $P \in \mathbb{R}[X]$ such that, for every $k \in \mathbb{N}$, $a \in q_k \widehat{\bA}$ and every $m,l \in \mathbb{N}$, we have $\Vert q_m \mathcal{F}(a) q_l \Vert \leq P(k) \Vert a \Vert_2$.
\end{dfn}

\noindent In fact there are other equivalent formulations of (RD), see \cite[Proposition and Definition 3.5]{VerRd} or \cite{JolRD} for the group case.

\subsection{Permanence properties of (RD)} We prove permanence properties of (RD) under graph products. In particular we prove that (RD) is preserved by free products. We collect some standard and well known observations in the next lemma. We include a proof for the convenience of the reader.

\begin{lem} \label{Lem=KappComputation}
Let $\bG$ be a compact quantum group of Kac type with discrete dual quantum group $\widehat{\bG} = (\widehat{\bA}, \widehat{\Delta})$. Let $\{ u^{\alpha} \mid \alpha \in \Irr(\bG) \}$ be a complete set of representatives of irreducible representations and let $u^\alpha_{i,j} = ( \id  \otimes \omega_{e_i, e_j} )(u^\alpha)$ denote its matrix coefficients with respect to some orthonormal basis $e_i$ of the representation space $\cH_\alpha$ for which,
\begin{equation}\label{Eqn=Ortho}
\omega( (u^\alpha_{i,j})^\ast u^\alpha_{k,l}) = \delta_{i,k} \delta_{j,l} n_\alpha^{-1},
 \end{equation}
 (see \cite[Proposition 2.1]{DawsPAMS}, \cite{WorQG}). The contragredient representation $\overline{\alpha}$ is given by  $u^{\overline{\alpha}}_{i,j} = u^\alpha_{j,i}$ (and this definition is consistent with \eqref{Eqn=Ortho}). Let $E_{i,j}^\alpha \in \widehat{\bA}$ be the matrix with entry 1 on the $i$-th row and $j$-th column of the matrix block indexed by $\alpha \in \Irr(\bG)$ and zeros elsewhere. Then $\widehat{\kappa}(E_{i,j}^\alpha ) = E_{j,i}^{\overline{\alpha}}$.
\end{lem}
\begin{proof}
The proof is a consequence of the relation $\kappa(u_{i,j}^\alpha) = (u_{j,i}^\alpha)^\ast$ and using duality between $\bG$ and $\widehat{\bG}$. So let $\omega_{i,j}^\alpha(\: \cdot \:) = n_\alpha \: \omega( (u_{i,j}^\alpha)^\ast \: \cdot \: )$ so that by orthogonality (see \cite[p. 1351]{DawsPAMS}, \cite{WorQG}) we have, \begin{equation}\label{Eqn=FTOMega}
(\omega_{i,j}^\alpha \otimes \id)(W) = E_{i,j}^\alpha,
\end{equation}
where $W = \oplus_{\alpha \in \Irr(\bG)} u^\alpha$.
 Then we have using that for Kac algebras $\kappa^2 = \id$, $\kappa$ is an anti-homomorphism, $\omega \circ \kappa = \omega$, traciality of the Haar state $\omega$ and the relation $\kappa(u_{i,j}^\alpha) = (u_{j,i}^\alpha)^\ast$, see \cite{Tim},
\[
\begin{split}
&\omega_{i,j}^\alpha \circ \kappa =  \omega( \kappa^2((u_{i,j}^\alpha)^\ast) \kappa(\: \cdot \:) )
= \omega( \kappa( \: \cdot \:\:  \kappa((u_{i,j}^\alpha)^\ast)) ) = \omega(  \: \cdot \:\:  \kappa((u_{i,j}^\alpha)^\ast) ) =
\omega(  \kappa((u_{i,j}^\alpha)^\ast) \: \: \cdot \:  ) \\
= &  \omega( \kappa(u_{i,j}^\alpha)^\ast \:\: \cdot \: ) =  \omega( u_{j,i}^\alpha \: \: \cdot \: )= \omega( (u_{j,i}^{\overline{\alpha}})^\ast \: \cdot \: ),
\end{split}
\]
so that $(\omega_{i,j}^\alpha \circ \kappa \otimes \id)(W) = E_{j,i}^{\overline{\alpha}}$ by \eqref{Eqn=FTOMega}. This means that using the relation $(\kappa \otimes \hat{\kappa})(W) = W$ \cite{KusVaeVNA},
\[
\widehat{\kappa}(E^\alpha_{i,j}) = \hat{\kappa} (\omega_{i,j}^\alpha \otimes \id)(W) = (\omega_{i,j}^\alpha \circ \kappa \otimes \id)(W) = E_{j,i}^{\overline{\alpha}}.
\]
\end{proof}

\noindent Now let us return to graph products. So let $\Gamma$ be again a simplicial graph and for each $v \in V\Gamma$ let $\bG_v$ be a compact quantum group with discrete dual $\widehat{\bG}_v$. Let $\bG$ be the graph product of $\bG_v, v \in V \Gamma$ and let $\widehat{\bG}$ be its discrete dual. From Theorem \ref{Thm=GraphCoreps} we see that the C$^\ast$-algebra $\widehat{A}$ associated to $\widehat{\bG}$ equals
\[
\oplus_{\alpha \in \Irr(\bG)}  M_{n_{\alpha_1}} \otimes \ldots \otimes M_{n_{\alpha_l}},
\]
in case $\alpha = \alpha_1 \otimes \ldots \otimes \alpha_l$. For $k \in \mathbb{N}$ we shall use the notation $\bA_{(k)}$ for the subspace defined by
\[
\oplus_{\alpha \in \Irr(\bG), \alpha = \alpha_1 \otimes \ldots \otimes \alpha_k}  M_{n_{\alpha_1}} \otimes \ldots \otimes M_{n_{\alpha_k}},
\] so the subspace of exactly $k$-fold tensor products of matrices.
Let $w \in V \Gamma$. We shall denote $P_w: \cH \rightarrow \cH$ for the projection onto the linear span of the Hilbert spaces $\cH_{\boldv}$ with $\boldv \in \cWmin$ a word that is equivalent to a word that starts with $w$.
Now we are able to state the following Lemma \ref{Lem=GraphProductLength}.

\begin{lem}\label{Lem=GraphProductLength}
For $v \in V \Gamma$ suppose that $L_v = \oplus_{\alpha \in \Irr(\bG)} L_v^{(\alpha)}$ is a central length for the discrete quantum group $\widehat{\bG}_v$. Define,
\begin{equation}\label{Eqn=GraphProductLength}
L = \prod_{\alpha \in \Irr(\bG) } \sum_{i=1}^{l(\alpha)} 1_{M_{n_{\alpha_1}}} \otimes \ldots \otimes 1_{M_{n_{\alpha_{i-1}}}} \otimes L_{v_i}^{(\alpha_i)} \otimes  1_{M_{n_{\alpha_{i+1}}}} \otimes \ldots \otimes 1_{M_{n_{\alpha_{l(\alpha)}}}},
\end{equation}
where each $\alpha \in \Irr(\bG)$ decomposes as the tensor product representation $\alpha_1 \otimes \ldots \otimes \alpha_{l(\alpha)}$ and $\alpha_i \in \Irr(\bG_{v_i})$. Then $L$ is a central length function for the discrete quantum group $\widehat{\bG}$.
\end{lem}
\begin{proof}

We first check that $\widehat{\Delta}(L) \leq L \otimes 1 + 1 \otimes L$. Recall from \cite[Eqn. (1) in Proposition 3]{FimaPropT} that,
\begin{equation}\label{Eqn=Projection}
\widehat{\Delta}(p_\gamma ) (p_\alpha \otimes p_\beta) = \left\{
\begin{array}{ll}
p_\gamma^{\alpha \otimes \beta} & {\rm if } \:\: \gamma \subseteq \alpha \otimes \beta, \\
0 & {\rm otherwise},
\end{array}
\right.
\end{equation}
where $p_\gamma^{\alpha \otimes \beta} \in \cB(\cH_\alpha  \otimes \cH_\beta)$ is the projection onto the sum of all subrepresentations of $\alpha \otimes \beta$ that are equivalent to $\gamma$. Since the length functions $L_v, v \in V\Gamma$ are central, we know that $L_v = \oplus_{\alpha \in \Irr(\bG_v) } f_v(\alpha) p_\alpha$ for some $f_v: \Irr(\bG_v) \rightarrow [0, \infty)$ and similarly $L = \oplus_{\alpha \in \Irr(\bG)} f(\alpha) p_\alpha$. In fact, by definition of $L$ we have that $f(\alpha) = f_{v_1}(\alpha_1) + \ldots + f_{v_n}(\alpha_n)$ in case $\alpha = \alpha_1 \otimes \ldots \otimes \alpha_n$. The condition $\widehat{\Delta}(L) \leq L \otimes 1 + 1 \otimes L$ now becomes equivalent to the property that for every $\alpha, \beta \in \Irr(\bG)$ we have $\widehat{\Delta}(L) (p_\alpha \otimes p_\beta) \leq (L \otimes 1 + 1 \otimes L)p_\alpha \otimes p_\beta $, which by \eqref{Eqn=Projection} is equivalent to,
\begin{equation}\label{Eqn=ProjectionCondition}
\sum_{\gamma \in \Irr(\bG), \gamma \subseteq \alpha \otimes \beta} f(\gamma) p_\gamma^{\alpha \otimes \beta} \leq (f(\alpha) + f(\beta)) p_\alpha \otimes p_\beta.
\end{equation}

\noindent Now fix $\alpha = \alpha_1 \otimes \ldots \otimes \alpha_n \in \Irr(\bG)$ and $\beta = \beta_1 \otimes \ldots \otimes \beta_m \in \Irr(\bG)$. Let $v_i$ and $w_i$ be such that $\alpha_i \in \Irr(\bG_{v_i})$ and $\beta_i \in \Irr(\bG_{w_i})$. $\alpha \otimes \beta$ is not necessarily irreducible, c.f. Theorem \ref{Thm=GraphCoreps}. If $(v_i, v_{i+1}) \in E\Gamma$ then  $\alpha_1 \otimes \ldots \otimes \alpha_n$ is unitarily equivalent to $\alpha_1 \otimes \ldots \otimes \alpha_{i-1} \otimes \alpha_{i+1} \otimes \alpha_{i} \otimes \alpha_{i+2} \otimes \ldots \alpha_n$ by intertwining with the flip map $\id^{\otimes i-1} \otimes \Sigma \otimes \id^{\otimes n-i-1}$ . Therefore, without loss of generality, we may assume that there exists $r$ such that $v_1\ldots v_r w_1\ldots w_m$ is reduced and $v_1\ldots v_rw_1\ldots w_m \simeq v_1 \ldots v_n w_1 \ldots v_m$. Note that this implies that $w_1, \ldots, w_{n-r}$ commute and $\{w_1, \ldots, w_{n-r}\} = \{v_{r+1}, \ldots, v_m \}$. Therefore, without loss of generality, we may assume that $v_{r+1} = w_1, \ldots, v_{n} = w_{n-r}$ (since $\beta$ is equivalent to a representation for which this is true, again by intertwining with flip maps).
Then $\alpha \otimes \beta$ is equivalent to
\begin{equation}\label{Eqn=ShuffledCorep}
\alpha_1 \otimes \ldots \otimes \alpha_r \otimes \alpha_{r+1} \otimes \beta_1 \otimes \alpha_{r+2} \otimes \beta_2 \otimes \ldots \otimes \alpha_n \otimes \beta_{n-r} \otimes \beta_{n-r+1} \otimes \ldots \otimes \beta_m.
\end{equation}
Suppose that $\gamma \in \Irr(\bG)$ is contained in \eqref{Eqn=ShuffledCorep}. Then by the Peter-Weyl decompositions of $\alpha_{r+1} \otimes \beta_1, \ldots, \alpha_n \otimes \beta_{n-r}$, there exist irreducible representations $\gamma_1, \ldots, \gamma_{n-r}$ with $\gamma_1 \subseteq \alpha_{r+1} \otimes \beta_1, \ldots, \gamma_{n-r} \subseteq \alpha_n \otimes \beta_{n-r}$ such that $\gamma \simeq \alpha_1 \otimes \ldots \otimes \alpha_r \otimes \gamma_1 \otimes \ldots \otimes \gamma_{n-r} \otimes \beta_{n-r+1} \otimes \ldots \otimes \beta_m$. This implies that $f(\gamma) = \sum_{i=1}^{r} f_{v_i}(\alpha_i) + \sum_{i=1}^{n-r} f_{v_{i+r}}(\gamma_{i}) + \sum_{i=n-r+1}^{m} f_{w_i}(\beta_i)$ and since $f_{v_{i+r}}$ is a length function, this implies that $f(\gamma) \leq \sum_{i=1}^r f_{v_i}(\alpha_i) + \sum_{i=1}^{n-r} \left( f_{v_{i+r}}(\alpha_{i+r}) + f_{u_i}(\beta_i) \right) + \sum_{i=n-r+1}^{m} f_{w_i}(\beta_i) = \sum_{i=1}^n f_{v_i}(\alpha_i) + \sum_{i=1}^m f_{w_i}(\beta_i) $ and so condition \eqref{Eqn=ProjectionCondition} holds.

\noindent Next we check the relation $\widehat{\kappa}(L) = L$. Let $\alpha \in \Irr(\bG)$ and assume that it decomposes as a reduced tensor product $\alpha_1 \otimes \ldots \otimes \alpha_n$. The contragredient representation (see Lemma \ref{Lem=KappComputation}) is then given by $\overline{\alpha_n} \otimes \ldots \otimes \overline{\alpha_1}$. This implies, using Lemma \ref{Lem=KappComputation} and its notation, that $\widehat{\kappa}(E_{i_1, j_1}^{\alpha_1} \otimes \ldots \otimes E_{i_n, j_n}^{\alpha_n} ) = E_{j_n, i_n}^{\overline{\alpha_n}} \otimes \ldots \otimes E_{j_1, i_1}^{\overline{\alpha_1}} = \widehat{\kappa}(E_{i_n, j_n}^{\alpha_n}) \otimes \ldots \otimes \widehat{\kappa}(E_{i_1, j_1}^{\alpha_1}  ) $. Applying the latter observation to \eqref{Eqn=GraphProductLength} yields that $\widehat{\kappa}(L) = L$.

\noindent Finally, we have
$
\widehat{\epsilon}(L) = f(\alpha_0) p_{\alpha_0}$,
    with $\alpha_0 \in \Irr(\bG)$ the trivial representation. Since $f(\alpha_0) = 0$ we have $\widehat{\epsilon}(L) = 0$.
\end{proof}

\noindent The following Lemma \ref{Lem=RDTensor} uses the notion of polynomial growth \cite{VerRd}. Examples of discrete quantum groups with polynomial growth can be found in \cite{BanicaVer}.
\begin{dfn}
Let $\bG$ be a compact quantum group with central length $L$ and dual Haar weight $\widehat{\omega}$. Set $q_n = \chi_{[n, n+1)}(L)$. $\bG$ has polynomial growth if there exists a polynomial $P$ such that for every $n$ we have $\widehat{\omega}(q_n) \leq P(n)$.
\end{dfn}

\begin{lem} \label{Lem=RDTensor}
Let $\bG_1$ and $\bG_2$ be compact quantum groups such that $(\widehat{\bG}_1, L_1)$ has (RD). If either $(\widehat{\bG}_2,L_2)$ has polynomial growth or is a classical discrete group with (RD) then $(\widehat{\bG}_1 \times \widehat{\bG}_2, L)$ has (RD) where $L$ was defined in Lemma \ref{Lem=GraphProductLength}.
\end{lem}
\begin{proof}
Let $q_k^{(1)}, q_k^{(2)}$ and $q_k$ be the spectral projections onto $[k, k+1)$ of respectively $L_1, L_2$ and $L$. Also write $q_{\leq k}^{(1)} = \sum_{j=0}^k q_j^{(1)}, q_{\leq k}^{(2)} = \sum_{j=0}^k q_j^{(2)}$.  Let $\widehat{\bA}_1, \widehat{\bA}_2$ and $\widehat{\bA}$ be the C$^\ast$-algebras associated to the duals of respectively $\bG_1, \bG_2$ and $\bG$. Let $P_1$, $P_2$ be the polynomials witnessing (RD) for $\widehat{\bG}_1$ and $\widehat{\bG}_2$ respectively.

\vspace{0.2cm}

\noindent In order to prove that, for $x \in q_k \widehat{\bA}$, we have $\Vert q_l \mathcal{F}(x) q_m \Vert \leq P(k) \Vert x \Vert_2$ for some polynomial $P$, it suffices to prove that, for every $k$, the following estimate holds,
\[
\Vert  \mathcal{F}(x)   \Vert \leq P(k) \Vert x \Vert_2, \qquad \textrm{ for all } x \in (q_{\leq k}^{(1)} \otimes q_{\leq k}^{(2)})  \widehat{\bA}.
\]
It follows from the fact that $q_{\leq k}^{(1)} \otimes q_{\leq k}^{(2)} \geq q_k$. Observe that, by definition, we have $\mathcal{F}(a\ot b)=\mathcal{F}_1(a)\ot\mathcal{F}_2(b)$ for all $a\in q_{\leq k}^{(1)}\widehat{\bA}_1$ and all $b\in q_{\leq k}^{(2)}\widehat{\bA}_2$.

\vspace{0.2cm}

\noindent First assume that $(\widehat{\bG}_2,L_2)$ has polynomial growth (so, in particular it is an amenable discrete Kac algebra with property (RD) \cite{VerRd}). Let $\widehat{\omega}_i$ be the Haar weight on $\widehat{\bA}_i$, $i=1,2$. By polynomial growth we have $\widehat{\omega}_2\left(q^{(2)}_n\right) \leq P_3(n)$, where $P_3$ is a polynomial with $P_3(n)\geq 1$ for all $n$. Let $x\in(q_{\leq k}^{(1)} \otimes q_{\leq k}^{(2)})  \widehat{\bA}$ be a finite sum $x=\sum_ia_i\ot b_i$, where $a_i\in q_{\leq k}^{(1)}\widehat{\bA}_1$ with $\Vert a_i\Vert_2<\infty$ and $b_i\in q_{\leq k}^{(2)}\widehat{\bA}_2$ with $\Vert b_i\Vert_2<\infty$ for all $i$. We may and will assume that $(a_i)$ is an orthonormal system with respect to the scalar product given by $\widehat{\omega}_1$. Hence, $\Vert x\Vert_2^2=\sum_i\Vert b_i\Vert_2^2$ and,
\begin{eqnarray*}
\Vert\mathcal{F}(x)\Vert  &=&  \Vert\sum_i\mathcal{F}_1(a_i)\ot\mathcal{F}_2(b_i)\Vert\leq\sum_i\Vert\mathcal{F}_1(a_i)\Vert\,\Vert\mathcal{F}_2(b_i)\Vert
\leq P_1(k)P_2(k)\sum_i\Vert a_i\Vert_2\Vert b_i\Vert_2\\
&=&P_1(k)P_2(k)\sum_i\Vert b_i\Vert_2\leq P_1(k)P_2(k)\sqrt{\dim(q_{\leq k}^{(2)}\widehat{\bA}_2)}\sqrt{\sum_i\Vert b_i\Vert_2^2}=P_1(k)P_2(k)\sqrt{\dim(q_{\leq k}^{(2)}\widehat{\bA}_2)}\Vert x\Vert_2\\
&=&P_1(k)P_2(k)\Vert x\Vert_2\sqrt{\widehat{\omega}_2\left(q_{\leq k}^{(2)}\right)}=P_1(k)P_2(k)\Vert x\Vert_2\sqrt{\sum_{j=0}^k\widehat{\omega}_2\left(q_{j}^{(2)}\right)}
\leq P_1(k)P_2(k)\Vert x\Vert_2\sqrt{\sum_{j=0}^kP_3(j)}\\
&\leq&P_1(k)P_2(k)\left(\sum_{j=0}^kP_3(j)\right)\Vert x\Vert_2\leq Q(k)\Vert x\Vert_2,
\end{eqnarray*}
where $Q$ is a polynomial.

\vspace{0.2cm}

\noindent Now assume that $(\widehat{\bG}_2,L_2)$ is a discrete group, denoted by $G$, with property (RD) so that we may take $x\in(q_{\leq k}^{(1)} \otimes q_{\leq k}^{(2)})  \widehat{\bA}$ to be a finite sum of the form $x=\sum_g a_g\ot\delta_g$, where $a_g\in q_{\leq k}^{(1)}\widehat{\bA}_1$ with $\Vert a_g\Vert_2<\infty$ for all $g$, $\delta_g\in l^{\infty}(G)$ is the Dirac function at $g\in G$ and $a_g=0$ for all $g\in G$ such that $L_2(g)> k$. Hence we have $\Vert x\Vert_2^2=\sum_g\Vert a_g\Vert_2^2$ and $\mathcal{F}(x)=\sum_g\mathcal{F}_1(a_g)\ot\lambda_g$, where $\lambda_g\in\mathcal{B}(l^2(G))$ is the left translation by $g\in G$. Let $\xi\in {\rm L}^2(\bG_1)\ot l^2(G)$ be a finite sum $\xi=\sum_h \xi_h \ot\delta_h$. One has:

\begin{eqnarray*}
\Vert\mathcal{F}(x)\xi\Vert_2^2 &=&\Vert\sum_{g,h}\mathcal{F}_1(a_g)\xi_h\ot\delta_{gh}\Vert_2^2
=\Vert\sum_{g,h}\mathcal{F}_1(a_g)\xi_{g^{-1}h}\ot\delta_{h}\Vert_2^2=\sum_h\Vert\sum_g\mathcal{F}_1(a_g)\xi_{g^{-1}h}\Vert_2^2\\
&\leq&\sum_h\left(\sum_g\Vert\mathcal{F}_1(a_g)\xi_{g^{-1}h}\Vert_2\right)^2\leq P_1(k)^2\sum_h\left(\sum_g\Vert a_g\Vert_2\,\Vert\xi_{g^{-1}h}\Vert_2 \right)^2
=P_1(k)^2\Vert\psi*\varphi\Vert_{l^2(G)}^2,
\end{eqnarray*}
where $\psi,\varphi\in l^2(G)$ are defined by $\psi(g)=\Vert a_g\Vert_2$ and $\varphi(g)=\Vert\xi_g\Vert$. Observe that $\Vert \psi\Vert_{l^2(G)}^2=\sum_g\Vert a_g\Vert_2^2=\Vert x\Vert_2^2$ and $\Vert\varphi\Vert_{l^2(G)}^2=\sum_g\Vert\xi_g\Vert^2=\Vert\sum_g\xi_g\ot\delta_g\Vert^2=\Vert\xi\Vert^2$. Since $\psi$ is supported on elements $g\in G$ of length less that $k$, we may use (RD) for $G$ and we find:

$$\Vert\mathcal{F}(x)\xi\Vert^2\leq P_1(k)^2P_2(k)^2\Vert\psi\Vert_{l^2(G)}^2\Vert\varphi\Vert_{l^2(G)}^2=P_1(k)^2P_2(k)^2\Vert x\Vert_2^2\Vert\xi\Vert^2.$$
This finishes the proof.
\end{proof}

\noindent Let $P_m: \cH \rightarrow \cH$ be the projection onto the closure of the span of the spaces $\cH_{\boldw}$ with $\boldw$ a minimal word of length $m \in \mathbb{N}$.

\begin{prop}\label{Prop=RDIntermediate}
Let $\Gamma$ be a finite graph and for every $v \in V\Gamma$ let $\bG_v$ be a compact quantum group such that $(\widehat{\bG}_v, L_v)$ has (RD). Moreover, assume that for every clique $\Gamma_0$ of $\Gamma$ the graph product $\widehat{\bG}_{\Gamma_0}$ has (RD).
 Let $\bG$ be the graph product with respect to $\Gamma$ and let  $\widehat{\bG} = (\widehat{A}, \widehat{\Delta})$ be its discrete dual. There exists a polynomial $P \in \mathbb{R}[X]$ such that for every $k,l,m \in \mathbb{N}$ such that $\vert k-l \vert \leq m \leq k+l$ and $a \in \widehat{\bA}_{(k)}$ we have $\Vert P_m \mathcal{F}(a) P_l \Vert \leq P(k) \Vert a \Vert_{2}$.
\end{prop}
\begin{proof}
For each $v \in V \Gamma$ we let $a_{v,j}, j \in J_v$, be elements of $\bA_v^\circ$ such that $\widehat{a}_{v,j} := \widehat{a}_{v,j}\Omega_v, j \in J_v$ is an orthonormal basis of $\cH_v^\circ$. Set $\xi_{v,j} = \widehat{a}_{v,j}^\ast = a_{v,j}^\ast\Omega_v, j \in J_v$ which also is an orthonormal basis of $\cH_v^\circ$ since $\bG_v$ has a tracial Haar weight \cite[Proposition 4.7]{VerRd}. In particular, $\Vert \xi_{v,j} \Vert = \Vert \widehat{a}_{v,j} \Vert$. Throughout the proof we shall use the convention that a summation  $\sum_{j} a_{v,j}$ in fact is the summation over $j \in J_v$.  To prove the proposition it suffices to assume that,
\[
\begin{split}
\mathcal{F}(a)  =& \sum_{\boldw \in \cWmin, l(\boldw) = k} \sum_{j_1 \ldots j_k} \lambda_{\boldw, j_1, \ldots, j_k} a_{w_1, j_1} \ldots a_{w_k, j_k},\\
\xi =  &  \sum_{\boldv \in \cWmin, l(\boldv) = l} \sum_{i_1 \ldots i_l} \mu_{\boldv, i_1, \ldots, i_l} \xi_{v_1, i_1} \otimes \ldots \otimes \xi_{v_l, i_l}.
\end{split}
\]
Firstly, using the notation introduced before Lemma \ref{Lem=GraphProductLength},
\begin{equation}\label{Eqn=LoadsOfPs}
\begin{split}
& P_m \mathcal{F}(a) \xi \\
= & P_m \left(  \sum_{\boldw \in \cWmin, l(\boldw) = k} \sum_{j_1, \ldots, j_k} \lambda_{\boldw, j_1, \ldots, j_k} (P_{w_1} + P_{w_1}^\perp ) a_{w_1, j_1} (P_{w_1} + P_{w_1}^\perp) \ldots (P_{w_k} + P_{w_k}^\perp) a_{w_k, j_k}  (P_{w_k} + P_{w_k}^\perp)  \right)\\
 &\times \left( \sum_{\boldv \in \cWmin, l(\boldv) = l}  \sum_{i_1, \ldots, i_l} \mu_{\boldv, i_1, \ldots, i_l} \xi_{v_1, i_1}  \otimes \ldots \otimes \xi_{v_l, i_l}   \right) \\
\end{split}
\end{equation}
A large part of the terms in the product of these sums vanishes in fact as follows from the following observations.

\vspace{0.3cm}

\noindent {\bf Reduction of the operator part.}
Firstly, consider an expression:
\begin{equation}\label{Eqn=QSeries}
Q_{w_1}^{(2)} a_{w_1, j_1} Q_{w_1}^{(1)}\ldots Q_{w_k}^{(2)} a_{w_k, j_k} Q_{w_k}^{(1)},
\end{equation}
with $Q_{w_i}^{(1)}$ and $Q_{w_i}^{(2)}$ equal to either $P_{w_i}$ or $P_{w_i}^\perp$. Assume that \eqref{Eqn=QSeries} is non-zero, then this implies the following:
\begin{enumerate}
\item\label{Item=EnumiI} If $Q_{w_i}^{(1)} = P_{w_i}^\perp$ then $Q_{w_i}^{(2)} = P_{w_i}$.
\item\label{Item=EnumiII} If $Q_{w_i}^{(2)} = P_{w_i}$, then it must be true that $Q_{w_{i-1}}^{(1)} = P_{w_{i-1}}^\perp$ or $(w_{i-1}, w_i) \in E\Gamma$.
\end{enumerate}
These observations yield that, without loss of generality, \eqref{Eqn=QSeries} can assumed to be of a specific form.
\begin{itemize}
\item We claim that \eqref{Eqn=QSeries} can be assumed to be of the form:
\begin{equation}\label{Eqn=QSemiSeries}
Q_{w_1}^{(2)} a_{w_1, j_1} Q_{w_1}^{(1)}\ldots Q_{w_s}^{(2)} a_{w_s, j_s} Q_{w_s}^{(1)} P_{w_{s+1}}^{\perp} a_{w_{s+1}, j_{s+1}} P_{w_{s+1}} \ldots P_{w_{k}}^{\perp} a_{w_{k}, j_{k}} P_{w_{k}},
\end{equation}
where for every $1 \leq i \leq s$ we do not have that $Q_{w_i}^{(2)} = P_{w_{i}}^{\perp}$ and $Q_{w_i}^{(1)} = P_{w_{i}}$. In order to prove this claim first note that if $Q_{w_i}^{(2)} = P_{w_{i}}^{\perp}$ and $Q_{w_i}^{(1)} = P_{w_{i}}$ then it follows from \eqref{Item=EnumiII} that either $Q_{w_{i+1}}^{(2)} = P_{w_{i+1}}^{\perp}$ and $Q_{w_{i+1}}^{(1)} = P_{w_{i+1}}$ or  $(w_{i}, w_{i+1}) \in E\Gamma$. It then suffices to show that in the latter case the operators $P_{w_{i}}^\perp a_{w_{i}, j_{i}} P_{w_{i}} $ and $Q_{w_{i+1}}^{(2)} a_{w_{i+1}, j_{i+1}} Q_{w_{i_1}}^{(1)}$ commute. So firstly observe that $P_{w_i} P_{w_{i+1}}$ is a projection  and hence   $P_{w_i}$ and  $P_{w_{i+1}}$ commute. By taking complements, any of the projections $P_{w_i}, P_{w_i}^\perp, P_{w_{i+1}}$ and $P_{w_{i+1}}^\perp$ commute. It follows from Lemma \ref{Prop=commutant} that $P_{w_{i}}^\perp a_{w_{i}, j_{i}} P_{w_{i}} $ and $Q_{w_{i+1}}^{(2)} a_{w_{i+1}, j_{i+1}} Q_{w_{i+1}}^{(1)}$ commute. This concludes \eqref{Eqn=QSemiSeries}.
\item An analogous argument as in the previous bullet point yields that, without loss of generality, we may assume that \eqref{Eqn=QSeries} has the form,
\begin{equation}\label{Eqn=QSemiSeriesII}
Q_{w_1}^{(2)} a_{w_1, j_1} Q_{w_1}^{(1)}\ldots Q_{w_r}^{(2)} a_{w_r, j_r} Q_{w_r}^{(1)}
P_{w_{r+1}}^{(2)} a_{w_{r+1}, j_{r+1}} P_{w_{r+1}}^{(1)} \ldots P_{w_{s}}^{(2)} a_{w_{s}, j_{s}} P_{w_{s}}^{(1)}
P_{w_{s+1}}^{\perp} a_{w_{s+1}, j_{s+1}} P_{w_{s+1}} \ldots P_{w_{k}}^{\perp} a_{w_{k}, j_{k}} P_{w_{k}},
\end{equation}
and that for every $1 \leq i \leq r$ we do not have that $Q_{w_i}^{(1)} = P_{w_{i}}$.
 \item If $Q_{w_i}^{(1)} = P_{w_{i}}^\perp$ then this implies that $Q_{w_i}^{(2)} = P_{w_{i}}$ by \eqref{Item=EnumiI}. So \eqref{Eqn=QSemiSeriesII} shows that the expression  \eqref{Eqn=QSeries} can be written as,
\begin{equation}\label{Eqn=PSeries}
P_{w_1} a_{w_1, j_1} P_{w_1}^\perp \ldots P_{w_r} a_{w_r, j_r} P_{w_r}^\perp P_{w_{r+1} } a_{w_{r+1}, j_{r+1} } P_{w_{r+1}} \ldots P_{w_s} a_{w_s, j_s} P_{w_s} P_{w_{s+1}}^\perp a_{w_{s+1}, j_{s+1} } P_{w_{s+1}} \ldots P_{w_k}^\perp a_{w_k, j_k}P_k,
\end{equation}
for some $0 \leq r \leq s \leq k$ (the cases $r=0$ and $s=k$ should be understood naturally).

\item Moreover, suppose that $s>r+1$. Then it follows from \eqref{Item=EnumiI} that $(w_{r+1}, w_{r+2}) \in E\Gamma$. As in the first bullet point  this implies that $P_{w_{r+1} } a_{w_{r+1}, j_{r+1} } P_{w_{r+1}}$ and $P_{w_{r+2} } a_{w_{r+2}, j_{r+2} } P_{w_{r+2}}$ commute. Hence it follows from \eqref{Item=EnumiII} that $(w_{r+1}, w_{r+3}) \in E\Gamma$ (provided that $s>r_2$) and inductively we find that $(w_{r+1}, w_{i}) \in E\Gamma$ for every $r+1 \leq i \leq s$. The same argument yields that actually $(w_i, w_j) \in E\Gamma$ for every $r+1 \leq i, j \leq s$. We conclude that $w_{r+1}, \ldots, w_s$ are in clique of $\Gamma$.
\end{itemize}

\noindent {\bf Reduction of the vector part.}
Now suppose that a vector $\xi_{v_1, i_1} \otimes \ldots \otimes \xi_{v_l, i_l}$ is not in the kernel of \eqref{Eqn=PSeries}. Then this implies that we may assume (using the commutation relations given by $E\Gamma$ to permute terms in \eqref{Eqn=PSeries})   that $v_1= w_k, \ldots, v_{k-r} = w_{r+1}$, that $w_{s} \ldots w_{r+1}$ is contained in a clique and furthermore that $v_{k-r+1} \not = w_{r}$. And in that case,
\begin{equation}\label{Eqn=OperatorActsOnSpace}
\begin{split}
&\left( P_{w_1} a_{w_1, j_1} P_{w_1}^\perp \ldots P_{w_r} a_{w_r, j_r} P_{w_r}^\perp P_{w_{r+1}} a_{w_{r+1}, j_{r+1} } P_{w_{r+1}} \ldots P_{w_s} a_{w_s, j_s } P_{w_s} P_{w_{s+1}}^\perp a_{w_{s+1}, j_{s+1}  }  P_{w_{s+1}}\ldots P_{w_k}^\perp a_{w_k, j_k} P_k \right) \\
& \left( \xi_{v_1, i_1} \otimes \ldots \otimes \xi_{v_l, i_l} \right) =\\
 &  \widehat{a}_{w_1, j_1} \otimes \ldots \otimes \widehat{a}_{w_r, j_r} \otimes  P_{w_s} a_{w_s, j_s} \xi_{v_{k-s+1}, i_{k-s+1} }  \otimes \ldots \otimes P_{w_{r+1}} a_{w_{r+1}, j_{r+1} } \xi_{v_{k-r}, i_{k-r} } \\
& \otimes \xi_{v_{k-r+1}, i_{k-r+1}} \otimes \ldots \otimes \xi_{v_l, i_l} \times \langle a_{w_k, j_k} \xi_{v_1, i_1}, \Omega \rangle \ldots \langle a_{w_{s+1}, j_{s+1} } \xi_{v_{k-s}, i_{k-s}}, \Omega \rangle,
\end{split}
\end{equation}
where we explicitly mention that some of the indices in the triple dots of the right hand side of this expression either  increase or decrease by steps of 1.
 Looking at the length of tensor products shows that \eqref{Eqn=OperatorActsOnSpace} is in the kernel of $P_m$ unless $m+k-l = s+r$.

 \vspace{0.3cm}

\noindent {\bf Remainder of the proof.}  Now we conclude from \eqref{Eqn=LoadsOfPs} and \eqref{Eqn=OperatorActsOnSpace} that,
\begin{equation}\label{Eqn=LargeSummation}
\begin{split}
& \Vert P_m \mathcal{F}(a) \xi \Vert_2^2  \\
\leq & \sum_{\begin{array}{c}   m+k-l = s+r \\ 0 \leq s,r \leq k \end{array}} \!\!\!\!\!\!\!\! \sum_{\begin{array}{c}(u_1, \ldots, u_{s-r}) \in {\rm Cliq}_\Gamma(s-r) \end{array} } \!\!\!\!\!\!\!\!\!\!\!\!\!\!\!\!\!\!\!\!\!\!\!\!\!\!\!\!\!\!\!\! \sum_{\begin{array}{c} \boldw, \boldv \in \cWmin,\\ l(\boldw) = k, l(\boldv) = l,\\ v_1 = w_k \ldots v_{k-r} = w_{r+1}, \\ (v_{k-s+1},\ldots, v_{k-r}  ) = ( w_{s}, \ldots, w_{r+1}) = (u_1, \ldots, u_{s-r}) \end{array}} \!\!\!\!\!\!\!\!\!\!\!\!\!\!\!\!\!\!\!\!\!\!\!\!\!\!\!\!\!\!\!\!\sum_{\begin{array}{c}j_1, \ldots, j_r,\\ j_{s+1}, \ldots, j_k\end{array}} \!\!\!\!\!\sum_{\begin{array}{c} i_1, \ldots, i_{k-s},\\ i_{k-r+1}, \ldots, i_l \end{array}} \\
 & \qquad \Vert \widehat{a}_{w_1, j_1} \Vert_2^2 \ldots \Vert \widehat{a}_{w_r, j_r} \Vert_2^2 \times  \delta_{j_k, i_1} \Vert  \widehat{a}_{w_k, j_k} \Vert_2^4  \ldots \delta_{j_{s+1}, i_{k-s}} \Vert \widehat{a}_{w_{s+1}, j_{s+1}} \Vert_2^4 \\
 &\qquad \left|\left|  \sum_{j_{r+1}, \ldots, j_s} a_{w_{r+1}, j_{r+1} }\ldots a_{w_s, j_s}  \sum_{i_{k-r}, \ldots, i_{k-s+1}}  \xi_{v_{k-r}, i_{k-r}} \otimes \ldots \otimes \xi_{v_{k-s+1}, i_{k-s+1}}  \right| \right|_2^2.
\end{split}
\end{equation}
We have, since $\widehat{\bG}_{\Gamma_0}$ has (RD) by assumption for every clique $\Gamma_0$ in $\Gamma$,
\begin{equation}\label{Eqn=SubSummation}
\begin{split}
&  \left|\left|  \sum_{j_{r+1}, \ldots, j_s} a_{w_{r+1}, j_{r+1} }\ldots a_{w_s, j_s}  \sum_{i_{k-r}, \ldots, i_{k-s+1}}  \xi_{v_{k-r}, i_{k-r}} \otimes \ldots \otimes \xi_{v_{k-s+1}, i_{k-s+1}}  \right| \right|_2^2 \\
\leq & P(s-r)  \left|\left|  \sum_{j_{r+1}, \ldots, j_s} a_{w_{r+1}, j_{r+1} }\ldots a_{w_s, j_s} \right|\right|_2^2 \left| \left|  \sum_{i_{k-r}, \ldots, i_{k-s+1}}  \xi_{v_{k-r}, i_{k-r}} \otimes \ldots \otimes \xi_{v_{k-s+1}, i_{k-s+1}}  \right| \right|_2^2,
\end{split}
\end{equation}
for some polynomial $P$. Let $Q$ be a polynomial such that $P(s-r) \leq Q(k)$ for any choice of $s, r \in \mathbb{N}$ with $0 \leq s,r \leq k$. We may choose $Q$ independent of the clique $\Gamma_0$ that defined $P$. Combining \eqref{Eqn=LargeSummation} and \eqref{Eqn=SubSummation} we see that,
\[
\begin{split}
& \Vert P_m \mathcal{F}(a) \xi \Vert_2^2 \\
\leq & \sum_{\begin{array}{c}   m+k-l = s+r \\ 0 \leq s,r \leq k \end{array}} \!\!\!\!\!\!\!\! \sum_{\begin{array}{c}(u_1, \ldots, u_{s-r}) \in {\rm Cliq}_\Gamma(s-r) \end{array} } \!\!\!\!\!\!\!\!\!\!\!\!\!\!\!\!\!\!\!\!\!\!\!\!\!\!\!\!\!\!\!\! \sum_{\begin{array}{c} \boldw, \boldv \in \cWmin,\\ l(\boldw) = k, l(\boldv) = l,\\ v_1 = w_k \ldots v_{k-r} = w_{r+1}, \\ (v_{k-s+1},\ldots, v_{k-r}  ) = ( w_{s}, \ldots, w_{r+1}) = (u_1, \ldots, u_{s-r}) \end{array}} \!\!\!\!\!\!\!\!\!\!\!\!\!\!\!\!\!\!\!\!\!\!\!\!\!\!\!\!\!\!\!\!\sum_{\begin{array}{c}j_1, \ldots, j_k \end{array}} \!\!\!\!\!\sum_{\begin{array}{c} i_1, \ldots, i_l \end{array}} \\
&    Q(k) \Vert \widehat{a}_{w_1, j_1} \Vert_2^2 \ldots \Vert \widehat{a}_{w_k, j_k} \Vert_2^2  \Vert  \xi_{v_{1}, i_{1}} \Vert_2^2  \ldots  \Vert \xi_{v_{l}, i_{l}} \Vert_2^2 \\
  \\
 \leq & M (k+1)^2 Q(k) \Vert a \Vert_2^2 \Vert \xi \Vert_2^2.
\end{split}
\]
where $M$ is the number of cliques in $\Gamma$, which is finite since $\Gamma$ is finite.
\end{proof}

\begin{lem}\label{Lem=OneMinusCounit}
Let $\bG$ be a compact quantum group and let $L$ be a central length associated with $\widehat{\bG}$. Then there exists a central length $L' \geq L$ associated with $\widehat{\bG}$ such that $L' p_\alpha \geq 1$ for every $\alpha \in \Irr(\bG)$ nontrivial.
\end{lem}
\begin{proof}
Since $L$ is a central length we may write $L = \oplus_{\alpha \in \Irr(\bG)} f(\alpha) p_\alpha$. We define the central length $L' = \oplus_{\alpha \in \Irr(\bG)} f'(\alpha) p_\alpha$, where $f'(\alpha) = f(\alpha)+1$ if $\alpha$ is nontrivial and $f'(\alpha) = f(\alpha)$ in case $\alpha$ is trivial. As in the proof of Lemma \ref{Lem=GraphProductLength} the condition $\widehat{\Delta}(L') \leq L' \otimes 1 + 1 \otimes L'$ is equivalent to checking that $\sum_{\gamma \in \Irr(G), \gamma \subseteq \alpha \otimes \beta} f'(\gamma) p_\gamma^{\alpha \otimes \beta} \leq (f'(\alpha) + f'(\beta)) p_\alpha \otimes p_\beta$.  However, this condition easily follows from the fact that if both $\alpha$ and $\beta$ are trivial then $\alpha \otimes \beta$ is trivial and  so $\gamma$ is trivial whenever  $\gamma \subseteq \alpha \otimes \beta$. The condition $\widehat{\kappa}(L') = L'$ follows as in the proof of Lemma \ref{Lem=GraphProductLength}, see also Lemma \ref{Lem=KappComputation}. And finally by definition of the counit we have $\widehat{\epsilon}(L') = f'(\alpha_0) = f(\alpha_0) = 0$ with $\alpha_0 \in \Irr(\bG)$ trivial.
\end{proof}

\begin{thm}\label{Thm=RDPreservence}
Let $\Gamma$ be a finite graph and let, for every $v \in V \Gamma$, $\bG_v$ be a compact quantum group such that $(\widehat{\bG}_v, L_v)$ has (RD). Assume that for every clique $\Gamma_0$ the graph product $\widehat{\bG}_{\Gamma_0}$ has (RD).  Then the graph product $(\widehat{\bG} := \widehat{\bG}_\Gamma, L)$ has (RD) for some central length $L$. If $L_v p_\alpha \geq 1$ for every $v \in V\Gamma$ and $\alpha \in \Irr(\bG_v)$ nontrivial then $L$ can be taken as in Lemma \ref{Lem=GraphProductLength}.
\end{thm}
\begin{proof}
For $v \in V \Gamma$ denote by $L_v$ a central length for $\widehat{\bG}_v$ and let $L$ be the central length defined in Lemma \ref{Lem=GraphProductLength}. Assume by Lemma \ref{Lem=OneMinusCounit} and \cite[Remark 3.6]{VerRd} that,
\begin{equation}\label{Eqn=BiggerThanOneCondition}
L_v p_\alpha \geq 1, \qquad \forall v \in V\Gamma, \alpha \in \Irr(\bG_v).
\end{equation}
This implies that $L p_\alpha \geq l(\alpha)$ where $l(\alpha)$ the length of the reduced expression $\alpha = \alpha_1 \otimes \ldots \otimes \alpha_{l(\alpha_0)}$ with $\alpha \in \Irr(\bG)$. By  Proposition \ref{Prop=RDIntermediate} there exist a polynomial $P$ such that for every $k,l,m \in \mathbb{N}$ such that $\vert k-l \vert \leq m \leq k+l$ and $a \in  \widehat{\bA}_{(k)}$ we have $\Vert P_m \mathcal{F}(a) P_l \Vert \leq P(k) \Vert a \Vert_{2}$. Now, let $a \in q_k \widehat{\bA}$ and write $a = \sum_{j=0}^k a_{(j)}$ with $a_{(j)} \in \widehat{\bA}_{(j)}$, which is possible by the first paragraph. Take a vector $v \in q_l \cH$ and write $v = \sum_{i=0}^l v_{(i)}$ with $v_{(i)} = P_i v$. Since $\sum_{r=0}^m P_r \geq q_m$ and the projections $P_r$ are orthogonal, it follows that $\Vert q_m \mathcal{F}(a_{(j)}) q_l v \Vert_2^2 \leq \sum_{r=0}^m \Vert P_r \mathcal{F}(a_{(j)}) q_l v \Vert_2^2$. Next, we have an elementary equality that follows by considering word lengths and an inequality which follows from Cauchy-Schwarz and the triangle inequality,
\[
\begin{split}
\sum_{r=0}^m \Vert P_r \mathcal{F}(a_{(j)}) q_l v \Vert_2^2 = & \sum_{r=0}^m  \Vert \sum_{i=\vert j-r\vert}^{j+r} P_r \mathcal{F}(a_{(j)}) v_{(i)} \Vert_2^2 \\
\leq & (2j+1) \sum_{r=0}^m  \sum_{i = \vert j - r\vert}^{j+r} \Vert P_r \mathcal{F}(a_{(j)}) v_{(i)} \Vert_2^2.
\end{split}
\]
Now, for $\vert j-i \vert \leq r \leq j+i$ we have,
\[
\Vert P_r \mathcal{F}(  a_{(j)}  ) v_{(i)} \Vert_2^2 \leq P(j)^2 \Vert a_{(j)} \Vert_2^2 \Vert v_{(i)} \Vert_2^2.
\]
For other values of $r$ we have $\Vert P_r \mathcal{F}(  a_{(j)}  ) v_{(i)} \Vert_2^2 = 0$. Since, as we observed, $\vert i-j \vert \leq r \leq j+i$ for any given value of $i$, this shows that we can estimate,
\[
\begin{split}
\sum_{r=0}^m \Vert P_r \mathcal{F}(a_{(j)}) q_l v \Vert_2^2 \leq & P(j)^2 (2j+1)   \sum_{i = 0}^{j+m} \sum_{r=0}^m \Vert a_{(j)} \Vert_2^2 \Vert v_{(i)} \Vert_2^2 \\
 \leq & P(j)^2 (2j+1)^2  \sum_{i = 0}^{j+m} \Vert a_{(j)} \Vert_2^2 \Vert v_{(i)} \Vert_2^2 \\
\leq &    P(j)^2 (2j+1)^2   \Vert a_{(j)} \Vert_2^2 \Vert v \Vert_2^2.
\end{split}
\]
Now, using the triangle inequality and the Cauchy-Schwarz inequality we have
\[
\begin{split}
& \Vert q_m \mathcal{F}(a) q_l v \Vert_2^2  \leq (\sum_{j=0}^k \Vert q_m \mathcal{F}(a_{(j)}) q_l v \Vert_2)^2
\leq  (k+1) \sum_{j=0}^k \Vert q_m \mathcal{F}(a_{(j)}) q_l v \Vert_2^2 \\
\leq &   (k+1) \sum_{j=0}^k P(j)^2 (2j +1)^2  \Vert a_{(j)} \Vert_{2}^2     \Vert v \Vert_2^2
\leq    (k+1) (2k+1)^2 P'(k)    \Vert a \Vert_{2}^2     \Vert v \Vert_2^2
=  P''(k)   \Vert a \Vert_{2}^2     \Vert v \Vert_2^2,
\end{split}
\]
for some polynomials $P, P', P''$ that satisfy the property that, for every $0 \leq j \leq k$, $P(j)^2 \leq P'(k)$ and $P''(k) = (k+1)(2k+1)^2 P'(k)$.
\end{proof}

\begin{cor}
Let $\Gamma$ be a finite graph. For $v \in V\Gamma$ let $\bG_v$ be a compact quantum group and assume that $\widehat{\bG}_v$ has either polynomial growth or is a classical compact group with (RD). Then the discrete dual of the graph product has (RD).
\end{cor}
\begin{proof}
This is a consequence of Theorem \ref{Thm=RDPreservence} and Lemma \ref{Lem=RDTensor}.
\end{proof}

\begin{cor}
Let $\Gamma$ be finite and without edges. Let $\bG = \star_{v \in V\Gamma} \bG_v$. If for every $v \in V\Gamma$, $\widehat{\bG}_v$ has (RD), then $\widehat{\bG}$ has (RD). I.e. (RD) is preserved by finite free products.
\end{cor}
\begin{proof}
This is a consequence of Theorem \ref{Thm=RDPreservence}.
\end{proof}

\begin{thebibliography}{9}

\bibitem[ALS07]{AccardiEtAl}
  L. Accardi, R. Lenczewski, R. Salapata,
  \emph{Decompositions of the free product of graphs},
  Infin. Dimens. Anal. Quantum Probab. Relat. Top. {\bf 10} (2007), no. 3, 303--334.

\bibitem[AnDr13]{AntDre}
  Y. Antolin, D. Dreesen,
  \emph{The Haagerup property is stable under graph products},
  arXiv:1305.6748.


\bibitem[AnMi11]{AntMin}
   Y. Antolin, A. Minasyan,
   \emph{Tits alternatives for graph products},
    Journal f\"ur die Reine und Angewandte Mathematik (to appear), arXiv:1111.2448.

 \bibitem[BaVe09]{BanicaVer}
   T. Banica, R. Vergnioux,
  \emph{Growth estimates for discrete quantum groups},
   Infin. Dimens. Anal. Quantum Probab. Relat. Top. {\bf 12} (2009), 321--340.


\bibitem[Bo93]{Boca}
  F. Boca,
  \emph{Completely positive maps on amalgamated product C$^\ast$-algebras},
   Math. Scand. {\bf 72} (1993), no. 2, 212--222.

\bibitem[Boc93]{BocaHAP}
  F. Boca,
  \emph{On the method of constructing irreducible finite index subfactors of Popa},
   Pacific J. Math. {\bf 161} (1993), no. 2, 201--231.

\bibitem[BLS96]{BozLeiSpe}
  M. Bozejko, M.  Leinert,  R.  Speicher,
  \emph{Convolution and limit theorems for conditionally free random variables},
  Pacific J. Math. {\bf 175} (1996), no. 2, 357--388.

\bibitem[BoSp94]{BozSpe}
 M. Bozejko, R. Speicher,
 \emph{Completely positive maps on Coxeter groups, deformed commutation relations, and operator spaces},
  Math. Ann. {\bf 300} (1994), no. 1, 97--120.


\bibitem[Cas15]{CasConnes}
  M. Caspers,
  \emph{Connes embeddability of graph products},
  Infinite Dimensional Analysis, Quantum Probability and Related Topics (to appear), arXiv: 1506.01873.

\bibitem[CLR13]{CasLeeRic}
  M. Caspers, H.H. Lee, E. Ricard,
   \emph{Operator biflatness of the $L^1$-algebras of compact quantum groups},
   J. Reine Angew. Math. 700 (2015), 235--244.

\bibitem[CaSk13]{CasSka}
  M. Caspers, A. Skalski,
  \emph{The Haagerup property for arbitrary von Neumann algebras},
   International Mathematical Research Notes (to appear), arXiv:1312.1491.

\bibitem[CaSk14]{CasSka2}
  M. Caspers, A. Skalski,
  \emph{The Haagerup approximation property for von Neumann algebras via quantum Markov semigroups and Dirichlet forms},
  Comm. Math. Phys. {\bf 336} (2015), no. 3, 1637--1664.

\bibitem[COST14]{COST}
  M. Caspers, R. Okayasu, A. Skalski, R. Tomatsu,
  \emph{Generalisations of the Haagerup approximation property to arbitrary von Neumann algebras},
  C. R. Acad. Sci. Paris, Ser. I {\bf 352 } (2014) 507--510.


\bibitem[Cho83]{Cho}
  M. Choda,
   \emph{Group factors of the Haagerup type},
   Proc. Japan Acad. Ser. A Math. Sci. {\bf 59} (1983), 174--177.


\bibitem[Chi12]{Chi}
  I. M. Chiswell,
   \emph{Ordering graph products of groups},
    Internat. J. Algebra Comput., {\bf 22} (2012), no. 4, 1250037, 14 pp.


\bibitem[CHR12]{CioHolRee2}
  L. Ciobanu, D. Holt and S. Rees,
    \emph{Sofic groups: graph products and graphs of groups},
    Pacific Journal of Mathematics (to appear) arXiv:1212.2739.


\bibitem[CHR13]{CioHolRee}
  L. Ciobanu, F. Holt, S. Rees,
  \emph{Rapid decay is preserved by graph products},
  J. Topol. Anal. {\bf 5} (2013), no. 2, 225--237.



\bibitem[Da10]{DawsPAMS}
  M. Daws,
  \emph{Operator biprojectivity of compact quantum groups},
   Proc. Amer. Math. Soc. {\bf 138} (2010), no. 4, 1349--359.




\bibitem[DFSW13]{DFSW}
  M. Daws, P. Fima, A. Skalski, S. White,
 \emph{The Haagerup property for locally compact quantum groups},
  Journal f\"ur die Reine und Angewandte Mathematik (to appear), arXiv:1303.3261.


\bibitem[DijKo93]{DijkKoo}
   M. Dijkhuizen, T. Koornwinder,
    \emph{CQG algebras: a direct algebraic approach to compact quantum
groups},
   Lett. Math. Phys. {\bf 32} (1994), no. 4, 315--330.

 \bibitem[Dy04]{Dyk04}
K. Dykema,
\emph{Exactness of reduced amalgamated free product C*-algebras},
 Forum Math. {\bf 16} (2004), 161--180.


\bibitem[Fi10]{FimaPropT}
 P. Fima,
  \emph{Kazhdan's property T for discrete quantum groups},
   Internat. J. Math. {\bf 21} (2010), no. 1, 47–65.


\bibitem[FiFr13]{FimFre}
 P. Fima, A. Freslon,
  \emph{Graphs of quantum groups and K-amenability},
  Advances in Mathematics (to appear),  arXiv:1307.5609.

\bibitem[Fr14]{FreStability}
  A. Freslon,
  \emph{Permanence of approximation properties for discrete quantum groups}, Annales de l'Institut Fourier (to appear),
   arXiv:1407.3143.


\bibitem[Gr90]{Green}
  E.R. Green,
  \emph{Graph products},
  PhD thesis, University of Leeds, 1990,
  http://ethesis.whiterose.ac.uk/236.


\bibitem[HeMei95]{HerMei}
 S. Hermiller, J. Meier,
  \emph{Algorithms and geometry for graph products of groups},
  J. Algebra {\bf 171} (1995), no. 1, 230--257.



\bibitem[HsWi99]{HsuWis}
  T. Hsu, T. Wise,
   \emph{On linear and residual properties of graph products},
    Michigan Math. J. {\bf 46} (1999) 251--259.


 \bibitem[Jo02]{Jol}
  P. Jolissaint,
   \emph{Haagerup approximation property for finite von Neumann algebras},
   J. Operator Theory {\bf 48} (2002), 549--571.

\bibitem[Jo90]{JolRD}
  P. Jolissaint,
  \emph{Rapidly decreasing functions in reduced C$^\ast$-algebras of groups},
   Trans. Amer. Math. Soc. {\bf 317} (1990), 167--196.

\bibitem[Ky11]{KyedJFA}
  D. Kyed,
 \emph{ A cohomological description of property (T) for quantum groups},
  J. Funct. Anal. {\bf 261} (2011), 1469--1493.


\bibitem[IPP08]{IPP08}
A. Ioana, J. Peterson, S. Popa,
\emph{Amalgamated free products of weakly rigid factors and calculation of their symmetry groups,} Acta Math. {\bf 200} (2008), 85--153.

\bibitem[Ku01]{KusUni}
  J. Kustermans,
  \emph{Locally compact quantum groups in the universal setting},
  Internat. J. Math. {\bf 12} (2001), 289--338.

\bibitem[KuVa03]{KusVaeVNA}
  J. Kustermans, S. Vaes,
  \emph{Locally compact quantum groups in the von Neumann algebraic setting},
   Math. Scand. {\bf 92} (2003), no. 1, 68--92.

\bibitem[OkTo13]{OkaTom}
  R. Okayasu, R. Tomatsu,
  \emph{Haagerup approximation property for arbitrary von Neumann algebras}, Publications of the RIMS (to appear),
   arXiv:1312.1033.

\bibitem[Spe93]{SpeicherLetter}
  R. Speicher,
  \emph{Generalized statistics of macroscopic fields},
  Lett. Math. Phys. {\bf 27} (1993), no. 2, 97--104.


\bibitem[Ti08]{Tim}
 T. Timmermann,
  \emph{An invitation to quantum groups and duality. From Hopf algebras to multiplicative unitaries and beyond.}
   EMS Textbooks in Mathematics. European Mathematical Society (EMS), Zürich, 2008.

\bibitem[Va06]{Va06}
S. Vaes
\emph{Rigidity results for Bernoulli actions and their von Neumann algebras}, Ast\'erisque {\bf 311} (2007), 237--294.


\bibitem[Ve07]{VerRd}
  R. Vergnioux,
  \emph{The Property of Rapid Decay for Discrete Quantum Groups},
  J. Operator Theory {\bf 57} (2007) 303--324




\bibitem[Vo85]{Voi}
  D.V. Voiculescu,
  \emph{Symmetries of some reduced free product C$^\ast$-algebras},
  Letcure Notes in Math. 1132 (1985), 556--588.


\bibitem[Vo14]{Voic}
  D.V. Voiculescu,
  \emph{Free probability for pairs of faces I},
   Comm. Math. Phys. {\bf 332} (2014), no. 3, 955--980.

\bibitem[Wa95]{WanCMP}
  S. Wang,
  \emph{Free products of compact quantum groups},
   Comm. Math. Phys. {\bf 167} (1995), no. 3, 671--692.


\bibitem[Wo87]{WorQG}
  S.L. Woronowicz,
  \emph{Compact matrix pseudogroups},
  Comm. Math. Phys. {\bf 111} (1987), 613--665.


\bibitem[Wo91]{WorAffiliated}
  S.L. Woronowicz,
  \emph{Unbounded elements affiliated with C$^\ast$-algebras and noncompact quantum groups},
    Comm. Math. Phys. {\bf 136} (1991), no. 2, 399--432.


\end{thebibliography}
\end{document}